\documentclass[reqno]{amsart}
\usepackage{amssymb}
\usepackage{hyperref}
\usepackage{mathrsfs}

\usepackage[margin=1in]{geometry}  
\usepackage{graphicx}              
\usepackage{amsmath}               
\usepackage{amsfonts}              
\usepackage{amsthm}                
\usepackage{amssymb}

\usepackage{color}
\numberwithin{equation}{section}
\newtheorem{theorem}{Theorem}[section]
\newtheorem{lemma}[theorem]{Lemma} 
\newtheorem{proposition}[theorem]{Proposition}
\newtheorem{definition}[theorem]{Definition}
\newtheorem{remark}[theorem]{Remark}
\newtheorem{corollary}[theorem]{Corollary}

\allowdisplaybreaks

\numberwithin{equation}{section}

\usepackage{graphicx}              
\graphicspath{{./img/}}	

\begin{document}

\title[ Nonlocal integrodifferential operators with  sign-changing weight functions]{ A class of elliptic equations involving a general nonlocal integrodifferential operators with  sign-changing weight functions}

\author{Lauren Maria Mezzomo Bonaldo, Olímpio Hiroshi Miyagaki and Elard Juárez Hurtado}
\thanks{Department of Mathematics Universidade Federal de S\~ao Carlos.
Carlos,  S\~ao Carlos SP, Brazil.
E-mail: laurenmbonaldo@hotmail.com. Supported
by CAPES}\, \thanks{Department of Mathematics Universidade Federal de S\~ao Carlos.
Carlos,  S\~ao Carlos SP, Brazil.
E-mail: ohmiyagaki@gmail.com Supported
by CAPES} \,
\thanks{Department of Mathematics, Universidade de Brasília, Brasília,  DF, Brazil.
E-mail: elardjh2@gmail.com Supported
by CAPES}

\date{}
\maketitle
\begin{abstract}
In this paper we investigate the existence of nontrivial weak solutions to a class of elliptic equations \eqref{p1} involving a general nonlocal integrodifferential operator $\mathscr{L}_{\mathcal{A}K},$  two real parameters, and two weight functions, which can be sign-changing.
Considering different situations concerning the growth of the nonlinearities involved in the problem \eqref{p1}, we prove the existence of 
two nontrivial distinct  solutions and the existence of a continuous family of eigenvalues. The proofs of the main results are based on ground state solutions using the Nehari method, Ekeland’s variational principle, and the direct method of the calculus of variations. The difficulties arise from the fact that the operator  $ \mathscr{L}_{\mathcal{A}K}$   is nonhomogeneous and the nonlinear term is undefined.

\end{abstract}

 
  \section{Introduction and statements of main results}
   \hfill \break
In this paper, we deal with results concerning the existence of  weak solutions for a  class of elliptic equations involving a general nonlocal integrodifferential operator with two real parameters, two weight functions which can be sign-changing  and different subcritical nonlinearities. More precisely,  in a smooth bounded domain $\Omega$ of $\mathbb{R}^{N}$ ($N\geqslant 2$), we consider the following problem
  \begin{equation}\tag{$\mathscr{P}$}
\label{p1}\left\{\begin{array}{rc}
\begin{split}
 \mathscr{L}_{\mathcal{A}K}u  &= \lambda \mathfrak{a}(x)|u|^{\mathfrak{m}_{1}(x)-2}u +\beta \mathfrak{b}(x)|u|^{\mathfrak{m}_{2}(x)-2}u  &\mbox{ in }& \Omega,\\
u  &=  0 &\mbox{ in }& \mathbb{R}^{N}\setminus\Omega,
\end{split}
\end{array}\right.
\end{equation}
where $\lambda$ and $\beta $ are real parameters, the  weight functions   $\mathfrak{a}, \mathfrak{b}:\overline{\Omega} \to \mathbb{R}$  can be sign-changing  in $\Omega$, $\mathfrak{m}_1\,$ and $\mathfrak{m}_2\in C^{+}(\overline{\Omega}),$
and to define the  general nonlocal integrodifferential operator $\mathscr{L}_{\mathcal{A}K}$ we will consider the variable exponent  $p(x):=p(x,x)$ for all $x \in \mathbb{R}^{N}$ with $p \in C(\mathbb{R}^{N} \times\mathbb{R}^{N})$  satisfying:
  \begin{equation}\label{a23}\tag{$\mathit{p_1}$}
  \begin{split}
   & p \mbox{ is symmetric, that is, } p(x,y)=p(y,x),  \\&  1< p^{-} :=\displaystyle{\inf _{(x,y) \in \mathbb{R}^{N}\times \mathbb{R}^{N}}}\,p(x,y) \leqslant \displaystyle{\sup _{(x,y)\in \mathbb{R}^{N}\times \mathbb{R}^{N}}}\,p(x,y):= p^{+}<\frac{N}{s},\hspace{0.2cm} s \in(0,1),
 \end{split}
  \end{equation}
  and we consider  the fractional critical variable exponent related to $ p  \in C(\mathbb{R}^{N} \times\mathbb{R}^{N})$ defined by $\displaystyle{p^{\star}_{s}(x)= \frac{Np(x)}{N-sp(x)}}.$ 
 
The  general nonlocal integrodifferential operator $\mathscr{L}_{\mathcal{A}K}$ is defined on suitable fractional Sobolev spaces (see Section 2) by
 \begin{equation*}
  \mathscr{L}_{\mathcal{A}K}u(x)= P.V. \int_{\mathbb{R}^{N}}\mathcal{A}(u(x)-u(y))K(x,y)\,dy, \hspace{0.2cm} x \in \mathbb{R}^{N},
  \end{equation*}
  where $P.V.$ is the principal value.

  The map $\mathcal{A}:\mathbb{R}\to \mathbb{R}$ is a measurable function  satisfying the next assumptions:
\begin{itemize}
 \item[ $(a_1)$] $\mathcal{A}$ is continuous, odd, and the map $\mathscr{A}:\mathbb{R}\to \mathbb{R}$ given by
 $$\mathscr{A}(t):= \int^{|t|}_{0} \mathcal{A}(\tau) d\tau$$
is  strictly convex;
 \item[$(a_{2})$]  There exist positive constants $c_{\mathcal{A}}$ and $C_{\mathcal{A}}$, such that  for  all $t \in \mathbb{R}$ and  for all $(x,y)\in \mathbb{R}^{N}\times\mathbb{R}^{N}$
\begin{equation*}
\mathcal{A}(t)t\geqslant  c_{\mathcal{A}} \vert t \vert^{p(x,y)} \hspace{0.3cm} \mbox{ and } \hspace{0.3cm} \vert \mathcal{A}(t) \vert\leqslant C_{\mathcal{A}} \vert t \vert^{p(x,y)-1};
\end{equation*}
\item[$(a_{3})$]$\mathcal{A}(t)t\leqslant p^{+}\mathscr{A}(t)$ for all $t \in \mathbb{R}$.
\end{itemize}
The kernel $K: \mathbb{R}^{N}\times \mathbb{R}^{N} \to \mathbb{R}^{+}$ is a measurable function  satisfying the following property:
\begin{itemize}
 \item[ $(\mathcal{K})$] There exist constants $b_0$ and $b_1$, such that $0<b_0\leqslant b_1$,
 $$b_0\leqslant K(x,y)|x-y|^{N+sp(x,y)}\leqslant b_1 \mbox{ for all } (x,y)\in \mathbb{R}^{N}\times\mathbb{R}^{N} \mbox{ and } x\neq y.$$
 \end{itemize}
 It is worth to note that the assumptions $(a_1)$-$(a_3)$ and $(\mathcal{K})$ were similarly introduced  in \cite{moli,brezish,colau,emdi,warma,elard1,elard2,kim,migio,mazon,puccisara,vazquez}.
A mathematical generalization very special for$\mathcal{A}$ and $K$ satisfying $(a_1)$-$(a_3)$ and $(\mathcal{K})$ is $\mathcal{A}(t)=|t|^{p(x,y)-2}t$ and  $K(x,y)= |x-y|^{-(N+sp(x,y))}$. Hence the operator $\mathscr{L}_{\mathcal{A}K}$ reduces to the  fractional $p(\cdot)$-Laplacian operator $(-\Delta)^{s}_{p(\cdot)}$, which is defined by
\begin{equation*}
  (-\Delta)^{s}_{p(\cdot)}u(x)= P.V. \int_{\mathbb{R}^{N}}\frac{|u(x)-u(y)|^{p(x,y)-2}(u(x)-u(y))}{|x-y|^{N+sp(x,y)}}\,dy \mbox{ for all } x \in \mathbb{R}^{N}.
  \end{equation*}

  During the last years, a remarkable interest has arisen in the study of partial differential equations involving fractional operators. The interest that has arisen in this type of problems is related to different questions both from the point of view of pure mathematics and from the point of view of applications, as for example in, continuum mechanics, optimization, finance phase transitions, image process, game theory, crystal dislocation, quasi-geostrophic flows, peridynamic theory, see for instance \cite{moll,bucu,caff,dinezza,vaal}.
In order to investigate the problems studied here, we need to work in spaces with variable exponents,
and the study these spaces is very interesting in itself (see Apendix). Indeed Fractional Sobolev spaces with variable exponents have been studied in depth during the last decade, for more references see instance\cite{anaour,bahrouni,ky,kaufmann}.
 
 
 Recently in  \cite{edcarlos} the authors obtained significant results for a problem involving the $p$-Laplace operator with concave-convex type nonlinearities and weight functions which can be sign-changing. This type of nonlinearity has received a lot of attention since being first investigated by Ambrosetti, Brezis, and Cerami in \cite{principio}. Since then the techniques for obtaining the existence of a solution for this type of nonlinearity vary according to the characterization of the problem. In \cite{edcarlos}, for example, they use the fibering method, introduced and developed by Pohozaev in \cite{poh1,poh2,poh3,poh4} and the constrained minimization in Nehari manifolds,  to determine the existence of two distinct solutions.
 
In \cite{azr,anaour,bahrouni,brasco,ky,kaufmann} the authors studied  the equation $\eqref{p1} $ when the operator $\mathscr{L}_{\mathcal{A}K} $  is the fractional $p(\cdot)$-Laplacian operator with $ \beta=1 $ or $ \beta=0 $, $\mathfrak{a}$ positive weight function, and associated nonlinearities with variable exponents. Under appropriate assumptions, the authors prove the results of the existence of solutions to
the problem  \eqref{p1} and also the existence of eigenvalues, via Mountain Pass Theorem, Ekeland’s variational principle, and the method of sub–super solutions.


Therefore, motivated by the latest research on the study of nonlocal problems, in this current paper we obtain new results on the existence of solutions for a wide class of elliptic equations involving general nonlocal integrodifferential operators with adequate nonlinearities involving weight functions with sign-changing, in this sense we extend, complement and improve some of the main results that appear in the following works \cite{azr,anaour,bahrouni,edcarlos,fann,fp,binge1} for a class of nonlocal problems. More precisely in this paper, we focus in the study of problem \eqref{p1} in three different situations. In the first situation, we are interested in showing the existence of two different nontrivial solutions for a convex-concave problem with sign-changing weight functions in the framework of Fractional Sobolev spaces with constant exponents via ground state solutions using the Nehari method. Already in the last two situations, we show the existence of a continuous family of values to the problem \eqref{p1} in the frame of the Fractional Sobolev spaces with variable exponents via the Ekeland’s variational principle and the direct method of the calculus of variations.

We would like to draw attention that in our knowledge, the operator $\mathcal{L}_{\mathcal{A}K}$ in the problem \eqref{p1} is a nonlocal general integrodiferencial operator where the map $ \mathcal{A} $ and kernel $ K $ under consideration are very general ($ K $ includes singular kernels). On the other hand  problem \eqref{p1} involves fractional $ p(\cdot)$-Laplacian operator, which have more complicated nonlinearities, for example, they are nonhomogeneous. In addition, by the characterization of our problem, taking as prototype space $W^{s,p}_{0}(\Omega)$, defined in \cite{ian1}, we define the following space 
\begin{equation*}\label{space}
  \mathscr{W} = W^{s,p(\cdot, \cdot)}_{0}:= \{u \in W^{s, p(\cdot, \cdot)}(\mathbb{R}^{N}): u=0 \mbox{ a.e. in }   \mathbb{R}^{N}\setminus \Omega \},
 \end{equation*} 
which is a separable and reflexive Banach space and that will play a crucial in this paper. It is important to mention that space $\mathscr{W}$ has some interesting properties similar to the fractional Sobolev spaces, specifically, we prove a equivalence of norms and that an important result of compact and continuous embedding remains valid, for more details, see Subsection  \ref{espacop}.

Our work also is motivated by the difficulty in applying the variational methods. For example,  in the Lemma \ref{ll1}  we will show  that the operator $\Phi'$  satisfies the property $(S_{+})$, which is a property of compactness of the operator and is usually essential to
obtain other properties such as Palais–Smale compactness condition or Cerami’s condition in a variational framework.
Besides, we highlight that changing the sign of the weights $ \mathfrak{a} $ and $ \mathfrak{b} $ creates some difficulties, among them, when analysing the Euler Lagrange functional associated with the problem \eqref{p1}, we can apply variational methods directly and the analysis for the existence of a problem solution for each type of nonlinearity becomes very delicate. For example, in the case of concave-convex nonlinearity, as detailed in the subsection \ref{nehari}, it was essential to divide the Nehari variety into two parts: $ \mathcal{N}_{\lambda, 1} = \mathcal{N}_{\lambda, 1}^{+} \cup \mathcal{N}_{\lambda, 1}^{-} $, use the fibration application to obtain a unique projection on each part $ \mathcal{N}^{\pm}_{\lambda, 1} $ and using the standard minimization procedure we get at least one solution for each set $ \mathcal{N}^{\pm}_{\lambda, 1} $.
 The same occurred for sublinear and superlinear nonlinearities, in each case it was essential to consider appropriate solution sets, as we will see in Subsections \ref{teorema3} and \ref{teorema4}, to obtain the expected results.



 Before we present our main results, we will give the following definition.

\begin{definition}
 We say that  $u \in \mathscr{W} $ is a weak solution of the   problem \eqref{p1} if and only if
 \begin{equation}\label{weak}
 \begin{split}
\int_{\mathbb{R}^{N}\times \mathbb{R}^{N}}\mathcal{A}(u(x)-u(y))(v(x)-v(y))K(x,y)\,dx\,dy  =&
   \int_{\Omega}\big[\lambda\mathfrak{a}(x)|u|^{\mathfrak{m}_{1}(x)-2}u\,v\,+ \beta \mathfrak{b}(x)|u|^{\mathfrak{m}_{2}(x)-2}u\,v\big]\,dx
  \end{split}
 \end{equation}
  for every $v \in \mathscr{W}$.
 
 When $\beta=0$, we say that $\lambda$ is an eigenvalue of the problem \eqref{p1}, if there exists $u \in \mathscr{W}\setminus\{0\}$ satisfying \eqref{weak}, that is, $u$ is the corresponding eigenfunction to $\lambda$.
 \end{definition}
We introduce the Euler Lagrange functional $ \mathcal{J}_{\lambda, \beta}: \mathscr{W} \to \mathbb{R}$ associated  with the problem \eqref{p1} defined  by
 \begin{equation}\label{energia}
\mathcal{J}_{\lambda, \beta}(u)= \Phi(u) - \lambda\mathcal{I}_{\mathfrak{a}}(u) - \beta \mathcal{I}_{\mathfrak{b}}(u), \mbox{ for all } u \in  \mathscr{W}
\end{equation}
where $\Phi$ is defined in the Lemma \ref{ll1},
 $$ \mathcal{I}_{\mathfrak{a}}(u)= \displaystyle{\int_{\Omega}\frac{\mathfrak{a}(x)|u|^{\mathfrak{m}_{1}(x)}}{\mathfrak{m}_{1}(x)}\,dx} \mbox{  and  } \mathcal{I}_{\mathfrak{b}}(u)= \displaystyle{\int_{\Omega}\frac{\mathfrak{b}(x)|u|^{\mathfrak{m}_{2}(x)}}{\mathfrak{m}_{2}(x)}\,dx}.$$
\noindent In our first results we will consider the problem  \eqref{p1} when the exponents are constant. For this, we assume that
  \begin{equation}\label{const} \tag{$\mathcal{H}$}
 \left\{ \begin{array}{lcc}
1< \mathfrak{m}_1<l \leqslant p \leqslant m < \mathfrak{m}_2 < p_{s}^{\star}=\frac{Np}{N-sp}, s \in(0,1);\\
             \\ 1<l \leqslant p \leqslant m< \frac{N}{s};  \\
             \\ (\mathfrak{m}_2-1l)(m-l)<(\mathfrak{m}_2-l)(m-\mathfrak{m}_1).
             \end{array}
   \right.
 \end{equation}
We assume   that map $\mathcal{A}:\mathbb{R}\to \mathbb{R}$  satisfies the conditions $(a_{1})$-$(a_2)$  and additionally, the conditions:
 \begin{itemize}
\item[$(a'_{3})$] $\mathcal{A}:\mathbb{R}\to \mathbb{R}$ is a map  of class $C^{2}(\mathbb{R}, \mathbb{R})$ and all $  t \in \mathbb{R}$ is hold:
 \begin{equation*}
 \begin{split}
& (i) \hspace{0.2cm} l\mathscr{A}(t) \leqslant \mathcal{A}(t)t\leqslant p\mathscr{A}(t) ;
             \\& (ii) \hspace{0.2cm} (l-1)\mathcal{A}(t)t \leqslant t^{2}\mathcal{A}'(t)\leqslant (m-1)\mathcal{A}(t)t;              \\ & (iii) \hspace{0.2cm}  (l-2)\mathcal{A}'(t)\leqslant \mathcal{A}''(t)t \leqslant (m-2)\mathcal{A}'(t).
          \end{split}
 \end{equation*}
\end{itemize}
Therefore, we obtain the following result involving concave-convex nonlinearities. 
 
 \begin{theorem}\label{con-convex 1}
Suppose that assumptions $(a_{1})$, $(a_{2})$, $(a_{3}')$, $(\mathcal{K})$, \eqref{const} hold, and that  weight functions $\mathfrak{a}, \mathfrak{b}\in L^{\infty}(\Omega)$ are such that  $a^{+}, b^{+}\not\equiv 0$, i.e., can be sign-changing  in $\Omega$. Then there exists $ \tilde{\lambda} > 0$ such that  problem \eqref{p1}, with $\beta =1$, admits at least one  ground state solution $u$ in $\mathcal{N}^{+}_{\lambda,1}$ satisfying $\mathcal{J}_{\lambda,1}(u)< 0$ for all $0< \lambda< \tilde{\lambda}$. ($\mathcal{N}^{+}_{\lambda,1}$ is defined in \eqref{neharispace} and  $\mathcal{J}_{\lambda,1}$ is defined in \eqref{energia})
\end{theorem}
\begin{theorem}\label{con-convex 2}
Under the same conditions of Theorem  \ref{con-convex 1} there exists $ \tilde{\lambda} > 0$ such that  problem \eqref{p1}, with $\beta =1$, admits at least one  ground state solution $u$ in $\mathcal{N}^{-}_{\lambda,1}$ satisfying $\mathcal{J}_{\lambda,1}(u)> 0$ for all $0< \lambda< \tilde{\lambda}$. ($\mathcal{N}^{-}_{\lambda,1}$ is defined in \eqref{neharispace} and $\mathcal{J}_{\lambda,1}$ is defined in \eqref{energia})
\end{theorem}

Our last results are for variable exponent.
 \begin{theorem}\label{peso1}
  Suppose that assumptions $(a_1)$-$(a_3)$, $(\mathcal{K})$ hold. Let  $ q \in C^{+}(\overline{\Omega})$ and  $  \underline{\mathfrak{m}}_{1}^{-}\leqslant \underline{\mathfrak{m}}_{1}^{+}    <  p^{-}\leqslant p^{+} <\frac{N}{s}< \underline{q}^{-}\leqslant \underline{q}^{+}.$ Moreover, assume    $\mathfrak{a} \in L^{q(\cdot)}(\Omega)$ and that there exists      $ \Omega_{0}\subset \Omega$ a measurable set with nonempty interior and measure positive  such that $\mathfrak{a}(x) >0 $ for all  $ x \in \overline{\Omega}_{0}$. Then there exists $\lambda^{\star}>0$ such that any $\lambda \in (0, \lambda^{\star})$ is an eigenvalue of  problem \eqref{p1} in $\mathscr{W}$ whenever $\beta=0$. 
 \end{theorem}

 \begin{theorem}\label{peso2}
  Suppose that assumptions $(a_1)$-$(a_3)$, $(\mathcal{K})$ hold.  Let  $ q \in C^{+}(\overline{\Omega})$,  $p^{-}\leqslant p^{+}< \underline{\mathfrak{m}}_{1}^{-}\leqslant \underline{\mathfrak{m}}_{1}^{+}$, and $\mathfrak{m}_{1}(x)<p^{\star}_{s}(x)$ for all $x \in \overline{\Omega}$. Moreover, let   $\mathfrak{a} \in L^{q(\cdot)}(\Omega)$  and  $ q(x)> \sup\, \Big \{ 1, \frac{Np(x)}{Np(x)+sp(x)\mathfrak{m}_{1}(x)-N\mathfrak{m}_{1}(x)} \Big\}$ for all $x \in \overline{\Omega}$. \\
  Then we have:
\begin{itemize}
\item[$1)$] There are $\lambda^{\star \star}$ and $\mu^{\star \star},$ positive and negative eigenvalue of  problem \eqref{p1}, respectively, satisfying $\mu^{\star \star} \leqslant\mu_{\star}<0< \lambda_{\star}\leqslant \lambda^{\star \star}$ in $\mathscr{W}$ whenever $\beta=0$. ($\lambda^{\star\star}$, $\mu^{\star \star}$, $\lambda_{\star}$ and $\mu_{\star}$ are defined in \eqref{definin})
\item[$2)$] $\lambda \in (-\infty, \mu^{\star \star})\cup(\lambda^{\star \star}, +\infty)$ is an eigenvalue of problem \eqref{p1}, while every $ \lambda \in (\mu_{\star}, \lambda_{\star})$ is not an eigenvalue in $\mathscr{W}$ whenever $\beta=0$.
\end{itemize}
\end{theorem}

\noindent\textbf{Contents of the paper:}
 In Section 2, we introduce some notation and preliminary results about Lebesgue and fractional Sobolev spaces with variable exponents. Besides, we define the function space for the problem \eqref{p1},  we will show an important result of compact and continuous embedding and we prove properties of Euler Lagrange functional associated to the problem \eqref{p1}. In Section 3, of form  constructive we prove the  Theorem \ref{con-convex 1} and   Theorem \ref{con-convex 2},  using constrained minimization in Nehari sets.
  Finally, in Section 4 using  Ekeland’s variational principle and the direct method of the calculus of variations  we prove the  Theorems  \ref{peso1} and \ref{peso2}.

 \section{Preliminaries}\label{sectio2}
 \hfill \break
 In this section, we review some notation and auxiliary results  for the Lebesgue and fractional Sobolev spaces with variable exponent which  will be useful throughout this paper to discuss the problem \eqref{p1}.
 
  The basic properties of the  Lebesgue spaces with variable exponent can be found in \cite{alves,dien,edf,fann,fan,radu} and references therein. Let $\Omega \subset \mathbb{R}^{N}$ be an bounded domain. Put
  $$C^{+}(\overline{\Omega})=\big\{h\in C(\overline{\Omega}): h(x)>1\mbox{ for all } x \in \overline{\Omega}\big\}$$
 and for all $h \in C^{+}(\overline{\Omega}) $, 
 we define $ \underline{h}^{-} := \inf_{x\in \overline{\Omega}}\,h(x) \mbox{ and } \underline{h}^{+} :=\sup_{x\in \overline{\Omega}}\,h(x).$
 
 For $h \in C^{+}(\overline{\Omega}) $,  the variable exponent Lebesgue space $L^{h(\cdot)}(\Omega)$  is defined by

 \begin{equation}\label{lp}
L^{h(\cdot)}(\Omega):=\Big\{u:\Omega \to \mathbb{R} \mbox{  measurable }: \exists\hspace{0.1cm} \zeta>0: \int_{\Omega}\Big|\frac{u(x)}{\zeta}\Big|^{h(x)}\, dx <+\infty \Big\}.
\end{equation} 
 We consider this space endowed with the so-called Luxemburg norm
 $$\|u\|_{L^{h(\cdot)}(\Omega)}:= \inf\Big\{\zeta>0:  \int_{\Omega}\Big|\frac{u(x)}{\zeta}\Big|^{h(x)}\, dx\leqslant 1 \Big\}.$$
 When  $p$ is constant, the Luxemburg norm $\|\cdot\|_{L^{h(\cdot)}(\Omega)}$ coincide with the standard norm $\|\cdot\|_{L^{h}(\Omega)}$ of the Lebesgue space $L^{h}(\Omega).$
\begin{proposition}\label{hold3}
\begin{itemize}
\item[$(a)$] The space $(L^{h(\cdot)}(\Omega), \|\cdot\|_{L^{h(\cdot)}(\Omega)})$ is a separable and reflexive Banach space;
\item[$(b)$]  Let $h_i \in C^{+}(\overline{\Omega})$ for $i=1,\ldots m$ with  $\sum_{i=1}^{m}\frac{1}{h_{i}(x)}=1$. If   $u_{i} \in L^{h_i(\cdot)}(\Omega)$, then
 \begin{equation*} 
\int_{\Omega}|u_{1}(x)\cdots u_m(x)|\,dx\leqslant C_{H}\|u_{1}\|_{ L^{h_1(\cdot)}(\Omega)}\cdots\|u_m\|_{ L^{h_m(\cdot)}(\Omega)}
\end{equation*}
where $C_{H}=\frac{1}{\underline{h}^{-}_{1}}+\frac{1}{\underline{h}^{-}_{2}}+\cdots+\frac{1}{\underline{h}^{-}_{m}}$.
\end{itemize}

\end{proposition}

 Let $h$ be a function in $C^{+}(\overline{\Omega})$. An important role in manipulating the generalized Lebesgue spaces is played by the $h(\cdot)$-modular of the  space $ L^{h(\cdot)}(\Omega)$, which is the convex
function $\rho_{h(\cdot)}: L^{h(\cdot)}(\Omega)\to \mathbb{R}$  defined by 
$$\rho_{h(\cdot)}(u)=\int_{\Omega}|u(x)|^{h(x)}dx,$$
 along any function  $u $ in  $ L^{h(\cdot)}(\Omega)$.\\
 
 The following result show relations between  the norm $\|\cdot\|_{L^{h(\cdot)}(\Omega)}$  and modular $\rho_{h(\cdot)}(\cdot)$.
 \begin{proposition}
\label{masmenos}
For $u\in L^{h(\cdot)}(\Omega) $ and $(u_{k})_{k\in \mathbb{N}}\subset L^{h(\cdot)}(\Omega),$ we have
\begin{itemize}
  \item[$(a)$]For $u \in L^{h(\cdot)}(\Omega)\setminus \{0 \}$, $\zeta = \|u\|_{L^{h(\cdot)}(\Omega) } $ if and only if $ \rho_{h(\cdot)}\big(\frac{u}{\zeta}\big)=1 $;
    \item[$(b)$]$\|u\|_{L^{h(\cdot)}(\Omega)}\geqslant 1\Rightarrow \|u\|_{L^{h(\cdot)}(\Omega) }^{\underline{h}^{-}}\leqslant \rho_{h(\cdot)}(u)\leqslant \|u\|_{L^{h(\cdot)}(\Omega) }^{\underline{h}^{+}};$
    \item[$(c)$] $\|u\|_{L^{h(\cdot)}(\Omega)}\leqslant 1\Rightarrow \|u\|_{L^{h(\cdot)}(\Omega) }^{\underline{h}^{+}}\leqslant \rho_{h(\cdot)}(u)\leqslant \|u\|_{L^{h(\cdot)}(\Omega) }^{\underline{h}^{-}};$ 
 \item[$(d)$] $\lim\limits_{k\to+\infty} \|u_{k}\|_{L^{h(\cdot)}(\Omega)}=0 \Leftrightarrow \lim\limits_{k\to+\infty} \rho_{h(\cdot)}(u_{k})=0;$
 \item[$(e)$] $\lim\limits_{k\to+\infty} \|u_{k}\|_{L^{h(\cdot)}(\Omega)}=+\infty \Leftrightarrow \lim\limits_{k\to+\infty} \rho_{h(\cdot)}( u_{k})=+\infty.$
\end{itemize}
\end{proposition}
\begin{proposition}\label{binge}
Let $h_1\in L^{\infty}(\Omega)$  such that $1 \leqslant h_1(x)h_{2}(x)\leqslant+\infty$  for a.e. $x$ in $\Omega$. Let  $u \in L^{h_2(\cdot)}(\Omega)$ and $u\neq 0$. Then 
\begin{equation*} 
\begin{split}
\|u\|_{L^{h_1(\cdot)h_2(\cdot)}(\Omega)}&\leqslant 1  \Rightarrow  \|u\|_{L^{h_1(\cdot)h_2(\cdot)}(\Omega)}^{\underline{h}_{1}^{+}}\leqslant \||u|^{h_{1}(x)}\|_{L^{h_2(\cdot)}(\Omega)} \leqslant \|u\|_{L^{h_1(\cdot)h_2(\cdot)}(\Omega)}^{\underline{h}_{1}^{-}},
\\
\|u\|_{L^{h_1(\cdot)h_2(\cdot)}(\Omega)}&\geqslant 1  \Rightarrow  \|u\|_{L^{h_1(\cdot)h_2(\cdot)}(\Omega)}^{\underline{h}_{1}^{-}}\leqslant \||u|^{h_{1}(x)}\|_{L^{h_2(\cdot)}(\Omega)} \leqslant \|u\|_{L^{h_1(\cdot)h_2(\cdot)}(\Omega)}^{\underline{h}_{1}^{+}}.
\end{split}
\end{equation*}
\end{proposition}

  \subsection{Fractional Sobolev spaces with variable exponents}\label{espacop}
  \hfill \break
 In this subsection, we introduce the fractional Sobolev spaces with variable exponent and some embedding results. The  properties for this space and  the results can be found in \cite{anaour,bahrouni,ky,kaufmann}.
 
 Let $\Omega \subset \mathbb{R}^{N}$ is a smooth bounded domain,  $s \in (0,1)$ and $q:\overline{\Omega}\to \mathbb{R}$,  $p:\overline{\Omega}\times\overline{\Omega}\to \mathbb{R}$  two continuous function. We consider that  
 \begin{equation}  \label{q10}\tag{$\mathit{q}_{1}$}
 \begin{split}
    &p \mbox{ is symmetric, }  \mbox{ this is, } p(x,y)=p(y,x),
   \\ &1<\underline{p}^{-}:=\inf_{(x,y)\in \overline{\Omega}\times\overline{\Omega} }\,p(x,y)\leqslant p(x,y)\leqslant\sup_{(x,y)\in \overline{\Omega}\times\overline{\Omega}}\,p(x,y):=\underline{p}^{+}<+\infty,
   \end{split}
\end{equation}
 and
 \begin{equation}\label{qq1} \tag{$\mathit{q}_{2}$}
    1<\underline{q}^{-}:=\inf_{x\in \overline{\Omega} }\,q(x)\leqslant q(x)\leqslant\sup_{x\in \overline{\Omega}}\,q(x):=\underline{q}^{+}<+\infty.
\end{equation}
We introduce the fractional Sobolev space with variable
exponents as follows:
\begin{equation*}
\begin{split}
    W^{s,q(\cdot),p(\cdot, \cdot)}(\Omega) 
    &:=\Big\{u\in L^{q(\cdot)}(\Omega): \int_{\Omega \times \Omega}\frac{|u(x)-u(y)|^{p(x,y)}}{\zeta^{p(x,y)}|x-y|^{N+sp(x,y)}}\, dx\,dy< +\infty, \mbox{ for some } \zeta>0 \Big\} 
\end{split}
\end{equation*}
 and we set
$$[u]^{s,p(\cdot,\cdot)}_{ \Omega}= \inf\Big\{ \zeta>0 : \int_{\Omega\times \Omega}\frac{|u(x)-u(y)|^{p(x,y)}}{\zeta^{p(x,y)}|x-y|^{N+sp(x,y)}} \,dx\,dy \leqslant 1 \Big\}$$
the  variable exponent Gagliardo seminorm.

 It is already known    that
$W^{s,q(\cdot),p(\cdot, \cdot)}(\Omega)$ is a separable and reflexive Banach space with the norm
 $$ \|u\|_{W^{s,q(\cdot), p(\cdot,\cdot)}(\Omega)}:=\|u\|_{L^{q(\cdot)}(\Omega)}+[u]_{\Omega}^{s,p(\cdot,\cdot)},$$
 see     \cite{azr,bahrouni,kaufmann}.
 \begin{remark}
 For brevity, when $q(x)=p(x,x)$ we denote $p(x)$ instead of $p(x, x)$ and  we will write $W^{s, p(\cdot, \cdot)}(\Omega)$ instead of $W^{s,p(\cdot), p(\cdot, \cdot)}(\Omega).$
 \end{remark}
 The next  result is an consequence of  \cite[Theorem 3.2]{ky}. 
\begin{corollary}\label{3.4a}
Let $\Omega\subset\mathbb{R}^{N}$  a smooth bounded  domain, $s\in(0,1)$, $p(x,y)$ and  $p(x)$ be continuous variable exponents such that \eqref{q10}-\eqref{qq1} be satisfied and that $ s\underline{p}^{+} < N$. Then, for all $r :\overline{\Omega}\rightarrow(1,+\infty)$  a continuous function such that
$p^{\star}_{s}(x)>r(x)$  for all $ x \in \overline{\Omega}$, 
 the space $W^{s, p(\cdot,\cdot)}(\Omega)$ is continuously and compactly embedding  in $L^{r(\cdot)}(\Omega)$.
\end{corollary} 
Now, we consider the space
\begin{equation*}
\begin{split}
  W^{s, p(\cdot,\cdot)}(\mathbb{R}^{N}):= \Big\{u\in L^{p(\cdot)}(\mathbb{R}^{N}): \int_{\mathbb{R}^{N} \times \mathbb{R}^{N}}\frac{|u(x)-u(y)|^{p(x,y)}}{\zeta^{p(x,y)}|x-y|^{N+sp(x,y)}} \,dx\,dy < +\infty, \mbox{ for some } \zeta>0 \Big\} 
\end{split}
\end{equation*}
where the space ${L^{p(\cdot)}(\mathbb{R}^{N})}$ is defined analogous the space ${L^{p(\cdot)}(\Omega)}$. The corresponding norm for this space  is
$$ \| u \|:= \| u \|_{L^{p(\cdot)}(\mathbb{R}^{N})} + [u]^{s, p(\cdot, \cdot)}_{\mathbb{R}^{N}}.$$ The space $ (W^{s, p(\cdot,\cdot)}(\mathbb{R}^{N}), \|\cdot\|)$  has the same properties that  $( W^{s, p(\cdot,\cdot)}(\Omega), \|\cdot\|_{W^{s, p(\cdot,\cdot)}(\Omega)})$, this is, it is  a  reflexive and separable Banach space.

 Now we define the space were will study the problem \eqref{p1}. Let we will consider   the variable exponents  $p(x):=p(x,x)$ for all $x \in \mathbb{R}^{N}$ with   $p \in C(\mathbb{R}^{N}\times \mathbb{R}^{N})$ satisfying \eqref{a23}  and we denote by  
 \begin{equation*}\label{space}
  \mathscr{W} = W^{s,p(\cdot, \cdot)}_{0}:= \{u \in W^{s, p(\cdot, \cdot)}(\mathbb{R}^{N}): u=0 \mbox{ a.e. in }   \mathbb{R}^{N}\setminus \Omega \}. 
 \end{equation*}
Note that $\mathscr{W}$ is a   closed subspace of $W^{s, p(\cdot, \cdot)}(\mathbb{R}^{N})$. Then,  $ \mathscr{W}$ is a reflexive and  separable Banach space with the norm
$$ \| u \|:= \| u \|_{L^{p(\cdot)}(\Omega)} + [u]^{s, p(\cdot, \cdot)}_{\mathbb{R}^{N}},$$ once the norms $\|\cdot\|_{L^{p(\cdot)}(\mathbb{R}^{N})}$  and $\|\cdot\|_{L^{p(\cdot)}(\Omega)}$  coincide in $ \mathscr{W} $. 

 
 The proofs for the next lemmas  under    will be referred to in Appendix \ref{apendice}.
 
\begin{lemma}\label{lw1}
Assume $\Omega$ be a smooth bounded domain in $\mathbb{R}^{N}$. Let  $p(x):=p(x,x)$ for all $x \in \mathbb{R}^{N}$ with $p \in C(\mathbb{R}^{N}\times \mathbb{R}^{N})$ satisfying \eqref{a23}  and $p^{\star}_{s}(x) > p(x)$ for $x \in \mathbb{R}^{N}$.   Then there exists $ \zeta_1>0$ such that
\begin{equation*}\label{w1}
\|u\|_{L^{p(\cdot)}(\Omega)}\leqslant \frac{1}{\zeta_1}[u]^{s, p(\cdot, \cdot)}_{\mathbb{R}^{N}}\mbox{ for all } \hspace{0.1cm} u \in \mathscr{W}.
\end{equation*}
\end{lemma}


   \begin{lemma}\label{2.11}
Assume $\Omega$ be a smooth bounded domain in $\mathbb{R}^{N}$. Let $p(x):=p(x,x)$ for all $x \in \mathbb{R}^{N}$ with  $p \in C(\mathbb{R}^{N}\times \mathbb{R}^{N})$ satisfying \eqref{a23}  and $p^{\star}_{s}(x) > p(x)$ for $x \in \mathbb{R}^{N}$. 
Assume that $r:\overline{\Omega}\rightarrow (1, +\infty)$ is a continuous function. Thus, the space $(\mathscr{W}, \|\cdot\|_{\mathscr{W}})$  is continuously and compactly embedding   in $L^{r(\cdot)}(\Omega)$ for all $r(x) \in (1, p^{\star}_{s}(x))$ for all $x \in \overline{\Omega}$.
\end{lemma}

\begin{remark}
From Lemma \ref{lw1} we will consider the space  $\mathscr{W}$ with norm $\|u\|_{\mathscr{W}}= [u]^{s, p(\cdot, \cdot)}_{\mathbb{R}^{N}}$.         Therefore, $(\mathscr{W}, \|\cdot\|_{\mathscr{W}}) $ is a reflexive and separable Banach space.
\end{remark}
  

 \noindent For all $u \in \mathscr{W} $ denoting 
 the  convex modular function $\rho_{\mathscr{W}}: \mathscr{W} \to \mathbb{R}$ defined by
  $$\displaystyle{\rho_{\mathscr{W}}(u)= \int_{\mathbb{R}^{N}\times \mathbb{R}^{N}}\frac{|u(x)-u(y)|^{p(x,y)}}{|x-y|^{N+sp(x,y)}} \,dx\,dy}.$$
The same way that the Proposition \ref{masmenos}, the following proposition has an important role in manipulating results regarding the relationship between the norm $\|\cdot\|_{\mathscr{W}}$   and the $\rho_{\mathscr{W}}$ convex modular function.
\begin{proposition}\label{lw0}
For $u\in \mathscr{W}$ and $(u_{k})_{k\in \mathbb{N}}\subset\mathscr{W}$, we have
\begin{itemize}
\item[$(a)$]For $u \in \mathscr{W}\setminus \{0 \}$, $\zeta = \|u\|_{\mathscr{W} } $ if and only if $ \rho_{\mathscr{W}}\big(\frac{u}{\zeta}\big)=1 $;
    \item[$(b)$]$\|u\|_{  \mathscr{W}}\geqslant 1\Rightarrow \|u\|_{ \mathscr{W}}^{p^{-}}\leqslant \rho_{ \mathscr{W}}(u)\leqslant \|u\|_{  \mathscr{W} }^{p^{+}};$
    \item[$(c)$] $\|u\|_{  \mathscr{W}}\leqslant 1\Rightarrow \|u\|_{  \mathscr{W}}^{p^{+}}\leqslant \rho_{ \mathscr{W}}(u)\leqslant \|u\|_{  \mathscr{W} }^{p^{-}};$ 
    \item[$(d)$] $\lim\limits_{k\to+\infty} \|u_{k}-u\|_{  \mathscr{W}}=0 \Leftrightarrow \lim\limits_{k\to+\infty} \rho_{ \mathscr{W}}(u_{k}-u)=0;$
 \item[$(e)$] $\lim\limits_{k\to+\infty} \|u_{k}\|_{  \mathscr{W}}=+\infty \Leftrightarrow \lim\limits_{k\to+\infty} \rho_{ \mathscr{W}}(u_{k})=+\infty.$
\end{itemize}
\end{proposition}

 \begin{lemma} \label{ll1}
Assume that  $(a_{1})$-$(a_{3})$, and $(\mathcal{K})$ is  hold. We consider the functional $\Phi:\mathscr{W}\to \mathbb{R}$ defined by
\begin{equation*}
\Phi(u)= \int_{\mathbb{R}^{N}\times \mathbb{R}^{N}}\mathscr{A}(u(x)-u(y))K(x,y)\,dx\,dy  \mbox{ for all }\hspace{0.1cm} u \in \mathscr{W},
\end{equation*}
has the following  properties:
\begin{itemize}
\item[$(i)$] The functional $\Phi$ is well defined on $\mathscr{W}$, is of   class  $C^{1}(\mathscr{W}, \mathbb{R})$, and its G\^ateaux derivative is given by
\begin{equation*}\label{phi'}
\langle \Phi'(u), v \rangle = \int_{\mathbb{R}^{N}\times \mathbb{R}^{N}}\mathcal{A}(u(x)-u(y))(v(x)-v(y))K(x,y)\,dx\,dy \mbox{ for all } \hspace{0.1cm} u, v \in \mathscr{W}.
\end{equation*} 
\item[$(ii)$]  The functional $\Phi$ is weakly lower semicontinuous, that is, $u_k \rightharpoonup u$ in $\mathscr{W}$ as $ k \to +\infty$ implies that $\Phi(u) \leqslant\displaystyle{\liminf_{k\to +\infty} \Phi(u_k)}$. 
\item[$(iii)$] The  functional $\Phi' : \mathscr{W}\to \mathscr{W}'$ is an operator of type $(S_{+})$ on $\mathscr{W}$, that is, if 
  \begin{equation} \label{inffo}
  u_k \rightharpoonup u  \mbox{ in }\mathscr{W} \mbox{ and } \limsup_{k \to +\infty}\,\langle \,\Phi'(u_k), u_k-u \rangle\leqslant 0,
  \end{equation}
  then $u_k\to u$ in $\mathscr{W}$ as $k\to +\infty$.
\end{itemize}
\end{lemma}
The proof of the  result above    will be referred to in Appendix \ref{apendice}.



  


\section{Proof of Theorems  \ref{con-convex 1} and  \ref{con-convex 2}}
   \hfill \break
 Now will show the    existence  of solution   to  problem \eqref{p1}   for constants exponents   with concave-convex nonlinearities and weight functions   $\mathfrak{a}, \mathfrak{b}:\overline{\Omega} \to \mathbb{R}$ that  are  sign-changing  in $\Omega$. In this case, the space $\mathscr{W}$  coincide with  the space  $W^{s,p}_{0}(\Omega):=\lbrace u \in W^{s,p}(\mathbb{R}^{N}): u=0 \mbox{  a.e. in }\mathbb{R}^{N} \setminus \Omega \}$ defined in \cite{ian1}, then  $\mathscr{W}=W^{s,p}_{0}(\Omega)$. We consider  $\mathcal{J}_{\lambda,1} $ the Euler Lagrange functional associated to  problem  \eqref{p1}. To proof  Theorem \ref{con-convex 1} and  Theorem \ref{con-convex 2} we will consider the Nehari manifold $\mathcal{N}_{\lambda,1}$ introduced in \cite{nehari},  the fibering map and the different “sign-subsets” of the Nehari set that will be used to find critical points of the  Euler Lagrange functional $\mathcal{J}_{\lambda,1} $. 


 \subsection{The Nehari Manifold}\label{nehari}
\hfill \break
The Nehari manifold associated to the functional $\mathcal{J}_{\lambda,1}$ is given by

\begin{equation}\label{3.1}
\begin{split}
\mathcal{N}_{\lambda,1} = & \Bigg\{ u \in W_{0}^{s,p}(\Omega)\setminus\{0\}: \langle\mathcal{J}'_{\lambda,1}(u),u\rangle=0\Bigg\}\\ =& \Bigg\{ u \in W_{0}^{s,p}(\Omega)\setminus\{0\}: \lambda\int_{\Omega}\mathfrak{a}(x)|u|^{\mathfrak{m}_1}\,dx+ \int_{\Omega}\mathfrak{b}(x)|u|^{\mathfrak{m}_{2}}\,dx  \\ &= \int_{\mathbb{R}^{N}\times \mathbb{R}^{N}}\mathcal{A}(u(x)-u(y))(u(x)-u(y))K(x,y)\,dx\,dy \Bigg\} .
\end{split}
\end{equation}
Note that when $u \in \mathcal{N}_{\lambda,1}$, by \eqref{3.1} we obtain
\begin{equation}\label{not1}
\begin{split}
\mathcal{J}_{\lambda,1}(u)=& \Phi(u) - \frac{1}{\mathfrak{m}_{1}}\int_{\mathbb{R}^{N}\times\mathbb{R}^{N}}\mathcal{A}(u(x)-u(y))(u(x)-u(y))K(x,y)\,dx\,dy \\&+\Bigg(\frac{1}{\mathfrak{m}_{1}}- \frac{1}{\mathfrak{m}_{2}}\Bigg)\int_{\Omega}\mathfrak{b}(x)|u|^{\mathfrak{m}_{2}}\,dx,
\end{split}
\end{equation} 
or it can be rewritten as
\begin{equation}\label{not2}
\begin{split}
\mathcal{J}_{\lambda,1}(u)=& \Phi(u) - \frac{1}{\mathfrak{m}_{2}}\int_{\mathbb{R}^{N}\times\mathbb{R}^{N}}\mathcal{A}(u(x)-u(y))(u(x)-u(y))K(x,y)\,dx\,dy\\&+\lambda\Bigg(\frac{1}{\mathfrak{m}_{2}}- \frac{1}{\mathfrak{m}_{1}}\Bigg)\int_{\Omega}\mathfrak{a}(x)|u|^{\mathfrak{m}_{1}}\,dx.
\end{split}
\end{equation}
The characterization above for the functional $\mathcal{J}_{\lambda,1}$   is relevant to results we will study the following.

As first step, we shall prove that $\mathcal{J}_{\lambda,1}$ is coercive and bounded below on $\mathcal{N}_{\lambda,1}\subset W_{0}^{s,p}(\Omega)$ which allows us to find a ground state solution that is a critical point for $\mathcal{J}_{\lambda,1}$.

\begin{proposition}\label{propo3.1}
The functional $\mathcal{J}_{\lambda,1}$  is coercive and bounded below on $\mathcal{N}_{\lambda,1}$.
\end{proposition}
\begin{proof}
For $u \in \mathcal{N}_{\lambda,1} $ using  \eqref{not2}, $(a_{2})$, $(a'_{3})$-$(i)$,  $(\mathcal{K})$, and \eqref{const}, we obtain
\begin{equation}\label{3.2}
\begin{split}
\mathcal{J}_{\lambda,1}(u) \geqslant & \Bigg( \frac{1}{p}-\frac{1}{\mathfrak{m}_{2}} \Bigg )  c_{\mathcal{A}}b_{0}\int_{\mathbb{R}^{N}\times\mathbb{R}^{N}}\frac{|u(x)-u(y)|^{p}}{|x-y|^{N+sp}}\,dx\,dy  +\lambda\Bigg(\frac{1}{\mathfrak{m}_{2}}- \frac{1}{\mathfrak{m}_{1}}\Bigg)\int_{\Omega}\mathfrak{a}(x)|u|^{\mathfrak{m}_{1}}\,dx.
\end{split}
\end{equation}
Now, from continuous embedding  $W^{s,p}_{0}(\Omega)\hookrightarrow L^{\mathfrak{m}_{1}}(\Omega)$,  $\mathfrak{a} \in L^{\infty}(\Omega)$, $\mathfrak{a}^{+}\not\equiv 0$, and  \eqref{const}, it follows that 
\begin{equation}\label{3.3}
\int_{\Omega}\mathfrak{a}(x)|u|^{\mathfrak{m}_{1}}\,dx \leqslant \|\mathfrak{a}^{+}\|_{\infty}\|u\|_{L^{\mathfrak{m}_{1}}}^{\mathfrak{m}_{1}} \leqslant \|\mathfrak{a}^{+}\|_{\infty}C_{\mathfrak{m}_1} ^{\mathfrak{m}_{1}(\Omega)}\|u\|_{W^{s,p}_{0}(\Omega)}^{\mathfrak{m}_{1}}.
\end{equation}
Then by \eqref{3.2} and \eqref{3.3}, we infer that\begin{equation*}
\begin{split}
\mathcal{J}_{\lambda,1}(u) \geqslant & \Bigg( \frac{1}{p}-\frac{1}{\mathfrak{m}_{2}} \Bigg )  c_{\mathcal{A}}b_{0}\|u\|_{W^{s,p}_{0}(\Omega)}^{p}  +\lambda\Bigg(\frac{1}{\mathfrak{m}_{2}}- \frac{1}{\mathfrak{m}_{1}}\Bigg)\|\mathfrak{a}^{+}\|_{\infty}C_{\mathfrak{m}_1} ^{\mathfrak{m}_{1}}\|u\|_{W^{s,p}_{0}(\Omega)}^{\mathfrak{m}_{1}}.
\end{split}
\end{equation*}
Therefore, since $p> \mathfrak{m}_{1}$ $\mathcal{J}_{\lambda,1}$ is coercive and consequently  $\mathcal{J}_{\lambda,1}$ is bounded below on $\mathcal{N}_{\lambda,1}$.

\end{proof}

Let us introduce the fibering maps associated to the functional $\mathcal{J}_{\lambda,1}$.  For every fixed $u \in W^{s,p}_{0}(\Omega)\setminus\{0\}$,  we will define the fibering map $\wp_{u}: (0, +\infty) \to \mathbb{R}$  by
\begin{equation*}
\wp_{u}(t):=\mathcal{J}_{\lambda,1}(tu)= \Phi(tu) - \frac{\lambda t^{\mathfrak{m}_{1}}}{\mathfrak{m}_1}\int_{\Omega}\mathfrak{a}(x)|u|^{\mathfrak{m}_1}\,dx- \frac{t^{\mathfrak{m}_{2}}}{\mathfrak{m}_{2}}\int_{\Omega}\mathfrak{b}(x)|u|^{\mathfrak{m}_{2}}\,dx \mbox{ for all } t \in (0, +\infty).
\end{equation*}

Our objective is we will analyze the behavior the fibering maps and show its relation with the Nehari manifold. More specifically as fibering maps are considered together with the Nehari manifold in order to ensure the existence of critical points for $\mathcal{J}_{\lambda,1}$.  In particular, for concave-convex nonlinearities,  knowledge  the geometry for $\wp_{u}$ is important, see for instance \cite{brown}.

 Furthermore,  using again arguing  as in the  Lemma \ref{ll1} and   standard arguments, we conclude that $\wp_{u}$ is of class $C^{1}(\mathbb{R}^{+}, \mathbb{R})$.  Then differentiating  $\wp_{u}(t)$ with respect to $t$, we obtain
\begin{equation} \label{3.4}
\begin{split}
\wp'_{u}(t)=  & \int_{\mathbb{R}^{N}\times \mathbb{R}^{N}}\mathcal{A}(tu(x)-tu(y))(u(x)-u(y))K(x,y)\,dx\,dy \\ &-\lambda^{\mathfrak{m}_{1}-1}\int_{\Omega}\mathfrak{a}(x)|u|^{\mathfrak{m}_{1}}\,dx  -\lambda^{\mathfrak{m}_{2}-1}\int_{\Omega}\mathfrak{b}(x)|u|^{\mathfrak{m}_{2}}\,dx.
\end{split}
\end{equation}

 Therefore, $ tu\in \mathcal{N}_{\lambda,1} $ if and only if $\wp_{u}'(t)=0$. In particular, $ u\in \mathcal{N}_{\lambda,1} $ if and only if $\wp_{u}'(1)=0$. In other words, it is sufficient to find stationary
points of fibering maps in order to get critical points for $\mathcal{J}_{\lambda,1}$ on $ \mathcal{N}_{\lambda,1} $.

Furthermore, again arguing  as in the  Lemma \ref{ll1} and   standard arguments, we  have that $\wp_{u}$ is of class $C^{2}(\mathbb{R}^{+}, \mathbb{R})$ with
second derivative given by
\begin{equation} \label{3.5}
\begin{split}
\wp''_{u}(t)=  & \int_{\mathbb{R}^{N}\times \mathbb{R}^{N}}\mathcal{A}'(tu(x)-tu(y))(u(x)-u(y))^{2}K(x,y)\,dx\,dy \\ &-\lambda t^{\mathfrak{m}_{1}-2}(\mathfrak{m}_{1}-1)\int_{\Omega}\mathfrak{a}(x)|u|^{\mathfrak{m}_{1}}\,dx  -t^{\mathfrak{m}_{2}-2}(\mathfrak{m}_{2}-1)\int_{\Omega}\mathfrak{b}(x)|u|^{\mathfrak{m}_{2}}\,dx.
\end{split}
\end{equation}
Thus, as $\wp''_{u} \in C^{2}(\mathbb{R}^{+}, \mathbb{R})$  it is natural to divide $ \mathcal{N}_{\lambda,1} $ into three sets as was pointed by \cite{brown1,brown}:
\begin{equation}\label{neharispace}
\begin{split}
 &\mathcal{N}_{\lambda,1}^{+}= \{ u \in \mathcal{N}_{\lambda,1}; \wp''_{u}(1) >0 \};
 \\ &\mathcal{N}_{\lambda,1}^{-}= \{ u \in \mathcal{N}_{\lambda,1}; \wp''_{u}(1) <0 \}; \\  &\mathcal{N}_{\lambda,1}^{0}= \{ u \in \mathcal{N}_{\lambda,1}; \wp''_{u}(1)=0 \}.
\end{split}
\end{equation}
Here we mention that $\mathcal{N}_{\lambda,1}^{+}, \mathcal{N}_{\lambda,1}^{-},$ and $\mathcal{N}_{\lambda,1}^{0}$ correspond to critical points of minimum,
maximum and inflexions points, respectively of $\wp_{u}.$
\begin{remark}\label{remark2}
Note that if $u \in \mathcal{N}_{\lambda,1}$, then  by \eqref{3.4} and \eqref{3.5}, we obtain
\begin{equation*}
\begin{split}
\wp''_{u}(1)= & (\mathfrak{m}_{1}-\mathfrak{m}_{2})\int_{\Omega}\mathfrak{b}(x)|u|^{\mathfrak{m}_{2}}\,dx +  \int_{\mathbb{R}^{N}\times \mathbb{R}^{N}}\mathcal{A}'(u(x)-u(y))(u(x)-u(y))^{2}K(x,y)\,dx\,dy \\ & +(1-\mathfrak{m}_{1})\int_{\mathbb{R}^{N}\times \mathbb{R}^{N}}\mathcal{A}(u(x)-u(y))(u(x)-u(y))K(x,y)\,dx\,dy \\ =& \lambda(\mathfrak{m}_{2}-\mathfrak{m}_{1})\int_{\Omega}\mathfrak{a}(x)|u|^{\mathfrak{m}_{1}}\,dx + \int_{\mathbb{R}^{N}\times \mathbb{R}^{N}}\mathcal{A}'(u(x)-u(y))(u(x)-u(y))^{2}K(x,y)\,dx\,dy \\ & +(1-\mathfrak{m}_{2})\int_{\mathbb{R}^{N}\times \mathbb{R}^{N}}\mathcal{A}(u(x)-u(y))(u(x)-u(y))K(x,y)\,dx\,dy. \\ &  
\end{split}
\end{equation*}
\end{remark}
\begin{lemma}\label{lem3.1}
For each $\lambda > 0  $ sufficiently small, we have that
\begin{itemize}
\item[$(1)$] $\mathcal{N}_{\lambda,1}^{0}= \emptyset$;
\item[$(2)$]$\mathcal{N}_{\lambda,1}=\mathcal{N}_{\lambda,1}^{+}\cup \mathcal{N}_{\lambda,1}^{-}$ is a $C^{1}$-manifold.
\end{itemize}
\end{lemma}
\begin{proof}
 $(1)$ We suppose that $\mathcal{N}_{\lambda,1}^{0}\neq \emptyset$. Let $u \in \mathcal{N}_{\lambda,1}^{0}$ be a fixed function. Thus, $\wp''_{u}(1)=\wp'_{u}(1)=0$. Using Remark \ref{remark2}, $(a_2)$, $(a'_{3})$-$(ii)$, $(\mathcal{K})$, and \eqref{const}, we obtain
\begin{equation}\label{3.44}
\begin{split}
 (\mathfrak{m}_{2}-\mathfrak{m}_{1})\int_{\Omega}\mathfrak{b}(x)|u|^{\mathfrak{m}_{2}}\,dx = & \int_{\mathbb{R}^{N}\times \mathbb{R}^{N}}\mathcal{A}'(u(x)-u(y))(u(x)-u(y))^{2}K(x,y)\,dx\,dy  \\ & +(1-\mathfrak{m}_{1})\int_{\mathbb{R}^{N}\times \mathbb{R}^{N}}\mathcal{A}(u(x)-u(y))(u(x)-u(y))K(x,y)\,dx\,dy \\ \geqslant  & (l-\mathfrak{m}_{1})c_{\mathcal{A}}b_0 \|u\|^{p}_{W^{s,p}_{0}(\Omega)}.
 \end{split}
\end{equation}
Now, from continuous embedding  $W^{s,p}_{0}(\Omega)\hookrightarrow L^{\mathfrak{m}_{2}}(\Omega)$, $\mathfrak{b} \in L^{\infty}(\Omega)$,  $\mathfrak{b}^{+}\not\equiv 0$, and  \eqref{const},  it follows that 
\begin{equation}\label{3.31}
\int_{\Omega}\mathfrak{b}(x)|u|^{\mathfrak{m}_{2}}\,dx \leqslant \|\mathfrak{b}^{+}\|_{\infty}\|u\|_{L^{\mathfrak{m}_{2}}}^{\mathfrak{m}_{2}} \leqslant \|\mathfrak{b}^{+}\|_{\infty}C_{\mathfrak{m}_2} ^{\mathfrak{m}_{2}(\Omega)}\|u\|_{W^{s,p}_{0}(\Omega)}^{\mathfrak{m}_{2}}.
\end{equation}
Thus by  \eqref{3.44}, \eqref{3.31}, and \eqref{const}, we achieve
\begin{equation}\label{3.6}
\|u\|^{\mathfrak{m}_{2}- p}_{W^{s,p}_{0}(\Omega)} \geqslant \Bigg( \frac{l-\mathfrak{m}_{1}}{\mathfrak{m}_{2}-\mathfrak{m}_{1}}\Bigg)\frac{c_{\mathcal{A}}b_0}{\|\mathfrak{b}^{+}\|_{\infty}C^{\mathfrak{m}_{2}}_{\mathfrak{m}_{2}}}:= C_1.
\end{equation}
Now, using again Remark \ref{remark2}, $(a_2)$,  $(a'_{3})$-$(ii)$, and \eqref{const},  we infer  that
\begin{equation}\label{3.66}
\begin{split}
\lambda (\mathfrak{m}_{2}-\mathfrak{m}_{1})\int_{\Omega}\mathfrak{a}(x)|u|^{\mathfrak{m}_{1}}\,dx = & \int_{\mathbb{R}^{N}\times \mathbb{R}^{N}}\mathcal{A}'(u(x)-u(y))(u(x)-u(y))^{2}K(x,y)\,dx\,dy  \\ & +\int_{\mathbb{R}^{N}\times \mathbb{R}^{N}}\mathcal{A}(u(x)-u(y))(u(x)-u(y))K(x,y)\,dx\,dy \\ & -\mathfrak{m}_{2}\int_{\mathbb{R}^{N}\times \mathbb{R}^{N}}\mathcal{A}(u(x)-u(y))(u(x)-u(y))K(x,y)\,dx\,dy\\ \geqslant & (\mathfrak{m}_{2}-m)c_{\mathcal{A}}b_0 \|u\|^{p}_{W^{s,p}_{0}(\Omega)}.
 \end{split}
\end{equation}
Therefore, by \eqref{3.3},  \eqref{3.6}, \eqref{3.66}, and \eqref{const}, we obtain that
\begin{equation*}
\lambda \geqslant \Bigg( \frac{\mathfrak{m}_{2}-m}{\mathfrak{m}_2-\mathfrak{m}_1}\Bigg)\frac{c_{\mathcal{A}}b_0}{C^{\mathfrak{m_1}}_{\mathfrak{m_1}}\|\mathfrak{a}^{+}\|_{\infty}}C_{1}^{\frac{p-\mathfrak{m}_{1}}{\mathfrak{m}_{2}-p}}>0.
\end{equation*}
Which is  a contradiction for each $\lambda>0$ small enough. Hence, the  proof of item $(1)$ it is
complete.
\\
\noindent $(2)$ Without loss of generality suppose that $u \in \mathcal{N}_{\lambda,1}^{+}$. Define the function $\mathsf{G}_{\lambda}: \mathcal{N}_{\lambda,1}^{+} \to \mathbb{R}$ 
\begin{equation*}
\begin{split}
\mathsf{G}_{\lambda}(u)=  &\langle\mathcal{J}'_{\lambda, 1}(u),u\rangle \\ =&   \int_{\mathbb{R}^{N}\times \mathbb{R}^{N}}\mathcal{A}(u(x)-u(y))(u(x)-u(y))K(x,y)\,dx\,dy - \lambda\int_{\Omega}\mathfrak{a}(x)|u|^{\mathfrak{m}_{1}}\,dx \\ & - \int_{\Omega}\mathfrak{b}(x)|u|^{\mathfrak{m}_{2}}\,dx \mbox{ for all } u \in \mathcal{N}_{\lambda,1}^{+}.
\end{split}
\end{equation*}
Note that  
\begin{equation} \label{3.7}
\langle\mathsf{G}'_{\lambda}(u),u\rangle =   \langle\mathcal{J}''_{\lambda, 1}(u)(u,u),u\rangle +  \langle\mathcal{J}'_{\lambda, 1}(u),u\rangle =  \wp''(1) \mbox{ for all } u \in \mathcal{N}_{\lambda,1}^{+}.
\end{equation}
Hence  $ \mathcal{N}_{\lambda,1}^{+}= \mathsf{G}^{-1}_{\lambda}(\{0\})$ is a $ C^{1}$-manifold. Indeed,   for $u \in  \mathcal{N}_{\lambda,1}^{+}$, using \eqref{3.7}, we get $\langle \mathsf{G}'_{\lambda}(u), u \rangle > 0$. Therefore, $\langle \mathsf{G}'_{\lambda}(u), u \rangle\neq 0$. As $\mathsf{G}_{\lambda}(u)$ is class $C^{1}(W^{s,p}_{0}(\Omega), \mathbb{R})$  it follows that  $\mathcal{N}_{\lambda,1}^{+}= \mathsf{G}^{-1}_{\lambda}(\{0\})$ is a $C^{1}$-manifold.  Similarly, we may show that $\mathcal{N}_{\lambda,1}^{-}$
is a $C^{1}$-manifold. Consequently, as $\mathcal{N}_{\lambda,1}^{0}= \emptyset$   for all $\lambda > 0$ small enough, it follows  that  $\mathcal{N}_{\lambda,1}=\mathcal{N}_{\lambda,1}^{+}\cup\mathcal{N}_{\lambda,1}^{-}$ is a $C^{1}$-manifold.
\end{proof}
\begin{lemma}
 Let $u_0$ be a local minimum (or local maximum) of $ \mathcal{J}_{\lambda,1}$ in such a way that $u_0 \notin \mathcal{N}_{\lambda,1}^{0} $. Then $u_0$ is a  critical point for $ \mathcal{J}_{\lambda,1}$.
\end{lemma}
\begin{proof}
Without  loss of generality, we suppose that $u_0$ is a local minimum of $ \mathcal{J}_{\lambda,1}$. Define the function $\mathsf{H}: W^{s,p}_{0}(\Omega)\to \mathbb{R}$ by
\begin{equation*}
\begin{split}
\mathsf{H}(u)= \langle\mathcal{J}'_{\lambda, 1}(u),u\rangle  \mbox{ for all } u \in W^{s,p}_{0}(\Omega).
\end{split}
\end{equation*}
 We  observe that $u_{0}$ is a solution for the minimization problem
\begin{equation}\label{3.70}
\left\{\begin{array}{rc} 
\begin{split}
&\inf \mathcal{J}_{\lambda, 1}(u) \\
&\mathsf{H}(u)  =  0.
\end{split}
\end{array}\right.
\end{equation}
Now note that 
\begin{equation}\label{16.0}
\langle\mathsf{H}'(u), v \rangle=  \langle\mathcal{J}''_{\lambda, 1}(u)(u,u),v\rangle +  \langle\mathcal{J}'_{\lambda, 1}(u),v\rangle\mbox{ for all } u, v \in W^{s,p}_{0}(\Omega).
\end{equation}
Then taking $u = v = u_0$ in \eqref{16.0}, we infer that
\begin{equation}\label{lala}
\langle\mathsf{H}'(u_0), u_0 \rangle =  \langle\mathcal{J}''_{\lambda, 1}(u_{0})(u_{0}, u_{0}),u_{0}\rangle +  \langle\mathcal{J}'_{\lambda, 1}(u_{0}),u_{0}\rangle= \wp''_{u_{0}}(1)>0. 
\end{equation}
Thus  $u_0 \notin \mathcal{N}_{\lambda,1}^{0} $ and by Lemma \ref{lem3.1} we conclude that  problem \eqref{3.70} has a solution in the following form
$$\mathcal{J}'_{\lambda, 1}(u_{0})=\mu \mathsf{H}'(u_0), $$
where $ \mu \in \mathbb{R}$ which is given by Lagrange Multipliers Theorem. 
\noindent Since $u_{0} \in \mathcal{N}_{\lambda,1}$ we obtain that
\begin{equation}\label{crticoo}
\mu\langle \mathsf{H}'(u_0), u_{0}\rangle =  \langle\mathcal{J}'_{\lambda, 1}(u_{0}), u_0\rangle =0.
\end{equation}
However by \eqref{lala} we infer that  $\langle \mathsf{H}'(u_0), u_{0}\rangle \neq 0$. Thus, by \eqref{crticoo} we conclude that $\mu =0$. Therefore, $\mathcal{J}'_{\lambda, 1}(u_{0})=0$ and $u_{0}$ is a critical point for $\mathcal{J}_{\lambda, 1}$ on $W^{s,p}_{0}(\Omega)$.
\end{proof}

\subsection{ The fibering map}
 \hfill \break
In this subsection, we will do a complete analysis of the fibering map associated with  problem \eqref{p1}. The essential nature for the fibering maps is determined by the signs of $\displaystyle{\int_{\Omega}\mathfrak{a}(x)|u|^{\mathfrak{m}_{1}}\,dx}$ and $\displaystyle{\int_{\Omega}\mathfrak{b}(x)|u|^{\mathfrak{m}_{2}}\,dx}$.

Throughout this subsection, fixed $u \in W^{s,p}_{0}(\Omega)\setminus \{0\}$ it is useful to consider the auxiliary function $\mathsf{M}_{u}:\mathbb{R}\to \mathbb{R}$ by
\begin{equation*}
\mathsf{M}_{u}(t)=t^{-\mathfrak{m}_{1}}\int_{\mathbb{R}^{N}\times \mathbb{R}^{N}}\mathcal{A}(tu(x)-tu(y))(tu(x)-tu(y))K(x,y)\,dx\,dy - t^{-\mathfrak{m}_{1}}\int_{\Omega}\mathfrak{b}(x)|tu|^{\mathfrak{m}_{2}}\,dx
\end{equation*}
for all $t \in \mathbb{R}$.

Note that   $\mathsf{M}_{u}$  has possible forms  when $\displaystyle{\int_{\Omega}\mathfrak{b}(x)|u|^{\mathfrak{m}_{2}}\,dx \leqslant 0}$ and $\displaystyle{\int_{\Omega}\mathfrak{b}(x)|u|^{\mathfrak{m}_{2}}\,dx > 0}$,  respectively  
\begin{figure}[h]
\begin{center}
\includegraphics[scale=0.34]{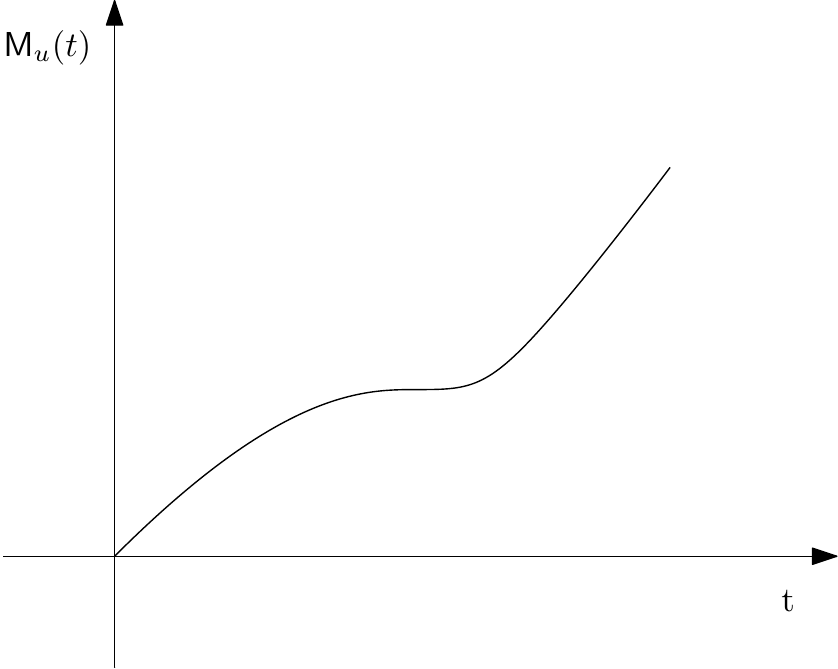}\,\,\,\,\,\,\,\,\,\,\,\,\,\,\,\,\,\,\,\,\,\,\,\,\,\,\,\,\,\,\,\,\,\,\,\,\,\,\,\,\,\,\,\,\,\,\,\,\,
\includegraphics[scale=0.28]{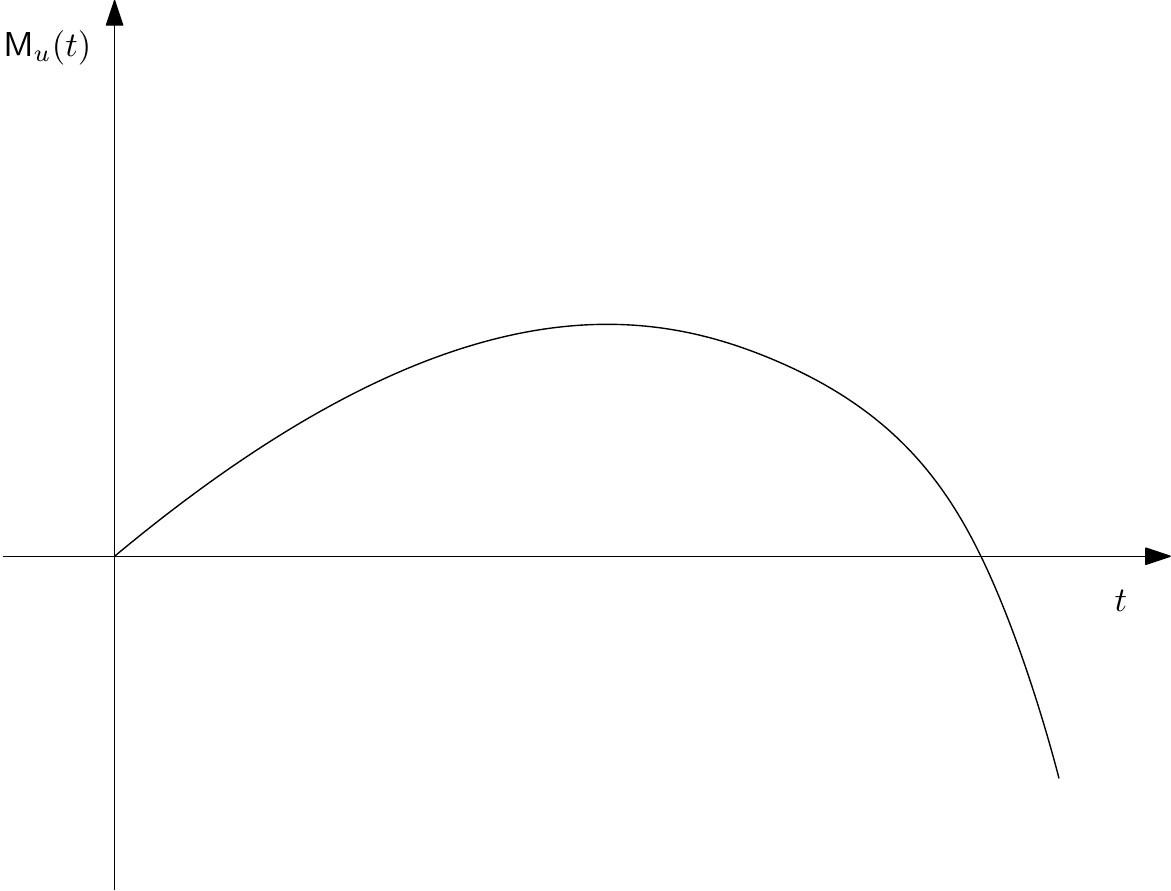}
\caption{Sketches of $\mathsf{M}_{u}$}
\end{center}
\end{figure}

\begin{lemma}\label{lem4.1}
 Let $t > 0$ be fixed. Then $tu \in \mathcal{N}_{\lambda,1}$ if and only if is a solution of $\mathsf{M}_{u}(t)= \lambda\displaystyle{\int_{\Omega}\mathfrak{a}(x)|u|^{\mathfrak{m}_{1}}\,dx.}$
\end{lemma}
\begin{proof}
Fix $t>0$ such that $tu \in \mathcal{N}_{\lambda,1}$. Then, 
\begin{equation}\label{3.71}
\begin{split}
\int_{\mathbb{R}^{N}\times \mathbb{R}^{N}}\mathcal{A}(tu(x)-tu(y))(tu(x)-tu(y))K(x,y)\,dx\,dy =& \lambda t^{\mathfrak{m}_1} \int_{\Omega}\mathfrak{a}(x)|u|^{\mathfrak{m}_{1}}\,dx \\& + t^{\mathfrak{m}_2}\int_{\Omega}\mathfrak{b}(x)|u|^{\mathfrak{m}_{2}}\,dx.
\end{split}
\end{equation}
Thus, multiplying \eqref{3.71} by $t^{- \mathfrak{m}_1}$, we have that
\begin{equation*}
\begin{split}
\lambda \int_{\Omega}\mathfrak{a}(x)|u|^{\mathfrak{m}_{1}}\,dx =& 
t^{- \mathfrak{m}_1}\int_{\mathbb{R}^{N}\times \mathbb{R}^{N}}\mathcal{A}(tu(x)-tu(y))(tu(x)-tu(y))K(x,y)\,dx\,dy \\&-  t^{\mathfrak{m}_2- \mathfrak{m}_1}\int_{\Omega}\mathfrak{b}(x)|u|^{\mathfrak{m}_{2}}\,dx.
\end{split}
\end{equation*}

\end{proof}
\begin{lemma}\label{lem4.2}
\begin{itemize} \item[]
\item[$(a)$] Assume that $\displaystyle{\int_{\Omega}\mathfrak{b}(x)|u|^{\mathfrak{m}_{2}}\,dx \leqslant  0}$.
 Then we obtain $$\mathsf{M}_{u}(0):=\displaystyle{\lim_{t\to 0^{+}}\mathsf{M}_{u}(t)=0}, \displaystyle{\mathsf{M}_{u}(\infty):=\lim_{t\to +\infty}\mathsf{M}_{u}(t)= +\infty}, \mbox{ and } \mathsf{M}'_{u}(t)>0 \mbox{ for all } t > 0;$$
\item[$(b)$] Assume that $\displaystyle{\int_{\Omega}\mathfrak{b}(x)|u|^{\mathfrak{m}_{2}}\,dx > 0}$ and $(\mathfrak{m}_2-1)(m-l)<(\mathfrak{m}_2-l)(m-\mathfrak{m}_1)$. Then there exists a unique critical point for  $\mathsf{M}_{u}$, i.e., there is a unique point $\tilde{t} > 0$ in such a way that $\mathsf{M}_{u}(\tilde{t})=0$. Furthermore, we know that $\tilde{t} > 0$ is a global maximum point for $\mathsf{M}_{u}$ and $\mathsf{M}_{u}(\infty)=- \infty$.
\end{itemize}
\end{lemma}

\begin{proof}
  $(a)$ Note that using $(a_{2})$, $(a_{3}')$-$(ii)$,$(\mathcal{K})$, and \eqref{const},  we obtain that
 \begin{equation}\label{4.0}
 \begin{split}
 \mathsf{M}_{u}'(t)\geqslant & (l-\mathfrak{m}_{1})t^{-\mathfrak{m}_{1}-1}\int_{\mathbb{R}^{N}\times \mathbb{R}^{N}}\mathcal{A}(tu(x)-tu(y))(tu(x)-tu(y))K(x,y)\,dx\,dy \\&- (\mathfrak{m}_{2}-\mathfrak{m}_{1})t^{\mathfrak{m}_{2}-\mathfrak{m}_{1}-1}\int_{\Omega}\mathfrak{b}(x)|u|^{\mathfrak{m}_{2}}\,dx. 
 \end{split}
 \end{equation}
 Once $\displaystyle{\int_{\Omega}\mathfrak{b}(x)|u|^{\mathfrak{m}_{2}}\,dx \leqslant 0}$, from \eqref{4.0}, we obtain $\mathsf{M}_{u}'(t)>0$ for all  $t>0$.
 
 Now, we shall prove that $ \mathsf{M}_{u}(0)=0$. Indeed, using $(a_2)$ and $(\mathcal{K})$, we
deduce that
\begin{equation}\label{4.1}
 \begin{split}
 \mathsf{M}_{u}(t) \geqslant & c_{\mathcal{A}}b_0 t^{p-\mathfrak{m}_{1}}\int_{\mathbb{R}^{N}\times \mathbb{R}^{N}}\frac{|u(x)-u(y)|^{p}}{|x-y|^{N+sp}}\,dx\,dy - t^{\mathfrak{m}_{2}-\mathfrak{m}_{1}}\int_{\Omega}\mathfrak{b}(x)|u|^{\mathfrak{m}_{2}}\,dx .
 \end{split}
 \end{equation}
 On the other hand, using $(a_1)$, $(a_2)$, $(a_3')$-$(i)$, and $(\mathcal{K})$, we infer that
 \begin{equation}\label{4.2}
 \begin{split}
 \mathsf{M}_{u}(t) \leqslant & t^{p-\mathfrak{m}_{1}} C_{\mathcal{A}}b_1
  \int_{\mathbb{R}^{N}\times \mathbb{R}^{N}}\frac{|u(x)-u(y)|^{p}}{|x-y|^{N+sp}}\,dx\,dy - t^{\mathfrak{m}_{2}-\mathfrak{m}_{1}}\int_{\Omega}\mathfrak{b}(x)|u|^{\mathfrak{m}_{2}}\,dx .
 \end{split}
 \end{equation}
 Using \eqref{4.1}, \eqref{4.2}, and \eqref{const}, we conclude that $\displaystyle{\lim_{t\to 0^{+}}\mathsf{M}_{u}(t)=0}$. Moreover, by \eqref{4.1} and \eqref{const} also we observe that $\mathsf{M}_{u}(\infty)=\displaystyle{\lim_{t\to +\infty}\mathsf{M}_{u}(t)= +\infty.}$ 

  \noindent $(b)$ As a first step we note that $\displaystyle{\lim_{t\to0^{+}} \mathsf{M}_{u}(t) = 0}$, $\mathsf{M}_{u}$ is increasing for $t > 0$ small enough and $\displaystyle{\lim_{t\to+\infty} \mathsf{M}_{u}(t) = -\infty}$. More specifically, for $0<t<1$ we observe that using \eqref{4.0} and the fact 
 that $\displaystyle{\int_{\Omega}\mathfrak{b}(x)|u|^{\mathfrak{m}_{2}}\,dx > 0}$   we obtain that $\mathsf{M}_{u}'(t) > 0$, i.e., $\mathsf{M}_{u}(t)$ is increasing
for $t \in(0,1)$. Moreover, from \eqref{4.0} and \eqref{4.1}, we obtain $\displaystyle{\lim_{t\to 0^{+}}\mathsf{M}_{u}(t)=0}$. Finally   using \eqref{4.2} and \eqref{const}, it follows that $\displaystyle{\lim_{t\to +\infty}\mathsf{M}_{u}(t)=-\infty.}$

   Now the main goal in this proof is to show that $\mathsf{M}_{u}$ has a unique critical point $\tilde{t}>0$. Note that $\mathsf{M}_{u}'(t) =0$ if and only if, we have
 \begin{equation*}
 \begin{split}
 (\mathfrak{m}_{2}-\mathfrak{m}_{1})\int_{\Omega}\mathfrak{b}(x)|u|^{\mathfrak{m}_{2}}\,dx = &t^{-\mathfrak{m}_{2}}\Bigg[\int_{\mathbb{R}^{N}\times \mathbb{R}^{N}} \mathcal{A}(tu(x)-tu(y))(tu(x)-tu(y))K(x,y)\,dx\,dy \\ & -\mathfrak{m}_{1}\int_{\mathbb{R}^{N}\times \mathbb{R}^{N}} \mathcal{A}(tu(x)-tu(y))(tu(x)-tu(y))K(x,y)\,dx\,dy \\&+\int_{\mathbb{R}^{N}\times \mathbb{R}^{N}} \mathcal{A}'(tu(x)-tu(y))(tu(x)-tu(y))^{2}K(x,y)\,dx\,dy\Bigg]. 
 \end{split}
 \end{equation*}
 Define the auxiliary function $\xi_{u}:\mathbb{R} \to \mathbb{R} $ given by
 \begin{equation*}
 \begin{split}
\xi_{u}(t) = &(1-\mathfrak{m}_{1})t^{-\mathfrak{m}_{2}}\int_{\mathbb{R}^{N}\times \mathbb{R}^{N}} \mathcal{A}(tu(x)-tu(y))(tu(x)-tu(y))K(x,y)\,dx\,dy \\&+t^{-\mathfrak{m}_{2}}\int_{\mathbb{R}^{N}\times \mathbb{R}^{N}} \mathcal{A}'(tu(x)-tu(y))(tu(x)-tu(y))^{2}K(x,y)\,dx\,dy. 
 \end{split}
 \end{equation*}
 Note that using $(a_2)$, $(a_3')$-$(ii)$,  $(\mathcal{K})$, and \eqref{const}, we infer that
  \begin{equation}\label{4.7}
 \begin{split}
\xi_{u}(t) \geqslant  &(l-\mathfrak{m}_{1})t^{p-\mathfrak{m}_{2}}c_{\mathcal{A}}b_{0} \|u\|^{p}_{W^{s,p}_{0}(\Omega)}.
 \end{split}
 \end{equation}
 Then using \eqref{4.7} and \eqref{const} for $0<t<1$, we obtain $\displaystyle{\lim_{t\to 0^{+}}\xi_{u}(t)= +\infty} $. Moreover, $\displaystyle{\lim_{t\to +\infty}\xi_{u}(t)}= 0$ and $\xi_{u}$ is a decreasing function. Indeed, using $(a_1)$, $(a_2)$, $(a_3')$-$(i)$,  $(a_3')$-$(ii)$, $(\mathcal{K})$, and \eqref{const}, we have that
 \\
 \begin{equation}\label{4.8}
 \begin{split}
\xi_{u}(t) \leqslant  &(m-\mathfrak{m}_{1})t^{p-\mathfrak{m}_{2}}C_{\mathcal{A}}b_{1}\|u\|^{p}_{W^{s,p}_{0}(\Omega)}.
 \end{split}
 \end{equation}
 Therefore, for all $t>1$  using  \eqref{4.7} and \eqref{4.8}, we obtain that $\displaystyle{\lim_{t\to +\infty}\xi_{u}(t)=0}$.
 
 Now using    $(a_2)$, $(a'_3)$, $(\mathcal{K})$, and \eqref{const}, we infer that
 \begin{equation*}\label{4.44}
 \begin{split}
 \xi'_{u}(t) \leqslant & ( \mathfrak{m}_{1} \mathfrak{m}_{2}- \mathfrak{m}_{1}l- \mathfrak{m}_{2}l+ml )t^{-\mathfrak{m}_{2}-1}\int_{\mathbb{R}^{N}\times \mathbb{R}^{N}} \mathcal{A}(tu(x)-tu(y))(tu(x)-tu(y))K(x,y)\,dx\,dy \\ &+ (l-m)t^{-\mathfrak{m}_{2}-1}\int_{\mathbb{R}^{N}\times \mathbb{R}^{N}} \mathcal{A}(tu(x)-tu(y))(tu(x)-tu(y))K(x,y)\,dx\,dy \\<&0.
 \end{split}
 \end{equation*}
Therefore,  $ \xi_{u}$ is decreasing function proving that $ \mathsf{M}_{u} $ has a unique critical
point which is a maximum critical point for $ \mathsf{M}_{u} $.   
\end{proof}
\begin{lemma}\label{lem4.3}
Let $ u \in W^{s,p}_{0}(\Omega)\setminus\{0\}$ be a fixed function. Then we shall consider the following assertions:
\begin{itemize}
\item[$(a)$] Assume that $\displaystyle{\int_{\Omega}\mathfrak{b}(x)|u|^{\mathfrak{m}_{2}}\,dx \leqslant 0}$.  Then $\wp'_{u}(t)\neq 0 $ for all $t>0$ and $\lambda >0$ whenever  $\displaystyle{\int_{\Omega}\mathfrak{a}(x)|u|^{\mathfrak{m}_{1}}\,dx \leqslant 0}$. Moreover, there exists a unique $t_{1}=t_{1}(u, \lambda)$ such
that $ \wp'_{u}(t_{1})=0 $ and $t_{1}u \in \mathcal{N}^{+}_{\lambda,1}$ whenever $\displaystyle{\int_{\Omega}\mathfrak{a}(x)|u|^{\mathfrak{m}_{1}}\,dx > 0}$;
\item[$(b)$] Assume that $\displaystyle{\int_{\Omega}\mathfrak{b}(x)|u|^{\mathfrak{m}_{2}}\,dx > 0}$.  Then exists a unique $t_{1}=t_{1}(u, \lambda)> \tilde{t}$ such that $\wp'_{u}(t_{1})=0 $ and $t_{1}u \in \mathcal{N}^{-}_{\lambda,1}$ whenever $\displaystyle{\int_{\Omega}\mathfrak{a}(x)|u|^{\mathfrak{m}_{1}}\,dx \leqslant 0}$;
\item[$(c)$] For each $ \lambda > 0$ small enough there exists unique $0<t_{1}=t_{1}(u, \lambda)<\tilde{t}<t_{2}=t_{2}(u, \lambda)$ such that $\wp'_{u}(t_{1})= \wp'_{u}(t_{2})=0 $, $t_{1}u \in \mathcal{N}^{+}_{\lambda,1}$, and $t_{2}u \in \mathcal{N}^{+}_{\lambda,1}$ whenever $\displaystyle{\int_{\Omega}\mathfrak{a}(x)|u|^{\mathfrak{m}_{1}}\,dx > 0}$, $\displaystyle{\int_{\Omega}\mathfrak{b}(x)|u|^{\mathfrak{m}_{2}}\,dx > 0}$, and $(\mathfrak{m}_2-1)(m-l)<(\mathfrak{m}_2-l)(m-\mathfrak{m}_1)$. 
\end{itemize}
\end{lemma}
\begin{remark}
We observe that for $\displaystyle{\int_{\Omega}\mathfrak{a}(x)|u|^{\mathfrak{m}_{1}}\,dx \leqslant 0}$ and  $\displaystyle{\int_{\Omega}\mathfrak{b}(x)|u|^{\mathfrak{m}_{2}}\,dx \leqslant 0}$,   $\wp_{u}$ has a graph as the Figure $2$(a). For  $\displaystyle{\int_{\Omega}\mathfrak{a}(x)|u|^{\mathfrak{m}_{1}}\,dx > 0}$ and  $\displaystyle{\int_{\Omega}\mathfrak{b}(x)|u|^{\mathfrak{m}_{2}}\,dx \leqslant 0}$,  $\wp_{u}$ has a graph as the Figure $2$(b).  For $\displaystyle{\int_{\Omega}\mathfrak{a}(x)|u|^{\mathfrak{m}_{1}}\,dx \leqslant 0}$ and  $\displaystyle{\int_{\Omega}\mathfrak{b}(x)|u|^{\mathfrak{m}_{2}}\,dx > 0}$,  $\wp_{u}$ has a graph as the Figure $2$(c). For $\displaystyle{\int_{\Omega}\mathfrak{a}(x)|u|^{\mathfrak{m}_{1}}\,dx > 0}$ and  $\displaystyle{\int_{\Omega}\mathfrak{b}(x)|u|^{\mathfrak{m}_{2}}\,dx > 0}$,  $\wp_{u}$ has a graph as the Figure $2$(d).
\end{remark}
\begin{figure}[h]\label{figurass}
\begin{center}
\includegraphics[scale=0.25]{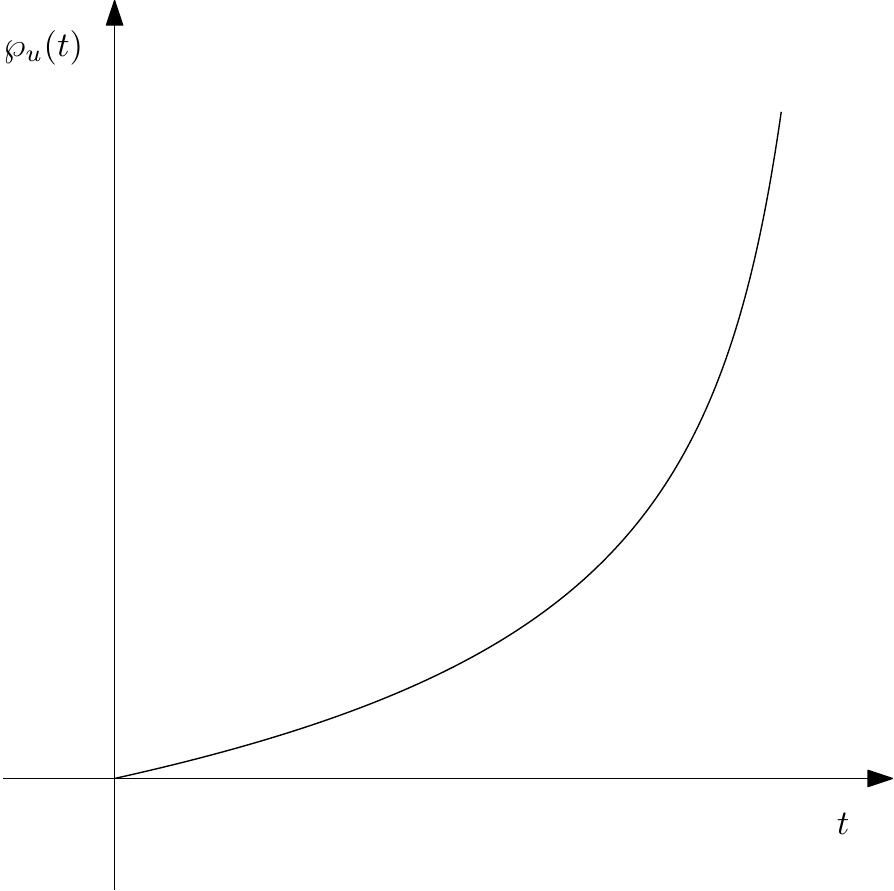}\,\,\,\,\,\,\,\,\,\,\,\,\,\,
\includegraphics[scale=0.4]{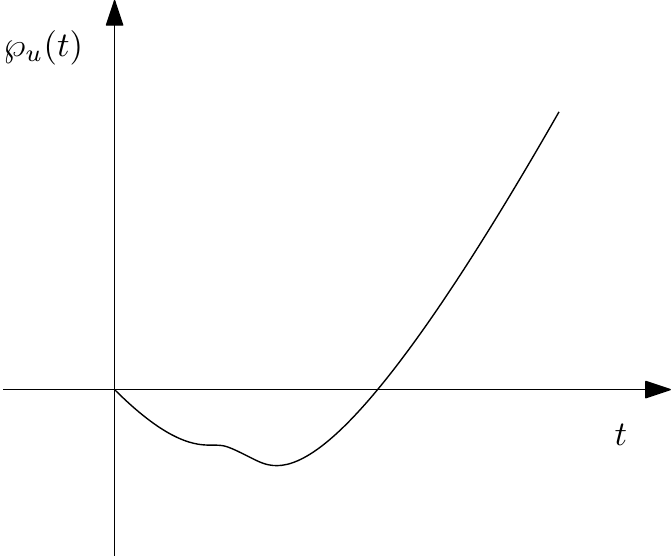} \,\,\,\,\,\,\,\,\,\,\,\,\,\,
\includegraphics[scale=0.25]{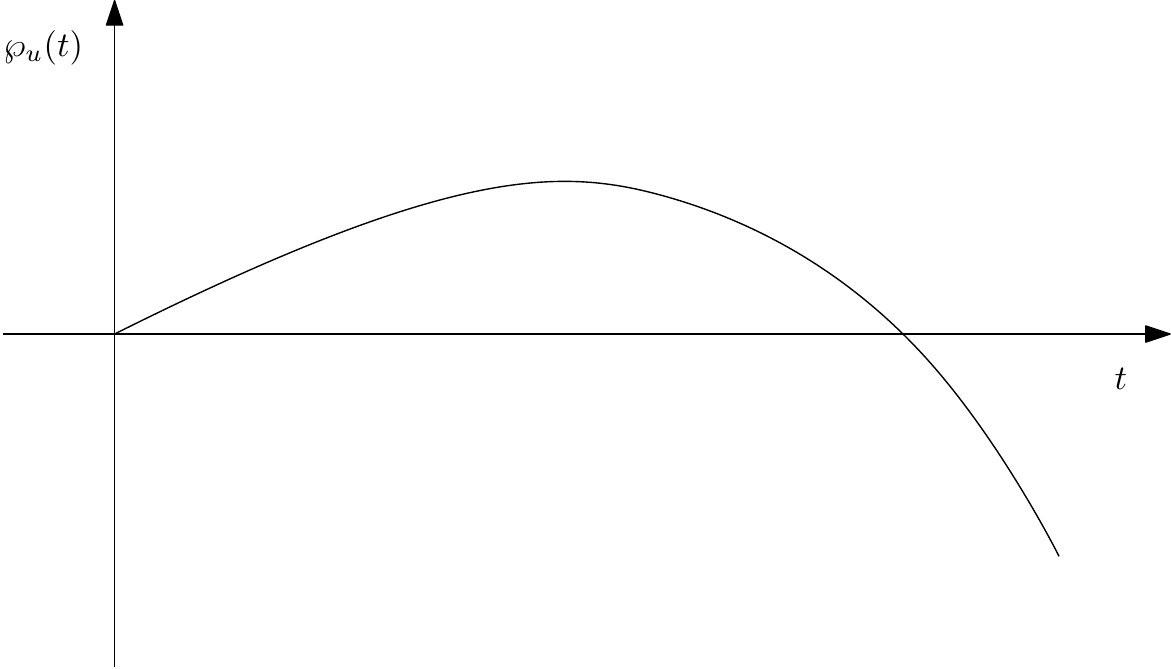}\,\,\,\,\,\,\,\,\,\,\,\,\,\,
\includegraphics[scale=0.25]{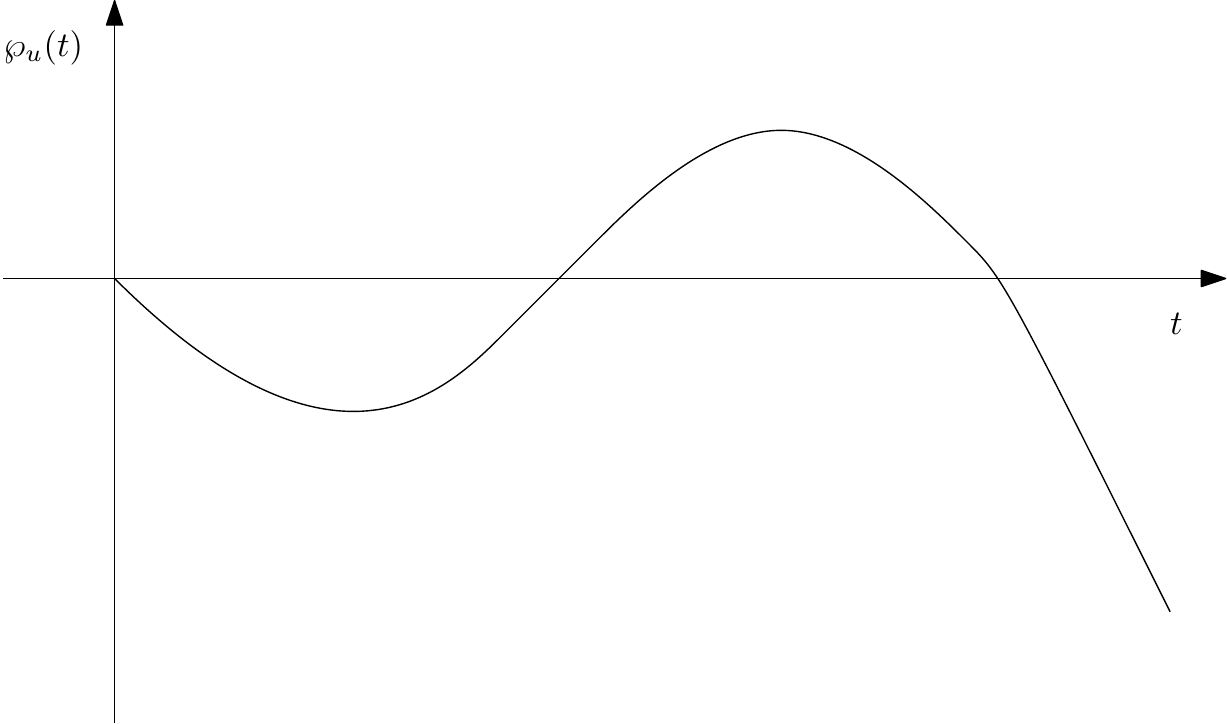}
\end{center}
\caption{Sketches of $\wp_{u}$}
\end{figure}

\begin{proof}
$(a)$ We shall consider the proof for the case $\displaystyle{\int_{\Omega}\mathfrak{b}(x)|u|^{\mathfrak{m}_{2}}\,dx \leqslant 0}$ and  $\displaystyle{\int_{\Omega}\mathfrak{a}(x)|u|^{\mathfrak{m}_{1}}\,dx \leqslant 0}$. Using the Lemma \ref{lem4.2}$-(a)$ it follows that
\begin{equation}\label{4.100}
\mathsf{M}_{u}(0)=0, \mathsf{M}_{u}(\infty):=\lim_{t\to +\infty}\mathsf{M}_{u}(t)= +\infty, \mbox{ and }\mathsf{M}'_{u}(t)>0 \mbox{ for all } t > 0.
\end{equation}
Thus, we achieve that $\mathsf{M}_{u}(t)\neq \lambda \displaystyle{\int_{\Omega}\mathfrak{a}(x)|u|^{\mathfrak{m}_{1}}\,dx \mbox{ for all } t > 0 \mbox{ and } \lambda >0}$. Then by Lemma \ref{lem4.1} we conclude that  $tu \notin \mathcal{N}_{\lambda,1}$ for all $t>0$. In particular, $\wp'_{u}(t)\neq 0$ for each $t>0$, that is, from \eqref{3.4}, $\wp'_{u}(t)> 0$ for each $t>0$.  
\noindent Now we shall consider  the case  $\displaystyle{\int_{\Omega}\mathfrak{b}(x)|u|^{\mathfrak{m}_{2}}\,dx \leqslant 0}$ and  $\displaystyle{\int_{\Omega}\mathfrak{a}(x)|u|^{\mathfrak{m}_{1}}\,dx > 0}$. From Lemma \ref{lem4.2}$-(a)$ is valid \eqref{4.100}. In particular, the equation $\displaystyle{\mathsf{M}_{u}(t)= \lambda \int_{\Omega}\mathfrak{a}(x)|u|^{\mathfrak{m}_{1}}}\, dx$ admits exactly one solution $t_{1}=t_{1}(u, \lambda)> 0$. Consequently by Lemma \ref{lem4.1}, we conclude  that $t_{1}u \in \mathcal{N}_{\lambda,1}$. Therefore, $ \wp'_{u}(t_1) = 0$.

\noindent  Moreover, since that 
$
\mathsf{M}_{u}(t)= t^{1-\mathfrak{m}_{1}}\wp'_{u}(t)+\lambda  \displaystyle{\int_{\Omega}\mathfrak{a}(x)|u|^{\mathfrak{m}_{1}}\,dx},$ 
we have that $ \displaystyle{
\mathsf{M}'_{u}(t)= t^{1-\mathfrak{m}_{1}}\wp''_{u}(t)+ (1-\mathfrak{m}_{1})t^{-\mathfrak{m}_{1}}\wp'_{u}(t).}$
Thus we conclude that $t_{1}^{1-\mathfrak{m}_{1}}\wp''_{u}(t_{1}) =\mathsf{M}'_{u}(t_1)>0$ and consequently  $t_1u \in \mathcal{N}^{+}_{\lambda,1}$.

\noindent $(b)$ Now consider the case $\displaystyle{\int_{\Omega}\mathfrak{b}(x)|u|^{\mathfrak{m}_{2}}\,dx > 0}$ and  $\displaystyle{\int_{\Omega}\mathfrak{a}(x)|u|^{\mathfrak{m}_{1}}\,dx \leqslant 0}$. Using Lemma \ref{lem4.2}$-(b)$ the function $\mathsf{M}_{u}$ admits a unique critical  point $\tilde{t} > 0$, i.e.,  $\mathsf{M}'_{u}(t)=0$, $ t > 0$ if and only if $t = \tilde{t}$. Moreover, $\tilde{t}$ a global maximum point for $\mathsf{M}_{u}$ such that $\mathsf{M}_{u}(\tilde{t})
>0 $ and $\mathsf{M}_{u}(\infty)=-\infty$. Hence, once $\mathsf{M}_{u}$ is a  decreasing function in $t \in(\tilde{t}, +\infty)$, there exists a unique $t_1>\tilde{t}$ such that $\displaystyle{\mathsf{M}_{u}(t_1)=\lambda \int_{\Omega}\mathfrak{a}(x)|u|^{\mathfrak{m}_{1}}\,dx \mbox{  and  } \mathsf{M}'_{u}(t_1)<0.}$
 Therefore, $t^{1-\mathfrak{m}_1}\wp''_{u}(t_1)=\mathsf{M}'_{u}(t_1)<0$ and we conclude that $t_{1}u \in \mathcal{N}^{-}_{\lambda,1}$.

 \noindent $(c)$ From $\displaystyle{\int_{\Omega}\mathfrak{a}(x)|u|^{\mathfrak{m}_{1}}\,dx > 0}$ we can consider
$ \lambda >0$ small enough in such a way that $\displaystyle{
\mathsf{M}_{u}(\tilde{t})>\lambda\int_{\Omega}\mathfrak{a}(x)|u|^{\mathfrak{m}_{1}}\,dx.}$ Moreover, due to the proof of Lemma \ref{lem4.2}$-(b)$, $\mathsf{M}_{u}$ is increasing in $(0, \tilde{t})$ and decreasing in $( \tilde{t}, +\infty)$. In this sense, there is exactly two points $0<t_{1}=t_{1}(u, \lambda)<\tilde{t}<t_{2}=t_{2}(u, \lambda)$ such that $\displaystyle{\mathsf{M}_{u}(t_{i})=\lambda\int_{\Omega}\mathfrak{a}(x)|u|^{\mathfrak{m}_{1}}\,dx}$ for $i=1,2$. Additionally, we have that $\mathsf{M}'_{u}(t_{1})>0$ and $\mathsf{M}'_{u}(t_{2})<0$. So with the same argument as before we get
$t^{1-\mathfrak{m}_1}\wp''_{u}(t_2)= \mathsf{M}'_{u}(t_2)<0$ and $t^{1-\mathfrak{m}_1}\wp''_{u}(t_1)=\mathsf{M}'_{u}(t_1)>0$.   Then $t_{2}u \in \mathcal{N}^{-}_{\lambda,1}$ and $t_{1}u \in \mathcal{N}^{+}_{\lambda,1}$.
\end{proof}

\begin{lemma}\label{lem4.4}
There exists $\tilde{\lambda} > 0$ such that $\wp_{u}$ takes positive values for all $u \in W^{s,p}_{0}(\Omega)\setminus \{0\}$ whenever $0<\lambda < \tilde{\lambda}$.
\end{lemma}
\begin{proof}
We shall split the proof into two cases.

\noindent\textbf{\textit{Case 1:}}  $\displaystyle{\int_{\Omega}\mathfrak{b}(x)|u|^{\mathfrak{m}_{2}}\,dx \leqslant 0}.$

\noindent From $(a_2)$, $(a_3')$-$(i)$, and $(\mathcal{K})$, we get
\begin{equation*}
\begin{split}
\wp_{u}(t) \geqslant &  \frac{c_{\mathcal{A}}b_{0}t^{p}}{p}\int_{\mathbb{R}^{N}\times \mathbb{R}^{N}}\frac{|u(x)-u(y)|^{p}}{|x-y|^{N+sp}}\,dx\,dy - \lambda \frac{t^{\mathfrak{m}_{1}}}{\mathfrak{m}_1}\int_{\Omega}\mathfrak{a}(x)|u|^{\mathfrak{m}_1}\,dx  - \frac{t^{\mathfrak{m}_{2}}}{\mathfrak{m}_{2}}\int_{\Omega}\mathfrak{b}(x)|u|^{\mathfrak{m}_{2}}\,dx.
\end{split}
\end{equation*} 
Since $\mathfrak{m_1}<p<\mathfrak{m_2}$ and  $\displaystyle{\int_{\Omega}\mathfrak{b}(x)|u|^{\mathfrak{m}_{2}}\,dx \leqslant 0}$, it follows that
$\displaystyle{\lim_{t \to +\infty}\wp_{u}(t)=+\infty}$. In particular, there is $\overline{t}>0$ such that $\wp_{u}(t)>0$ for each $t> \overline{t}$.

\noindent\textbf{\textit{Case 2:}}  $\displaystyle{\int_{\Omega}\mathfrak{b}(x)|u|^{\mathfrak{m}_{2}}\,dx > 0}.$

\noindent First, let $\mathsf{h}_{u}:\mathbb{R}\to \mathbb{R}$   defined by 
$$\mathsf{h}_{u}(t):= \int_{\mathbb{R}^{N}\times \mathbb{R}^{N}}  \mathscr{A}(tu(x)-tu(y))K(x,y)\,dx\,dy -\frac{t^{\mathfrak{m}_{2}}}{\mathfrak{m}_{2}} \int_{\Omega}\mathfrak{b}(x)|u|^{\mathfrak{m}_{2}}\,dx \mbox{ for all }  t \in \mathbb{R}$$
an function of class $C^{1}(\mathbb{R}, \mathbb{R})$. Note that
$$\mathsf{h}'_{u}(t)= \int_{\mathbb{R}^{N}\times \mathbb{R}^{N}}  \mathcal{A}(tu(x)-tu(y))(u(x)-u(y))K(x,y)\,dx\,dy -t^{\mathfrak{m}_{2}-1} \int_{\Omega}\mathfrak{b}(x)|u|^{\mathfrak{m}_{2}}\,dx.$$
Then $\mathsf{h}_{u}$ admits a critical point $t > 0$, this is, $\mathsf{h}'_{u}(t)=0$
for some point $t > 0$ which is a local maximum point for $\mathsf{h}_{u}$ (see Lemma \ref{lem4.3} items $(b)$ and $(c)$). Moreover, for all $t > 0$, we observe that $\mathsf{h}'_{u}(t)=0$ if and only if 
\begin{equation}\label{4.11}
\frac{t}{\mathfrak{m}_{2}}\int_{\mathbb{R}^{N}\times \mathbb{R}^{N}}  \mathcal{A}(tu(x)-tu(y))(u(x)-u(y))K(x,y)\,dx\,dy= \frac{t^{\mathfrak{m}_{2}}}{\mathfrak{m}_{2}}\int_{\Omega}\mathfrak{b}(x)|u|^{\mathfrak{m}_{2}}\,dx.
\end{equation}
Thus using $(a_2)$, $(a'_3)$-$(i)$,  $(\mathcal{K})$, \eqref{const}, and \eqref{4.11}, we obtain
\begin{equation}\label{4.12}
\begin{split}
\mathsf{h}_{u}(t) \geqslant   \Bigg( \frac{1}{p} - \frac{1}{\mathfrak{m}_{2}} \Bigg)c_{\mathcal{A}}b_{0} \|tu\|^{p}_{W^{s,p}_{0}(\Omega)} >  0.
\end{split}
\end{equation}
 Now,  using $(a_{2})$, $(\mathcal{K})$,   the continuous embedding  $W^{s,p}_{0}(\Omega)\hookrightarrow L^{\mathfrak{m}_{2}}(\Omega)$,  \eqref{3.3} \eqref{3.31}, and  \eqref{4.11}, it follows that
\begin{equation}\label{4.13}
\|tu\|_{W^{s,p}_{0}(\Omega)}\geqslant \Bigg( \frac{c_{A}b_{0}}{C^{\mathfrak{m}_{2}}_{\mathfrak{m}_{2}}\|\mathfrak{b}^{+}\|_{\infty}}\Bigg)^{\frac{1}{\mathfrak{m}_{2}
-p}}.
\end{equation}
Therefore, by \eqref{4.12} and \eqref{4.13}, we obtain that
\begin{equation*}
\begin{split}
\mathsf{h}_{u}(t) \geqslant  \Bigg( \frac{1}{p} - \frac{1}{\mathfrak{m}_{2}} \Bigg)c_{\mathcal{A}}b_{0}\Bigg( \frac{c_{A}b_{0}}{C^{\mathfrak{m}_{2}}_{\mathfrak{m}_{2}}\|\mathfrak{b}^{+}\|_{\infty}}\Bigg)^{\frac{p}{\mathfrak{m}_{2}
-p}}:=\delta > 0 \mbox{ for all } u \in W^{s,p}_{0}(\Omega)\setminus \{0\} \mbox{  and }  t \in \mathbb{R}.
\end{split}
\end{equation*}

Now, by the continuous embedding  $W^{s,p}_{0}(\Omega)\hookrightarrow L^{\mathfrak{m}_{1}}(\Omega)$ using  \eqref{3.3} and \eqref{4.12}, we infer that
\begin{equation*}
\begin{split}
\int_{\Omega}\frac{\mathfrak{a}(x)|tu|^{\mathfrak{m}_{1}}}{\mathfrak{m}_{1}}\,dx \leqslant & \frac{\|\mathfrak{a}^{+}\|_{\infty}C^{\mathfrak{m}_{1}}_{\mathfrak{m}_{1}}\|tu\|^{\mathfrak{m}_{1}}_{W^{s,p}_{0}(\Omega)}}{\mathfrak{m}_{1}}\\ \leqslant & \frac{\|\mathfrak{a}^{+}\|_{\infty}C^{\mathfrak{m}_{1}}_{\mathfrak{m}_{1}}}{\mathfrak{m}_{1}\big(\big( \frac{1}{p} - \frac{1}{\mathfrak{m}_{2}} \big)c_{\mathcal{A}}b_{0}\big)^{\frac{\mathfrak{m}_{1}}{p}}}\big(\mathsf{h}_{u}(t)\big) ^{\frac{\mathfrak{m}_{1}}{p}}:= \tilde{D}\big(\mathsf{h}_{u}(t)\big) ^{\frac{\mathfrak{m}_{1}}{p}},
\end{split}
\end{equation*}
where $\tilde{D}=\frac{\|\mathfrak{a}^{+}\|_{\infty}C^{\mathfrak{m}_{1}}_{\mathfrak{m}_{1}}}{\mathfrak{m}_{1}\big(\big( \frac{1}{p} - \frac{1}{\mathfrak{m}_{2}} \big)c_{\mathcal{A}}b_{0}\big)^{\frac{\mathfrak{m}_{1}}{p}}}>0.$

\noindent Therefore,
\begin{equation*}
\begin{split}
\wp_{u}(t)=  \mathsf{h}_{u}(t) -\frac{\lambda t ^{\mathfrak{m}_{1}}}{\mathfrak{m}_{1}}\int_{\Omega}\mathfrak{a}(x)|u|^{\mathfrak{m}_{1}}\,dx \geqslant  \big(\mathsf{h}_{u}(t)\big) ^{\frac{\mathfrak{m}_{1}}{p}}\big[\big(\mathsf{h}_{u}(t)\big) ^{1-\frac{\mathfrak{m}_{1}}{p}} -\lambda\tilde{D}\big].
\end{split}
\end{equation*}
Since $\mathsf{h}_{u}(t) > \delta$, taking 
$\tilde{\lambda}= \frac{\delta^{{1-\frac{\mathfrak{m}_{1}}{p}}}}{2 \tilde{D}}> \lambda$
we obtain that $\wp_{u}(t)> \frac{1}{2}\delta>0$. This  conclude the proof of Lemma.
\end{proof}

\begin{lemma}\label{lem4.5}
There exist $\tilde{\delta}>0$ and $ \tilde{\lambda}>0$ in such a way that $\mathcal{J}_{\lambda,1}(u)\geqslant\tilde{\delta}$ for all $u \in \mathcal{N}_{\lambda,1}^{-} $ where $0< \lambda <  \tilde{\lambda}$.
\end{lemma}
\begin{proof}
Fix $u \in \mathcal{N}_{\lambda,1}^{-}$, consequently  $\mathcal{J}_{\lambda,1}$ admits a global maximum in $t = 1$ and $\displaystyle{\int_{\Omega}\mathfrak{b}(x)|u|^{\mathfrak{m}_{2}}\,dx > 0}.$  Indeed, for $u \in \mathcal{N}_{\lambda,1}^{-}$  the   fibering map has a behavior described in Lemma \ref{lem4.3}. On the other hand, using Lemma \ref{lem4.4} there is $t_{0}>0$ such that $\mathsf{h}_{u}(t_{0})>\mathsf{h}_{u}(t)$  for each $t>0$.
Hence
\begin{equation*}
\mathcal{J}_{\lambda,1}(u)=\wp_{u}(1)\geqslant \wp_{u}(t)\geqslant\mathsf{h}_{u}(t_{0}) ^{\frac{\mathfrak{m}_{1}}{p}}(\mathsf{h}_{u}(t_{0}) ^{1-\frac{\mathfrak{m}_{1}}{p}} -\lambda\tilde{D})\geqslant \delta ^{\frac{\mathfrak{m}_{1}}{p}}(\delta ^{1-\frac{\mathfrak{m}_{1}}{p}} -\lambda\tilde{D})
\end{equation*}
 for $\delta > 0$  obtained in the  Lemma \ref{lem4.4}. Thus, taking $\tilde{\delta}=\delta ^{\frac{\mathfrak{m}_{1}}{p}}(\delta ^{1-\frac{\mathfrak{m}_{1}}{p}} -\lambda\tilde{D})$, the proof for this
lemma is completed by choosing $0<\lambda<\tilde{\lambda}$ small enough.  
\end{proof}

\subsection{Proof of Theorem \ref{con-convex 1}}
\begin{proof} [Proof of Theorem \ref{con-convex 1}]
From Proposition \ref{propo3.1} the functional $\mathcal{J}_{\lambda,1}$  is bounded below on $\mathcal{N}_{\lambda,1}$ and    so on $\mathcal{N}_{\lambda,1}^{+}$, then there  exists a minimizing
sequence $(u_{k})_{k\in \mathbb{N}}\subset \mathcal{N}_{\lambda,1}^{+} $ such that
\begin{equation}\label{quadr}
\lim_{k \to +\infty} \mathcal{J}_{\lambda, 1}(u_{k})= \inf_{u \in \mathcal{N}_{\lambda,1}^{+}} \mathcal{J}_{\lambda, 1}(u):=  \mathbb{J}^{+}.
\end{equation}
Again by  Proposition \ref{propo3.1} the functional $\mathcal{J}_{\lambda,1}$ is coercive on $\mathcal{N}_{\lambda,1}^{+}$, consequently  $(u_{k})_{k\in \mathbb{N}}$ is a bounded sequence in $W^{s,p}_{0}(\Omega)$. Since $W^{s,p}_{0}(\Omega)$ is a reflexive Banach space, there exists $u_0 \in W^{s,p}_{0}(\Omega)$ such that, up to a subsequence $u_{k}\rightharpoonup u_{0}$  in $W^{s,p}_{0}(\Omega)$, 
 \begin{equation*}
 u_{k}(x)\to u_{0}(x) \mbox{ a.e. } x \in \Omega,\hspace{0.2cm} u_{k}\to u_{0} \mbox{ in } L^{\mathfrak{m}_{1}}(\Omega) \hspace{0.2cm}\mbox{ and } \hspace{0.2cm} u_{k}\to u_{0} \mbox{ in } L^{\mathfrak{m}_{2}}(\Omega) \mbox{ as } k \to +\infty. 
 \end{equation*}
Consequently as $\mathfrak{a}, \mathfrak{b} \in L^{\infty}(\Omega)$, we conclude that
\begin{equation}\label{31.0}
\begin{split}
\int_{\Omega}\mathfrak{a}(x)|u_{k}|^{\mathfrak{m}_{1}}\,dx
\to \int_{\Omega}\mathfrak{a}(x)|u_{0}|^{\mathfrak{m}_{1}}\,dx,\,\,\,\,\,\int_{\Omega}\mathfrak{b}(x)|u_{k}|^{\mathfrak{m}_{2}}\,dx \to \int_{\Omega}\mathfrak{b}(x)|u_{0}|^{\mathfrak{m}_{2}}\,dx \mbox{ as } k \to +\infty.
\end{split}
\end{equation}
 \noindent Now suppose by contradiction that  $u_{k}\nrightarrow u_{0}$ in  $W^{s,p}_{0}(\Omega)\mbox{ as } k \to +\infty $. Hence from Lemma \ref{ll1}
\begin{equation} \label{31}
\begin{split}
\int_{\mathbb{R}^{N}\times \mathbb{R}^{N}} \mathcal{A}(u_{0}(x)-u_{0}(y))(u_{0}(x)-u_{0}(y))K(x,y)\,dx\,dy < \\ \liminf_{k\to +\infty}\int_{\mathbb{R}^{N}\times \mathbb{R}^{N}} \mathcal{A}(u_{k}(x)-u_{k}(y))(u_{k}(x)-u_{k}(y))K(x,y)\,dx\,dy.
\end{split}
\end{equation} 
Now, fix $t_{0}>0$ such that $t_{0}u_{0}\in \mathcal{N}^{+}_{\lambda,1}$ and using \eqref{31.0} and   \eqref{31}, we get
\begin{equation*}
\begin{split}
0=\wp'_{u_{0}}(t_{0}) <   \liminf_{k\to +\infty} \wp'_{u_{k}}(t_{0}).
\end{split}
\end{equation*}
Consequently $\wp'_{u_{k}}(t_{0})>0$ for all $k \geqslant k_{0}$ where $k_{0}$ is big enough. 
\\
 Now, note that for all $u_k \in \mathcal{N}^{+}_{\lambda,1}$, using Lemma \ref{lem4.3}, we obtain that $\wp'_{u_k}(t)<0$ for all $t \in (0,1)$ and $\wp'_{u_k}(1)=0$. Indeed, as $u_{k} \in \mathcal{N}^{+}_{\lambda,1}$ then $\wp''_{u_{k}}(1)>0$ and  $\wp'_{u_{k}}(1)=0$ for all $k \in \mathbb{N}$. Moreover using $(a'_{3})$-$(i)$ and \eqref{const}, we infer that
\begin{equation}\label{4.51}
\begin{split}
  \wp''_{u_{k}}(1) +& (\mathfrak{m}_{2}-m )\int_{\mathbb{R}^{N}\times \mathbb{R}^{N}} \mathcal{A}(u_{k}(x)-u_{k}(y))(u_{k}(x)-u_{k}(y))K(x,y)\,dx\,dy   \\ \leqslant&  \lambda(\mathfrak{m}_{2}-\mathfrak{m}_{1}) \int_{\Omega}\mathfrak{a}(x)|tu_{k}|^{\mathfrak{m}_1}\,dx  \mbox{ for all } k \in \mathbb{N}.
\end{split}
\end{equation}
 By \eqref{4.51}, \eqref{const}, and the fact that $\wp''_{u_{k}}(1)>0$, we conclude that  $\displaystyle{\int_{\Omega}\mathfrak{a}(x)|tu_{k}|^{\mathfrak{m}_1}\,dx>0}$. Therefore, by Lemma \ref{lem4.3} items $(b)$ and $(c)$ we obtain that $\wp''_{u_{k}}(t)<0$ for $t \in(0,1)$ and $\wp'_{u_k}(1)=0$ for all $k \in \mathbb{N}$. Consequently, $t_{0}>1$. 
 \\
 On the other hand, as $t_{0}u_{0} \in \mathcal{N}^{+}_{\lambda,1}$, $\wp_{u_{0}}(t)$ is decreasing   for $t \in (0,t_{0})$. 
 Indeed, note that if exists  $u \in W^{s,p}_{0}(\Omega)$ and $\displaystyle{\int_{\Omega}\mathfrak{a}(x)|tu|^{\mathfrak{m}_1}\,dx>0}$, by Lemma \ref{lem4.3}, should exist $t_{1}(u,\lambda) \in \mathcal{N}^{+}_{\lambda,1}$ such that  $\wp_{u}(t_{1}(u,\lambda))= \mathcal{J}_{\lambda,1}(t_{1}(u,\lambda)u)<0$. Therefore,
 \begin{equation}\label{estrela}
 \inf_{u \in \mathcal{N}^{+}_{\lambda,1}}\mathcal{J}_{\lambda,1}(u)<0.
 \end{equation}
Now observe that
\begin{equation}\label{4.520}
\begin{split}
\lambda\Bigg(\frac{1}{\mathfrak{m}_2} - \frac{1}{\mathfrak{m}_1} \Bigg)\int_{\Omega}\mathfrak{a}(x)|u_{k}|^{\mathfrak{m}_1}\,dx = & \int_{\mathbb{R}^{N}\times \mathbb{R}^{N}} \frac{\mathcal{A}(u_{k}(x)-u_{k}(y))(u_{k}(x)-u_{k}(y))K(x,y)}{\mathfrak{m}_{2}}\,dx\,dy \\ & -  \int_{\mathbb{R}^{N}\times \mathbb{R}^{N}} \mathscr{A}(u_{k}(x)-u_{k}(y))K(x,y)\,dx\,dy + \mathcal{J}_{\lambda,1}(u_k)
\end{split}
\end{equation}
for all $k \in \mathbb{N}$.
\\
Then taking $\displaystyle{\liminf}$ in \eqref{4.520} and using $(a_{3})$-$(i)$, \eqref{quadr}, \eqref{31.0}, \eqref{31}, and \eqref{estrela},  we obtain that
\begin{equation*}\label{4.52}
\begin{split}
\lambda\Bigg(\frac{1}{\mathfrak{m}_1} - \frac{1}{\mathfrak{m}_2} \Bigg)\int_{\Omega}\mathfrak{a}(x)|u_{0}|^{\mathfrak{m}_1}\,dx \geqslant & \int_{\mathbb{R}^{N}\times \mathbb{R}^{N}}\frac{ \mathcal{A}(u_{0}(x)-u_{0}(y))(u_{0}(x)-u_{0}(y))K(x,y)}{p}\,dx\,dy \\
&- \int_{\mathbb{R}^{N}\times \mathbb{R}^{N}} \frac{\mathcal{A}(u_{0}(x)-u_{0}(y))(u_{0}(x)-u_{0}(y))K(x,y)}{\mathfrak{m}_{2}}\,dx\,dy \\ &- \inf_{u \in \mathcal{N}^{+}_{\lambda,1}}\mathcal{J}_{\lambda,1}(u).
\end{split}
\end{equation*}
Thus  $\displaystyle{\int_{\Omega}\mathfrak{a}(x)|u_{0}|^{\mathfrak{m}_1}\,dx>0}$. Consequently, as $t_{0}u_{0} \in \mathcal{N}^{+}_{\lambda,1} $,  from Lemma \ref{lem4.3} we obtain that $ \wp_{u_{0}}$ is decreasing for $t \in (0, t_{0})$.
Therefore,
\begin{equation*}
\mathcal{J}_{\lambda,1}(t_{0}u_{0})< \mathcal{J}_{\lambda,1}(u_{0}) <\liminf_{k\to +\infty}\mathcal{J}_{\lambda,1}(u_{k}):= \mathbb{J}^{+}
\end{equation*}
which contradicts \eqref{quadr}. Then $ u_{k}\to u_{0}$ in $W^{s,p}_{0}(\Omega)$ as $k \to  + \infty$. Hence, from Lemma \ref{ll1} and \eqref{31.0}, we get
\begin{equation*}
\mathcal{J}_{\lambda,1}(u_{0}) = \lim_{k \to +\infty} \mathcal{J}_{\lambda,1}(u_{k}),
\end{equation*}
that is, $u_{0}$  is a minimizer for $\mathcal{J}_{\lambda,1}$ in  $\mathcal{N}^{+}_{\lambda,1}$. Moreover using \eqref{estrela}, we get $\displaystyle{\mathcal{J}_{\lambda,1}(u_{0})=\inf_{u \in \mathcal{N}^{+}_{\lambda,1}}\mathcal{J}_{\lambda,1}(u)<0}.$
This finishes the proof.
\end{proof}
\subsection{Proof Theorem \ref{con-convex 2}}
\begin{proof} [Proof Theorem \ref{con-convex 2}]
From Lemma \ref{lem4.5} there exists $\tilde{\delta} > 0$ such that $\mathcal{J}_{\lambda,1}(u)\geqslant \tilde{\delta}$ for all $u \in \mathcal{N}^{-}_{\lambda,1} $. Thus
\begin{equation}\label{n6}
\mathbb{J}^{-}:= \inf_{u \in \mathcal{N}^{-}_{\lambda,1} }\mathcal{J}_{\lambda,1}(u)\geqslant \tilde{\delta}>0.
\end{equation}
Let a minimizer sequence $(u_{k})_{k\in \mathbb{N}} \subset \mathcal{N}^{-}_{\lambda,1} $, this is, $\displaystyle{\lim_{k\to +\infty}\mathcal{J}_{\lambda,1}(u_{k})=\mathbb{J}^{-}}$. From Proposition \ref{propo3.1}, $\mathcal{J}_{\lambda,1}$ is coercive on $\mathcal{N}_{\lambda,1} $ and also on  $\mathcal{N}^{-}_{\lambda,1} $, then  $(u_{k})_{k\in \mathbb{N}}$ is a bounded sequence in $ W^{s,p}_{0}(\Omega)$.
Since $W^{s,p}_{0}(\Omega)$ is a reflexive Banach space, there exists $\tilde{u}_{0} \in W^{s,p}_{0}(\Omega)$ such that, up to a subsequence $u_{k}\rightharpoonup \tilde{u}_{0}$  in $W^{s,p}_{0}(\Omega)$,
 \begin{equation*}
 u_{k}(x)\to \tilde{u}_{0} (x) \mbox{ a.e. } x \in \Omega,\hspace{0.2cm} u_{k}\to \tilde{u}_{0}  \mbox{ in } L^{\mathfrak{m}_{1}}(\Omega) \hspace{0.2cm}\mbox{ and } \hspace{0.2cm} u_{k}\to \tilde{u}_{0}  \mbox{ in } L^{\mathfrak{m}_{2}}(\Omega) \mbox{ as } k \to +\infty. 
 \end{equation*}
Consequently as $\mathfrak{a}, \mathfrak{b} \in L^{\infty}(\Omega)$ we conclude that is valid \eqref{31.0}.
We want to prove that $u_{k}\to \tilde{u}_{0} $ in  $W^{s,p}_{0}(\Omega)$ $\mbox{ as } k \to +\infty$ and conclude that
\begin{equation*}
\mathcal{J}_{\lambda,1}(\tilde{u}_{0} )= \lim_{k\to +\infty}\mathcal{J}_{\lambda,1}(u_{k})= \inf_{u \in \mathcal{N}^{-}_{\lambda,1} }\mathcal{J}_{\lambda,1}(u).
\end{equation*}
Now, using \eqref{not1},  $(a'_{3})$-$(i)$, $(\mathcal{K})$, and \eqref{const} note that

\begin{equation}\label{5.101}
\begin{split}
\mathcal{J}_{\lambda,1}(u_{k})  \leqslant & \Bigg( 1- \frac{l}{\mathfrak{m}_{1}} \Bigg)\int_{\mathbb{R}^{N}\times\mathbb{R}^{N}}\mathscr{A}(u_{k}(x)-u_{k}(y))K(x,y)\,dx\,dy \\ &+\Bigg(\frac{1}{\mathfrak{m}_{1}}- \frac{1}{\mathfrak{m}_{2}}\Bigg)\int_{\Omega}\mathfrak{b}(x)|u_{k}|^{\mathfrak{m}_{2}}\,dx \mbox{ for all } k \in \mathbb{N}.
\end{split}
\end{equation}
Hence, from \eqref{5.101},  $(a'_{3})$-$(i)$, $(\mathcal{K})$, and \eqref{const}, we obtain that
\begin{equation}\label{5.103}
\begin{split}
\Bigg(\frac{1}{\mathfrak{m}_{1}}- \frac{1}{\mathfrak{m}_{2}}\Bigg)\int_{\Omega}\mathfrak{b}(x)|u_{k}|^{\mathfrak{m}_{2}}\,dx  \geqslant &  \int_{\mathbb{R}^{N}\times\mathbb{R}^{N}}\frac{\mathcal{A}(u_{k}(x)-u_{k}(y))(u_{k}(x)-u_{k}(y))K(x,y)}{\mathfrak{m}_{1}p}\,dx\,dy  \\ & - \int_{\mathbb{R}^{N}\times\mathbb{R}^{N}}\frac{\mathcal{A}(u_{k}(x)-u_{k}(y))(u_{k}(x)-u_{k}(y))K(x,y)}{p}\,dx\,dy\\&+ \mathcal{J}_{\lambda,1}(u_{k}) \mbox{ for all } k \in \mathbb{N}.
\end{split}
\end{equation}
Consequently using \eqref{31.0}, \eqref{n6},  $(a_{2})$, $(\mathcal{K})$, and  \eqref{const} in \eqref{5.103},  we conclude that $\displaystyle{\int_{\Omega}\mathfrak{b}(x)|\tilde{u}_{0}|^{\mathfrak{m}_{2}}\,dx>0}$. Thus from Lemma \ref{lem4.3}, the fibering map $\wp_{\tilde{u}_{0}}$ admits a unique critical point $t_1 > 0 $ in such a way that $\wp'_{\tilde{u}_{0}}(t_{1})=0 \mbox{ and } t_{1}\tilde{u}_{0} \in \mathcal{N}^{-}_{\lambda}.$

\noindent Now we suppose  by contradiction  that $u_{k}\nrightarrow \tilde{u}_{0}$ in $W^{s,p}_{0}(\Omega)$. Using $(a_{1})$, $(a_{2})$, $(a_{3}')$, $(\mathcal{K})$, and the Brezis–Lieb Lemma (see \cite{brascolieb}), we infer that
\begin{equation}\label{lieb}
\int_{\mathbb{R}^{N}\times\mathbb{R}^{N}}\mathscr{A}(\tilde{u}_{0}(x)-\tilde{u}_{0}(y))K(x,y)\,dx\,dy < \liminf_{k\to + \infty}\int_{\mathbb{R}^{N}\times\mathbb{R}^{N}}\mathscr{A}(u_{k}(x)-u_{k}(y))K(x,y)\,dx\,dy.
\end{equation}
Since $(u_k)_{k\in \mathbb{N}}\subset \mathcal{N}^{-}_{\lambda,1} $, we obtain $\wp_{u_k}(1)\geqslant\wp_{u_k}(t)$ for all $t>0$, consequently 
\begin{equation}\label{lieb2}
\mathcal{J}_{\lambda,1}(u_{k})\geqslant \mathcal{J}_{\lambda,1}(tu_{k}) \mbox{  for all  } t>0 \mbox{ and } k \in \mathbb{N}.
\end{equation}
Thence, from \eqref{31.0}, \eqref{lieb}, and \eqref{lieb2}, we conclude that
\begin{equation*}
\begin{split}
\mathcal{J}_{\lambda,1}(t_{1}\tilde{u}_{0}) <  & \liminf_{k\to +\infty}\Bigg(\int_{\mathbb{R}^{N}\times\mathbb{R}^{N}}\mathscr{A}(u_{k}(x)-u_{k}(y))K(x,y)\,dx\,dy \\& -\frac{\lambda }{\mathfrak{m}_1}\int_{\Omega}\mathfrak{a}(x)|u_{k}|^{\mathfrak{m}_1}\,dx- \frac{1}{\mathfrak{m}_{2}}\int_{\Omega}\mathfrak{b}(x)|u_{k}|^{\mathfrak{m}_{2}}\,dx\Bigg) \\ <&  \liminf_{k\to +\infty}\mathcal{J}_{\lambda, 1}( u_k)=  \mathbb{J}^{-}. 
\end{split}
\end{equation*}
Therefore, $t_{1}\tilde{u}_{0}\in \mathcal{N}^{-}_{\lambda}$ and 
$\mathcal{J}_{\lambda,1}(t_{1}\tilde{u}_{0})< \mathbb{J}^{-}$ which is  a contradiction due the fact that $(u_k)_{k\in \mathbb{N}}$ is minimizer sequence. Hence $u_k \to \tilde{u}_{0}$ in $W_{0}^{s,p}(\Omega)$ as $k \to  + \infty$. Thus from Lemma \ref{ll1} and by equation \eqref{31.0} we can conclude that $\mathcal{J}_{\lambda,1}(u_{k}) \to \mathcal{J}_{\lambda,1}(\tilde{u}_{0})$ as $k \to  + \infty$ and $\tilde{u}_{0}$ is point  minimum of $\mathcal{J}_{\lambda, 1}$ in $\mathcal{N}^{-}_{\lambda}$, then a critical point of the functional $\mathcal{J}_{\lambda, 1}$.
\end{proof}

 \section{Proof of Theorems \ref{peso1} and \ref{peso2} }
 \subsection{Proof of Theorem \ref{peso1} }\label{teorema3}
\hfill \break 
To prove  Theorem \ref{peso1}  we will apply Ekeland’s Variational Principle, \cite[Theorem 2.4]{wil}  combined with the Lemmas that we will prove in the sequence.

\noindent First we denote by $q'(x)$ the conjugate exponent of the function $q(x)$ and put $ \alpha(x) := \displaystyle{ \frac{q(x)\mathfrak{m}_{1}(x)}{q(x)-\mathfrak{m}_{1}(x)}}$ for all $x \in \overline{\Omega}$. Since $ \displaystyle{ \underline{\mathfrak{m}}_{1}^{-}\leqslant \underline{\mathfrak{m}}_{1}^{+}    <  p^{-}\leqslant p^{+} <\frac{N}{s}< \underline{q}^{-}\leqslant \underline{q}^{+}} $, 	we have $ q'(x)\mathfrak{m}_{1}(x)< \alpha(x)$ and $\alpha(x)< p_{s}^{\star}(x)$, for all $x$ in $ \overline{\Omega}$. Thus, by Lemma \ref{2.11}, the embeddings $\mathscr{W}\hookrightarrow L^{\mathfrak{m}_{1}(\cdot)q'(\cdot)}(\Omega) $ and $\mathscr{W}\hookrightarrow L^{\alpha(\cdot)}(\Omega) $ are compact and continuous.
 
  The proof of  Lemma \ref{step1} follows by standards arguments.
  
 \begin{lemma} \label{step1} Assume the hypotheses of  Theorem \ref{peso1} are fulfilled. Then the functional $ \mathcal{J}_{\lambda, 0} \in C^{1}( \mathscr{W}, \mathbb{R})$.
\end{lemma}  
 \begin{lemma}\label{step2} Assume the hypotheses of  Theorem \ref{peso1} are fulfilled. Then there exists $\lambda^{\star}>0$ such that, for all $\lambda \in (0, \lambda^{\star})$, there exist $\overline{\mathcal{R}}$, $R>0$ where $ \mathcal{J}_{\lambda, 0}(u)\geqslant \overline{\mathcal{R}}> 0$
for all $u \in \mathscr{W} $  with $ \|u\|_{\mathscr{W}}= R$.
  \end{lemma}
  \begin{proof}
Indeed,  since embedding $\mathscr{W}\hookrightarrow L^{\mathfrak{m}_{1}(\cdot)q'(\cdot)}(\Omega)$ is continuous  there exists a  constant $c_4>0$ such that
  \begin{equation}\label{pep2}
\|u\|_{L^{\mathfrak{m}_{1}(\cdot)q'(\cdot)}(\Omega)}\leqslant c_4\|u\|_{\mathscr{W}} \mbox{ for all } u \in\mathscr{W}. 
\end{equation}
  Fix $R \in (0, 1)$ such that $R <\frac{1}{c_4}$. From \eqref{pep2}, we obtain
   \begin{equation*}\label{pep3}
\|u\|_{L^{\mathfrak{m}_{1}(\cdot)q'(\cdot)}(\Omega)}\leqslant 1 \mbox{ for all } u \in \mathscr{W} \mbox{ with } R = \|u\|_{\mathscr{W}}. 
\end{equation*}
Thus, by Proposition \ref{hold3} and  Proposition \ref{binge}, we have 
\begin{equation}\label{menes}
\begin{split}
\int_{\Omega}\mathfrak{a}(x)|u|^{\mathfrak{m}_{1}(x)}dx \leqslant \|\mathfrak{a}\|_{L^{q(\cdot)}(\Omega)} \|u\|_{L^{\mathfrak{m}_{1}(\cdot)q'(\cdot)}(\Omega)}^{\underline{\mathfrak{m}}_{1}^{-}} \mbox{ for all } u \in \mathscr{W}. 
\end{split}
\end{equation}
 Since $\|u\|_{\mathscr{W}} =R <1 $, using $(a_{2})$, $(a_{3})$, $(\mathcal{K})$,  Proposition \ref{lw0}, and \eqref{menes}, we have that
\begin{equation}\label{pi2}
\begin{split}
 \mathcal{J}_{\lambda, 0}(u) \geqslant  R^{\underline{\mathfrak{m}}_{1}^{-}} \Bigg( \frac{c_{\mathcal{A}}b_0}{p^{+}} R^{p^{+}-\underline{\mathfrak{m}}_{1}^{-}}  - \frac{\lambda c_{4}^{\underline{\mathfrak{m}}_{1}^{-}}\|\mathfrak{a}\|_{L^{q(\cdot)}(\Omega)}}{\underline{\mathfrak{m}}_{1}^{-}}\Bigg).   
\end{split}
\end{equation}
Thus, by  \eqref{pi2}, we can choose $\lambda^{\star}$ in order to $\displaystyle{\frac{c_{\mathcal{A}}b_0}{p^{+}} R^{p^{+}-\underline{\mathfrak{m}}_{1}^{-}}  - \frac{\lambda c_{4}^{\underline{\mathfrak{m}}_{1}^{-}}\|\mathfrak{a}\|_{L^{q(\cdot)}(\Omega)}}{\underline{\mathfrak{m}}_{1}^{-}}}>0 .$ Therefore, for all $\lambda \in (0, \lambda^{\star})$ with
\begin{equation}\label{lamb}
 \lambda^{\star}= \frac{c_{\mathcal{A}}b_0 R^{p^{+}-\underline{\mathfrak{m}}_{1}^{-}}}{2p^{+}}  \cdot \frac{\underline{\mathfrak{m}}_{1}^{-}}{ c_{4}^{\underline{\mathfrak{m}}_{1}^{-}}\|\mathfrak{a}\|_{L^{q(\cdot)}(\Omega)}}
\end{equation}
 and for all $u \in  \mathscr{W} $  with $ \|u\|_{\mathscr{W}} = R $, there exists $\displaystyle{\overline{\mathcal{R}}=\frac{c_{\mathcal{A}}b_0 R^{p^{+}-\underline{\mathfrak{m}}_{1}^{-}}}{2p^{+}}>0}$ such that $ \mathcal{J}_{\lambda, 0}(u) \geqslant \overline{\mathcal{R}} >0.$
  \end{proof}
 
 \begin{lemma}\label{step3} Assume the hypotheses of  Theorem \ref{peso1} are fulfilled. Then there exists $ \underline{\upsilon} \in  \mathscr{W}$ such that,  $\underline{\upsilon} \neq 0$ and  $ \mathcal{J}_{\lambda, 0}(\gamma_{1}\underline{\upsilon})<0$ 
for all $ \gamma_{1} $ small enough.
\end{lemma}
 \begin{proof}
Indeed, since $\mathfrak{m}_{1}(x)< p(x,y)$ for all $x, y \in \overline{\Omega}_{0}$. In the sequence we will use the following notation, $\underline{\mathfrak{m}}^{-}_{0}:=\displaystyle{\inf_{x \in \overline{\Omega}_{0}}\mathfrak{m}_{1}(x)}$ and $\underline{\mathfrak{m}}^{+}_{0}:=\displaystyle{\sup_{x \in \overline{\Omega}_{0}}\mathfrak{m}_{1}(x)}$. Thus, there exists $\varepsilon_{0} >0$ such that $\underline{\mathfrak{m}}^{-}_{0}+\varepsilon_{0} \leqslant p^{-}$. Since  $\mathfrak{m}_{1} \in C^{+}(\overline{\Omega})$ there exists an open  set $\Omega_{1}\subset \Omega_{0}$ such that
\begin{equation}\label{ester}
|\mathfrak{m}_{1}(x)-\underline{\mathfrak{m}}^{-}_{0}| \leqslant \varepsilon \mbox{ for all } x \in \Omega_{1}.
\end{equation}
Consequently, we can conclude that $\mathfrak{m}_{1}(x) \leqslant \underline{\mathfrak{m}}^{-}_{0}+ \varepsilon_{0} \leqslant p^{-}$ for all $x \in \Omega_{1}$.

\noindent Let $\underline{\upsilon} \in C^{\infty}_{0}(\Omega_{0})$  such that $\overline{\Omega}_{1} \subset supp\,\underline{\upsilon}$, $\underline{\upsilon}(x)=1$ for all $x \in \overline{\Omega}_{1}$ and for  $0 < \underline{\upsilon}<1$ in $\Omega_{0}$. Without loss of generality, we way assume $\|\underline{\upsilon}\|_{\mathscr{W}}=1$, by Proposition \ref{lw0}, it follows
\begin{equation}\label{ester2}
\int_{\mathbb{R}^{N}\times \mathbb{R}^{N}}\frac{|\underline{\upsilon}(x)-\underline{\upsilon}(y)|^{p(x,y)}}{|x-y|^{N+sp(x,y)}} \,dx\,dy =1.
\end{equation} 
 Then, by $(a_2)$, $(a_3)$, $(\mathcal{K})$, \eqref{ester}, \eqref{ester2},  Proposition \ref{lw0}, and $\gamma_{1} \in (0,1)$, we have
\begin{equation*}\label{pip2}
\begin{split}
 \mathcal{J}_{\lambda, 0}(\gamma_{1}\underline{\upsilon}) &    \leqslant \frac{ \gamma_{1}^{p^{-}}}{p^{-}} C_{\mathcal{A}}b_1- \frac{\lambda \gamma_{1}^{\underline{\mathfrak{m}}^{-}_{0}+\varepsilon}}{\underline{\mathfrak{m}}^{+}_{0}} \int_{\Omega_{1}}\mathfrak{a}(x)|\underline{\upsilon}|^{\mathfrak{m}_{1}(x)}\,dx.
\end{split}
\end{equation*}
Thus $ \mathcal{J}_{\lambda, 0}(\gamma_{1}\underline{\upsilon})<0$ for all $\gamma_{1}< \beta^{\frac{1}{p^{-}-\underline{\mathfrak{m}}^{-}_{0}-\varepsilon}} $ with

$$
0<\beta< \inf \Bigg\{ 1, \frac{\lambda p^{-}}{\underline{\mathfrak{m}}^{+}_{0}C_{\mathcal{A}}b_1} \displaystyle{\int_{\Omega_{1}}\mathfrak{a}(x)|\underline{\upsilon}|^{\mathfrak{m}_{1}(x)}\,dx}\Bigg\}. 
$$
Therefore, the proof this Lemma is completed.
\end{proof}

 \begin{proof}[Proof of Theorem \ref{peso1}]
   Consider $\lambda^{\star}>0$ be defined  as in \eqref{lamb} and let $\lambda \in (0, \lambda^{\star})$. By Lemma \ref{step1}, $ \mathcal{J}_{\lambda, 0} \in C^{1}( \mathscr{W}, \mathbb{R})$, and by Lemma \ref{step2} it follows that
 \begin{equation}\label{inf}
 \inf_{v\in\partial B_{R}(0)} \mathcal{J}_{\lambda, 0}(v)>0,
 \end{equation}
 where $\partial B_{R}(0)= \{ u \in B_{R}(0): \|u\|_{\mathscr{W}}= R \}$ and $B_{R}(0)$ is the ball centred at the origin in $\mathscr{W}$.
 
\noindent Besides, by Lemma \ref{step3},  there exists $\underline{\upsilon} \in \mathscr{W}$ such that $ \mathcal{J}_{\lambda, 0}(\gamma_{1}\underline{\upsilon})<0$ for all small enough $\gamma_{1} > 0$. Moreover, by \eqref{pi2}, for all $u \in B_{R}(0)$, we have
 \begin{equation}\label{pii2}
\begin{split}
 \mathcal{J}_{\lambda, 0}(u) \geqslant \frac{c_{\mathcal{A}}b_0}{p^{+}}  \|u\|^{p^{+}}_{\mathscr{W}}  - \frac{\lambda c_{1}^{\underline{\mathfrak{m}}_{1}^{-}}}{\underline{\mathfrak{m}}_{1}^{-}} \|u\|^{\underline{\mathfrak{m}}_{1}^{-}}_{\mathscr{W}}.  
\end{split}
\end{equation}
 Hence,
\begin{equation}\label{pii3}
-\infty < \overline{d} =  \inf_{v \in \overline{ B_{R}(0)}}{ \mathcal{J}_{\lambda, 0}}(v)< 0.
\end{equation}
Consequently, by \eqref{inf} and \eqref{pii3}  let $\varepsilon >0$ such that
\begin{equation*}
0 < \varepsilon <  \inf_{v\in\partial B_{R}(0)} \mathcal{J}_{\lambda, 0}(v) -  \inf_{v\in \overline{ B_{R}(0)}} \mathcal{J}_{\lambda, 0}(v). 
\end{equation*}
 Applying the Ekeland's Variational Principle, \cite[Theorem 2.4]{wil}, to the functional $ \mathcal{J}_{\lambda, 0}: \overline{B_{R}(0)}\to \mathbb{R}$, there exists $u_{\varepsilon} \in  \overline{B_{R}(0)}$ such that
 \begin{equation}
\label{pekl1}\left\{\begin{array}{rc} 
\begin{split}
 \mathcal{J}_{\lambda, 0}(u_{\varepsilon})& < \inf_{v\in\overline{ B_{R}(0)}}{ \mathcal{J}_{\lambda, 0}(v)} + \varepsilon,\\
 \mathcal{J}_{\lambda, 0}(u_{\varepsilon})& < \mathcal{J}_{\lambda,0}(u)+ \varepsilon\|u-u_{\varepsilon}\|_{\mathscr{W}} \mbox{ for all } u\neq u_{\varepsilon}.
\end{split}
\end{array}\right.
\end{equation}
 Since,
\begin{equation*}
 \mathcal{J}_{\lambda, 0}(u_{\varepsilon})\leqslant \inf_{v\in \overline{ B_{R}(0)}}{ \mathcal{J}_{\lambda, 0}(v)} + \varepsilon < \inf_{ v\in \partial B_{R}(0)}{ \mathcal{J}_{\lambda, 0}}(v)
\end{equation*}
 we deduce that $u_{\varepsilon} \in B_{R}(0)$.

\noindent Now, we  define $\mathbb{T}_{\lambda,\mathfrak{a}}^{\varepsilon}: \overline{ B_{R}(0)} \to \mathbb{R} $ by $\mathbb{T}_{\lambda,\mathfrak{a}}^{\varepsilon}(u)= \mathcal{J}_{\lambda,0}(u)+\varepsilon\|u -u_{\varepsilon}\|_{\mathscr{W}}$. Note that by \eqref{pekl1}, we get  $\mathbb{T}_{\lambda,\mathfrak{a}}^{\varepsilon}(u_{\varepsilon})= \mathcal{J}_{\lambda,0}(u_{\varepsilon})<\mathbb{T}_{\lambda,\mathfrak{a}}^{\varepsilon}(u)$ for all $u\neq u_{\varepsilon}$. Thus $u_{\varepsilon}$ is a minimum point of $ \mathbb{T}_{\lambda,\mathfrak{a}}^{\varepsilon} $ on $\overline{B_{R}(0)}$. Therefore,
 \begin{equation*}
\frac{\mathbb{T}_{\lambda,\mathfrak{a}}^{\varepsilon}(u_{\varepsilon}+tv)- \mathbb{T}_{\lambda,\mathfrak{a}}^{\varepsilon}(u_{\varepsilon})}{t}\geqslant 0 \mbox{ for all  small enough } t>0 \mbox{ and } v \in B_{R}(0).
\end{equation*}

\noindent By this fact, we obtain
\begin{equation*}
\frac{ \mathcal{J}_{\lambda, 0}(u_{\varepsilon}+tv)- \mathcal{J}_{\lambda,0}(u_{\varepsilon})}{t} + \varepsilon\|v\|_{\mathscr{W}}\geqslant 0.
\end{equation*}
 Taking $t \to 0^{+}$, it follows that $
\langle \mathcal{J}'_{\lambda,0}(u_{\varepsilon}), v\rangle  + \varepsilon \|v \|_{\mathscr{W}} \geqslant 0$ and we infer that
 \begin{equation} \label{ekps}
\| \mathcal{J}'_{\lambda, 0}(u_{\varepsilon})\|_{\mathscr{W}'} \leqslant \varepsilon.
\end{equation}
Then, by \eqref{pii3} and \eqref{ekps} we deduce that there exists a sequence $(w_k)_{k\in \mathbb{N}}\subset B_{R}(0)$ such that
 \begin{equation}\label{eks1}
 \mathcal{J}_{\lambda, 0}(w_k)\to \overline{d} \hspace{0.2cm}\mbox{ and } \hspace{0.2cm}
 \mathcal{J}'_{\lambda, 0}(w_k)\to 0 \hspace{0.2cm}\mbox{ as } k\to +\infty.
\end{equation}
 From \eqref{pii2} and \eqref{eks1},  we have that sequence $(w_k)_{k\in\mathbb{N}}$ is bounded in $\mathscr{W}$. Indeed, if $ \|w_k\|_{\mathscr{W}}\to+\infty$, by  \eqref{pii2} and since   $\underline{\mathfrak{m}}_{1}^{-} <  p^{+}$  we get   $ \mathcal{J}_{\lambda, 0}(w_k)\to + \infty$, which is a contradiction with \eqref{eks1}. Therefore the  sequence $(w_k)_{k\in\mathbb{N}}$ is bounded in $\mathscr{W}$. From Lemma \ref{2.11}, there exists $w \in \mathscr{W}$ such that $w_k \rightharpoonup w$ in $\mathscr{W},$ $w_k \to w$  in $L^{\mathfrak{m}_{1}(\cdot)}(\Omega)$, and $w_{k}(x) \to w(x)$ a.e.  $x\in  \Omega$ as $k \to +\infty$.
 
 \noindent To finalize the proof we will show that $w_k \to w $ in $\mathscr{W}$ as $k \to  +\infty$.\\
 \textbf{Claim c1. }
\begin{equation*}
 \lim_{k\to +\infty}\int_{\Omega}\mathfrak{a}(x)|w_k|^{\mathfrak{m}_{1}(x)-2}w_k(w_k-w)\, dx=0.
 \end{equation*}
Indeed, from Proposition \ref{hold3} and Proposition \ref{binge}, we have
\begin{equation}\label{zana}
\begin{split}
\int_{\Omega}\mathfrak{a}(x)|w_k|^{\mathfrak{m}_{1}(x)-2}w_k(w_k-w)\, dx \leqslant  \|\mathfrak{a}\|_{L^{q(\cdot)}(\Omega)} (1+ \|w_k\|_{L^{\mathfrak{m}_{1}(\cdot)}(\Omega)}^{\underline{\mathfrak{m}}_{1}^{+}-1})\|w_k-w\|_{L^{\alpha(\cdot)}(\Omega)}.
\end{split}
\end{equation}
Since $\mathscr{W}$ is continuously embedded in $L^{\mathfrak{m}_{1}( \cdot)}(\Omega)$ and $(w_k)_{k \in \mathbb{N}}$ is bounded in $\mathscr{W}$, so  $(w_k)_{k\in \mathbb{N}}$ is bounded in $L^{\mathfrak{m}_{1}(\cdot)}(\Omega)$. From Lemma \ref{2.11}, the embedding  $\mathscr{W}\hookrightarrow L^{\alpha(\cdot)}(\Omega)$ is compact, we deduce $\|w_k -w \|_{L^{\alpha(\cdot)}(\Omega)}\to 0$ as $k\to +\infty$. Therefore, using \eqref{zana} the proof of \textbf{Claim c1} is complete.

\noindent On the other hand, by \eqref{eks1}, we infer that

 \begin{equation}\label{ll2b}
 \lim_{k\to +\infty}\langle \mathcal{J}'_{\lambda,0}(w_k), w_k-w\rangle=0.
 \end{equation}
Consequently by \textbf{ Claim c1} and \eqref{ll2b}, we get
 \begin{equation*}
 \lim_{k\to +\infty}\langle \Phi'(w_k), w_k-w\rangle=\lim_{k\to + \infty}\langle \mathcal{J}'_{\lambda,0}(w_k), w_k-w\rangle=0.
 \end{equation*}
 Thus, by Lemma \ref{ll1} it follows that $w_k \to w $  in $\mathscr{W}$ as $k \to +\infty$. Since  $ \mathcal{J}_{\lambda, 0} \in C^{1}(\mathscr{W}, \mathbb{R})$, using \eqref{eks1}, we obtain
\begin{equation*}
 \mathcal{J}_{\lambda, 0}(w)\leqslant \lim_{k\to +\infty} \mathcal{J}_{\lambda,0}(w_k)= \overline{d}<0  \hspace{0.2cm}\mbox{ and } \hspace{0.2cm}
 \mathcal{J}'_{\lambda, 0}(w)=0.
\end{equation*}
 Therefore, $w$ is a nontrivial weak solution to  problem \eqref{p1} and thus any $\lambda \in (0, \lambda^{\star})$ is an eigenvalue of problem \eqref{p1}.
 \end{proof}
  \subsection{Proof of Theorem \ref{peso2}} \label{teorema4}
  \hfill \break 
  Before we prove  Theorem \ref{peso2} we define 
  $$\mathscr{W}^{+}= \Bigg\{ u \in \mathscr{W}: \int_{\Omega}\mathfrak{a}(x)|u|^{\mathfrak{m}_{1}(x)}dx>0\Bigg\}, \mathscr{W}^{-}= \Bigg\{ u \in \mathscr{W}: \int_{\Omega}\mathfrak{a}(x)|u|^{\mathfrak{m}_{1}(x)}dx<0\Bigg\},$$
\begin{equation}\label{definin}
  \begin{split}
  \lambda^{\star \star}= \inf_{u \in \mathscr{W}^{+}} \frac{\Phi(u)}{\mathcal{I}_{\mathfrak{a}}(u)}, \hspace{0.3cm} \lambda_{\star}=\inf_{u \in \mathscr{W}^{+}} \frac{\displaystyle{\int_{\mathbb{R}^{N}\times \mathbb{R}^{N}}\mathcal{A}(u(x)-u(y))(u(x)-u(y))K(x,y)\,dx\,dy}}{\displaystyle{\int_{\Omega}\mathfrak{a}(x)|u|^{\mathfrak{m}_{1}(x)}\,dx}}, \\ \mu^{\star \star}= \sup_{u \in \mathscr{W}^{-}} \frac{\Phi(u)}{\mathcal{I}_{\mathfrak{a}}(u)}, \hspace{0.3cm} \mu_{\star}=\inf_{u \in \mathscr{W}^{-}} \frac{\displaystyle{\int_{\mathbb{R}^{N}\times \mathbb{R}^{N}}\mathcal{A}(u(x)-u(y))(u(x)-u(y))K(x,y)\,dx\,dy}}{\displaystyle{\int_{\Omega}\mathfrak{a}(x)|u|^{\mathfrak{m}_{1}(x)}\,dx}}.
  \end{split} 
  \end{equation} 
 
 \begin{proof}[Proof of Theorem \ref{peso2}]
Note that if $\lambda$  is an eigenvalue of problem \eqref{p1} with weight function  $\mathfrak{a}$, then $- \lambda$ is an eigenvalue of problem \eqref{p1} with weight $-\mathfrak{a}$ . Hence, it is enough to show Theorem \ref{peso2} only for $\lambda >0$. So  problem  \eqref{p1} has
only to be considered in $ \mathscr{W}^{+}$. For this case, the proof is divided into the following four steps.\\
\textbf{\textit{Step 1.}} $\lambda_{\star}>0$.
  
 \noindent First we observe that since  $\mathscr{A}$ is strictly convex (see $(a_{1})$) and  by \cite[Lemma 15.4]{kavian}, we get
  \begin{equation}\label{kav1}
\mathscr{A}(t) \leqslant \mathcal{A}(t)t \mbox{ for all } t \in \mathbb{R}.
\end{equation}
  Then, by \eqref{kav1}  it follows that 
\begin{equation}\label{rr1}
\begin{split}
\frac{\Phi(u)}{\mathcal{I}_{\mathfrak{a}}(u)} &\leqslant  \underline{\mathfrak{m}}_{1}^{+}\frac{\displaystyle{\int_{\mathbb{R}^{N}\times \mathbb{R}^{N}}\mathcal{A}(u(x)-u(y))(u(x)-u(y))K(x,y)\,dx\,dy}}{\displaystyle{\int_{\Omega}\mathfrak{a}(x)|u|^{\mathfrak{m}_{1}(x)}\,dx}} \mbox{ for all } u \in \mathscr{W}^{+}.
\end{split}
\end{equation}
 On the other hand, from $(a_{3})$
\begin{equation}\label{rr2}
\begin{split}
\frac{\Phi(u)}{\mathcal{I}_{\mathfrak{a}}(u)} & \geqslant  \frac{\underline{\mathfrak{m}}_{1}^{-}}{p^{+}}\frac{\displaystyle{\int_{\mathbb{R}^{N}\times \mathbb{R}^{N}}\mathcal{A}(u(x)-u(y))(u(x)-u(y))K(x,y)}\,dx\,dy}{\displaystyle{\int_{\Omega}\mathfrak{a}(x)|u|^{\mathfrak{m}_{1}(x)}\,dx}}  \mbox{ for all } u \in \mathscr{W}^{+}.
\end{split}
\end{equation}
  Then by \eqref{definin}, \eqref{rr1}, and \eqref{rr2}, we get
  \begin{equation}\label{corri}
  \frac{\underline{\mathfrak{m}}_{1}^{-}}{p^{+}} \lambda_{\star}\leqslant \lambda^{\star \star} \leqslant \underline{\mathfrak{m}}_{1}^{+}\lambda_{\star}.
  \end{equation}
 
\noindent  Since $p^{+}< \underline{\mathfrak{m}}_{1}^{-}$, it follows that $\lambda^{\star \star} \geqslant \lambda_{\star} \geqslant 0.$
  
\noindent \textbf{Claim:}
   \begin{equation*}\label{claim1}
(a)\lim_{ \| u\|_{\mathscr{W}}\to 0, u\in \mathscr{W}^{+} }\,\frac{\Phi(u)}{\mathcal{I}_{\mathfrak{a}}(u)}= + \infty;\,\,\,\,\,\,\, (b) \lim_{ \| u\|_{\mathscr{W}}\to + \infty, u\in \mathscr{W}^{+} }\,\frac{\Phi(u)}{\mathcal{I}_{\mathfrak{a}}(u)}= + \infty.
\end{equation*}
  Indeed, using Proposition \ref{hold3} and Proposition \ref{binge}, it follows that
  \begin{equation}\label{528}
  \begin{split}
  |\mathcal{I}_{\mathfrak{a}}(u)|   \leqslant & \frac{2}{\underline{\mathfrak{m}}_{1}^{-}}\|\mathfrak{a}\|_{L^{q(\cdot)}(\Omega)}\|u\|_{L^{\mathfrak{m}_{1}(\cdot)q'(\cdot)}(\Omega)} ^{\underline{\mathfrak{m}}_{1}^{l}}, 
  \end{split}
  \end{equation} 
   where $l=-$, if $\|u \|_{L^{\mathfrak{m}_{1}(\cdot)q'(\cdot)}(\Omega)}<1$ and $l=+$, if $\|u \|_{L^{\mathfrak{m}_{1}(\cdot)q'(\cdot)}(\Omega)}\geqslant 1$.
  
 \noindent  Since   $q(x)> \sup\, \Bigg\{ 1, \frac{Np(x)}{Np(x)+sp(x)\mathfrak{m}_{1}(x)-N\mathfrak{m}_{1}(x)} \Bigg\}$ for all $x \in \overline{\Omega}$, we have $1<q'(x)\mathfrak{m}_{1}(x)< p^{\star}_{s}(x)$ for all $x \in \overline{\Omega}$, then by Lemma \ref{2.11} there exist a constant $c_{5}>0$ such that 
   \begin{equation}\label{merguu}
  \|u\|_{L^{\mathfrak{m}_{1}(\cdot)q'(\cdot)}(\Omega)}\leqslant c_{5}\|u\|_{\mathscr{W}}.
  \end{equation}
  Thus by \eqref{528} and \eqref{merguu}, we get
   \begin{equation}\label{a15}
 |\mathcal{I}_{\mathfrak{a}}(u)|\leqslant  \frac{2c_{5}^{\underline{\mathfrak{m}}_{1}^l}}{\underline{\mathfrak{m}}_{1}^{-}}\|\mathfrak{a}\|_{L^{q(\cdot)}(\Omega)}\|u\|_{\mathscr{W}}^{\underline{\mathfrak{m}}_{1}^{l}}.
 \end{equation}
   On the other hand, by $(a_{2})$, $(a_{3})$, and $(\mathcal{K})$, we infer that
  \begin{equation}\label{a16}
  \begin{split}
  \Phi(u)\geqslant & \frac{c_{\mathcal{A}}b_0}{p^{+}}  \int_{\mathbb{R}^{N}\times \mathbb{R}^{N}}\frac{|u(x)-u(y)|^{p(x,y)}}{|x-y|^{N+sp(x,y)}}\,dx\,dy.
  \end{split}
  \end{equation}
   Then for $u \in \mathscr{W}^{+}$ with  $\|u\|_{\mathscr{W}}\leqslant 1$  by \eqref{merguu}, \eqref{a15}, \eqref{a16}, and Proposition \ref{lw0}, we have
  \begin{equation}\label{a17}
 \begin{split}
 \frac{\Phi(u)}{\mathcal{I}_{\mathfrak{a}}(u)} \geqslant & \frac{c_{\mathcal{A}}b_0\underline{\mathfrak{m}}_{1}^{-}}{2c_{5}^{\underline{\mathfrak{m}}_{1}^{-}}p^{+}} \frac{\|u\|_{\mathscr{W}}^{p^{+}-\underline{\mathfrak{m}}_{1}^{-}}} { \|\mathfrak{a}\|_{L^{q(\cdot)}(\Omega)}}. 
 \end{split}
 \end{equation}
 Since $\underline{\mathfrak{m}}_{1}^{+}>p^{+}$, using  \eqref{a17} we conclude that
  \begin{equation*}
  \frac{\Phi(u)}{\mathcal{I}_{\mathfrak{a}}(u)} \to + \infty \hspace{0.2cm}\mbox{ as }   \hspace{0.2cm} \|u\|_{\mathscr{W}}\to 0, \hspace{0.2cm} u \in \mathscr{W}^{+}.
  \end{equation*}
 Therefore, the  relation $(a)$ holds.

 \noindent Now, since $\underline{\mathfrak{m}}_{1}^{+}- \frac{1}{2}< \underline{\mathfrak{m}}_{1}^{-}$ it follows that there exists $\eta >0$  such that  $\underline{\mathfrak{m}}_{1}^{+}- \frac{1}{2}<\eta< \underline{\mathfrak{m}}_{1}^{-}$, which implies that
\begin{equation}\label{a21}
\underline{\mathfrak{m}}_{1}^{+}- 1 < \underline{\mathfrak{m}}_{1}^{+}- \frac{1}{2}<  \eta \Rightarrow 1+ \eta - \underline{\mathfrak{m}}_{1}^{+}>0 \mbox{ and }
2(\underline{\mathfrak{m}}_{1}^{-}-\eta) \leqslant  2(\underline{\mathfrak{m}}_{1}^{+}-\eta)<1.
\end{equation} 
 
 \noindent   Taking $r(x)$ be any measurable function satisfying
  \begin{equation}\label{a22}
\begin{split}
\sup\Bigg\{ \frac{q(x)}{1+ \eta q(x)}, \frac{p^{\star}_{s}(x)}{p^{\star}_{s}(x)+ \eta -\mathfrak{m}_{1}(x)}\Bigg\} \leqslant r(x) \leqslant & \inf \Bigg\{  \frac{p^{\star}_{s}(x)}{p^{\star}_{s}(x)+ \eta q(x)}, \frac{1}{1+ \eta- \mathfrak{m}_{1}(x)}\Bigg\} \\ \mbox{ for all } x \in \overline{\Omega}, \mbox{ and } \\
\eta\Bigg(\frac{\underline{r}^{+}}{\underline{r}^{-}}+1\Bigg)&<\underline{\mathfrak{m}}_{1}^{-}.
\end{split}
\end{equation} 
 Thus by \eqref{a21} and \eqref{a22} $r \in L^{\infty}(\Omega)$ and $1<r(x)<q(x)$ for all $x \in \overline{\Omega}$.
 Then, for $u \in \mathscr{W}^{+} $ using  Proposition  \ref{hold3}, we obtain
 
 \begin{equation}\label{a233}
  \begin{split}
  |\mathcal{I}_{\mathfrak{a}}(u)|\leqslant & \frac{1}{\underline{\mathfrak{m}}_{1}^{-}}\Bigg(\frac{1}{\underline{r}^{-}} + \frac{1}{\underline{r}'^{-}} \Bigg)\|\mathfrak{a}|u|^{\eta}\|_{L^{r(\cdot)}(\Omega)}\||u|^{\mathfrak{m}_{1}(x)-\eta}\|_{L^{r'(\cdot)}(\Omega)} .
  \end{split}
  \end{equation}
 Now, note that by Proposition \ref{masmenos}
  \begin{equation*}\label{a25}
\|\mathfrak{a}|u|^{\eta}\|_{L^{r(\cdot)}(\Omega)} \leqslant \left\{\begin{array}{rc} 
\begin{split}
&1,  \mbox{ if } \|\mathfrak{a}|u|^{\eta}\|_{L^{r(\cdot)}(\Omega)}\leqslant 1; \\
& \Bigg[ \int_{\Omega} |\mathfrak{a}(x)|^{r(x)}|u|^{\eta r(x)}\,dx \Bigg]^{\frac{1}{\underline{r}^{-}}} \leqslant 2  \||\mathfrak{a}|^{r(x)}\|_{L^{\frac{q(\cdot)}{r(\cdot)}}(\Omega)}^{\frac{1}{\underline{r}^{-}}} \||u|^{\eta r(x)}\|_{L^{\big(\frac{q(\cdot)}{r(\cdot)}\big)'}(\Omega)}^{\frac{1}{\underline{r}^{-}}}, \\ &\mbox{ if }  \|\mathfrak{a}|u|^{\eta}\|_{L^{r(\cdot)}(\Omega)}> 1.
\end{split}
\end{array}\right.
\end{equation*}
 From Proposition \ref{binge}, we have
  \begin{equation}\label{a26}
  \begin{split}
\|\mathfrak{a}|u|^{\eta}\|_{L^{r(\cdot)}(\Omega)}  \leqslant & 1 +  2 \Bigg( 1+  \|\mathfrak{a}\|_{L^{q(\cdot)}(\Omega)}^{\frac{\underline{r}^{+}}{\underline{r}^{-}}}\Bigg)\Bigg(1+ \|u\|^{\eta \frac{\underline{r}^{+}}{\underline{r}^{-}}}_{L^{\eta r(\cdot)\big(\frac{q(\cdot)}{r(\cdot)}\big)'}(\Omega)}\Bigg).
\end{split}
\end{equation}
Similarly  using the same arguments as above,  we obtain
  \begin{equation*}\label{a27}
\||u|^{\mathfrak{m}_{1}(x)-\eta}\|_{L^{r'(\cdot)}(\Omega)} \leqslant \left\{\begin{array}{rc} 
\begin{split}
&1,  \mbox{ if } \|u\|_{L^{(\mathfrak{m}_{1}(\cdot) -\eta)r'(\cdot)}(\Omega)}\leqslant 1; \\
&  \|u\|_{L^{(\mathfrak{m}_{1}(\cdot) -\eta)r'(\cdot)}(\Omega)}^{\underline{\mathfrak{m}}_{1}^{+}-\eta}, \mbox{ if } \|u\|_{L^{(\mathfrak{m}_{1}(\cdot) -\eta)r'(\cdot)}(\Omega)} > 1.
\end{split}
\end{array}\right.
\end{equation*}
Thus, 
\begin{equation}\label{a28}
\||u|^{\mathfrak{m}_{1}(x)-\eta}\|_{L^{r'(\cdot)}(\Omega)} \leqslant 1 + \|u\|_{L^{(\mathfrak{m}_{1}(\cdot) -\eta)r'(\cdot)}(\Omega)}^{\underline{\mathfrak{m}}_{1}^{+}-\eta} .
\end{equation}
Therefore, using \eqref{a233}, \eqref{a26}, and \eqref{a28}, we have that
\begin{equation}\label{a29}
\begin{split}
  |\mathcal{I}_{\mathfrak{a}}(u)|  \leqslant & \frac{1}{\underline{\mathfrak{m}}_{1}^{-}}\Bigg(\frac{1}{\underline{r}^{-}} + \frac{1}{\underline{r}'^{-}} \Bigg) \Bigg[ 1 + \|u\|_{L^{(\mathfrak{m}_{1}(\cdot) -\eta)r'(\cdot)}(\Omega)}^{\underline{\mathfrak{m}}_{1}^{+}-\eta} + 2(1+  \|\mathfrak{a}\|_{L^{q(\cdot)}(\Omega)}^{\frac{\underline{r}^{+}}{\underline{r}^{-}}})\|u\|^{\underline{\mathfrak{m}}_{1}^{+}-\eta}_{L^{(\mathfrak{m}_{1}(\cdot)- \eta )r'(\cdot)}(\Omega)} \\& +2 \Bigg(1+ \|\mathfrak{a}\|_{L^{q(\cdot)}(\Omega)}^{\frac{\underline{r}^{+}}{\underline{r}^{-}}} \Bigg)+ 2\Bigg(1+  \|\mathfrak{a}\|_{L^{q(\cdot)}(\Omega)}^{\frac{\underline{r}^{+}}{\underline{r}^{-}}}\Bigg)\|u\|^{\eta \frac{\underline{r}^{+}}{\underline{r}^{-}}}_{L^{\eta r(\cdot)\big(\frac{q(\cdot)}{r(\cdot)}\big)'}(\Omega)} \\ &  + \Bigg(1+  \|\mathfrak{a}\|_{L^{q(\cdot)}(\Omega)}^{\frac{\underline{r}^{+}}{\underline{r}^{-}}}\Bigg)\Bigg(\|u\|^{2(\underline{\mathfrak{m}}_{1}^{+}-\eta) }_{L^{(\mathfrak{m}_{1}(\cdot)- \eta )r'(\cdot)}(\Omega)} +\|u\|^{2\eta \frac{\underline{r}^{+}}{\underline{r}^{-}}}_{L^{\eta r(\cdot)\big(\frac{q(\cdot)}{r(\cdot)}\big)'}(\Omega)}\Bigg)\Bigg] .
  \end{split}
\end{equation}
Now, since $r(x)$ is chosen such that \eqref{a22} is fulfilled, then

\begin{equation*}
1\leqslant  \eta r(x)\Bigg(\frac{q(x)}{r(x)}\Bigg)', \mbox{ and } (\mathfrak{m}_{1}(x)-\eta)r'(x)\leqslant p^{\star}_{s}(x) \hspace{0.1cm}\mbox{ a.e.}
\hspace{0.1cm} x \in\overline{\Omega}.
\end{equation*}
Thus, we obtain continuous embedding  $\mathscr{W}\hookrightarrow L^{\eta r(\cdot)\big(\frac{q(\cdot)}{r(\cdot)}\big)'}(\Omega)$ and $\mathscr{W}\hookrightarrow L^{(\mathfrak{m}_{1}(\cdot)- \eta )r'(\cdot)}(\Omega)$. Consequently, there  exist positive constants $c_6$,  $c_7$, $M_1$, and $M_2$ such that by \eqref{a29}, follows that 
\begin{equation}\label{a31}
\begin{split}
  |\mathcal{I}_{\mathfrak{a}}(u)|  \leqslant M_1 + M_2\Bigg( \|u\|^{2(\underline{\mathfrak{m}}_{1}^{+}-\eta) }_{\mathscr{W}} +\|u\|^{2\eta \frac{\underline{r}^{+}}{\underline{r}^{-}}}_{ \mathscr{W}    }\Bigg).
\end{split}
\end{equation}
Since $p^{-}>1> 2(\underline{\mathfrak{m}}_{1}^{+}-\eta) \geqslant 2 (\underline{\mathfrak{m}}_{1}^{-}-\eta) \geqslant 2\eta\frac{\underline{r}^{+}}{\underline{r}^{-}}> 2\eta$ by \eqref{a21} and \eqref{a22}, using  \eqref{a16}, \eqref{a31}, and Proposition \ref{lw0} for all $u \in \mathscr{W}^{+}$ with $\|u\|_{\mathscr{W}}>1$, we have that 
\begin{equation*}\label{a32}
 \begin{split}
 \frac{\Phi(u)}{\mathcal{I}_{\mathfrak{a}}(u)} \geqslant & \frac{1}{ p^{+}} \frac{\|u\|_{\mathscr{W}}^{p^{-}}} { M_1 + M_2\Bigg( \|u\|^{2(\underline{\mathfrak{m}}_{1}^{+}-\eta) }_{\mathscr{W}} +\|u\|^{2\eta \frac{\underline{r}^{+}}{\underline{r}^{-}}}_{ \mathscr{W}   }\Bigg)} \to + \infty \mbox{  as  } \|u\|_{\mathscr{W}}\to + \infty,
 \end{split}
 \end{equation*}
Therefore, the relation $(b)$ is holds.

\noindent  Now, we  proving that $\lambda_{\star}>0$. Let us suppose by contradiction that  $\lambda_{\star}=0$, then by \eqref{corri} it follows that  $\lambda^{\star \star}=0$. Let $(u_{k})_{k\in \mathbb{N}}\subset \mathscr{W}^{+}\setminus\{0\}$ be such that 
\begin{equation}\label{zeroo}
\lim_{k\to +\infty} \frac{\Phi(u_k)}{\mathcal{I}_{\mathfrak{a}}(u_k)}=0.
\end{equation}
Note that by  \eqref{a17}, we have 
\begin{equation}\label{a331}
\frac{\Phi(u_k)}{\mathcal{I}_{\mathfrak{a}}(u_k)} \geqslant \left\{\begin{array}{rc} 
\begin{split}
\frac{c_{\mathcal{A}}b_0\underline{ \mathfrak{m}}_{1}^{-}}{2c_{5}^{\underline{\mathfrak{m}}_{1}^+} p^{+}} \frac{\|u\|_{\mathscr{W}}^{p^{+}-\underline{\mathfrak{m}}_{1}^{-}}} { \|\mathfrak{a}\|_{L^{q(\cdot)}(\Omega)}},  &\mbox{ if } \|u\|_{\mathscr{W}}\leqslant 1; \\
\frac{c_{\mathcal{A}}b_0\underline{ \mathfrak{m}}_{1}^{-}}{2c_{5}^{\underline{\mathfrak{m}}_{1}^-} p^{+}} \frac{\|u\|_{\mathscr{W}}^{p^{-}-\underline{\mathfrak{m}}_{1}^{+}}} { \|\mathfrak{a}\|_{L^{q(\cdot)}(\Omega)}}, &\mbox{ if }  \|u\|_{\mathscr{W}}> 1 .
\end{split}
\end{array}\right.
\end{equation}
By hypothesis, we know that $p^{+}-\underline{\mathfrak{m}}_{1}^{-}<0$ and $p^{-}-\underline{\mathfrak{m}}_{1}^{+}<0$. Thus  \eqref{a331} implies that $\|u_k\|_{\mathscr{W}}\to +\infty$ as $k \to +\infty$ and using $(b)$, we conclude
$$\lim_{k\to +\infty} \frac{\Phi(u_k)}{\mathcal{I}_{\mathfrak{a}}(u_k)}=+\infty.$$
However, this contradicts \eqref{zeroo}. Consequently, we have $\lambda_{\star}>0$.\\
  \textbf{\textit{Step 2.}} $\lambda^{\star \star}$ is an eigenvalue of problem \eqref{p1}.

\noindent Indeed, let $(u_k)_{k \in \mathbb{N}}\subset \mathscr{W}^{+}\setminus\{0\}$ be a minizing sequence for $\lambda^{\star \star}$, that is
  \begin{equation}\label{a33d}
  \lim_{k\to +\infty} \frac{\Phi(u_k)}{\mathcal{I}_{\mathfrak{a}}(u_k)}= \lambda^{\star \star}>0.
  \end{equation}
Note that by $(a)$ and \eqref{a33d}, we have $(u_k)_{k \in \mathbb{N}}$ is bounded in $\mathscr{W}$. Since $\mathscr{W}$  is reflexive, it follows that there exists $u_0 \in \mathscr{W} $ such that $ u_k\rightharpoonup u_0$  in $\mathscr{W}$, thus by Lemma \ref{ll1}, we have that  
\begin{equation}\label{a34}
\lim_{k\to +\infty} \Phi(u_k) \geqslant\Phi(u_0). 
\end{equation}

\noindent On the other hand, by Lemma \ref{2.11} we get $
u_k \to u_0 \mbox{ in } L^{\mathfrak{m}_{1}(\cdot)q'(\cdot)}(\Omega)$. Then, we infer that
$\|u_k\|_{L^{\mathfrak{m}_{1}(\cdot)q'(\cdot)}(\Omega)}\to \|u_0\|_{L^{\mathfrak{m}_{1}(\cdot)q'(\cdot)}(\Omega)}$,  $\||u_k|^{\mathfrak{m}_{1}(\cdot)}\|_{L^{q'(\cdot)}(\Omega)}\to \||u_0|^{\mathfrak{m}_{1}(\cdot)}\|_{L^{q'(\cdot)}(\Omega)}$,  $\||u_k|^{\mathfrak{m}_{1}(\cdot)}\|_{L^{q'(\cdot)}(\Omega)}$ is bounded and  $|u_k|^{\mathfrak{m}(\cdot)}\rightharpoonup|u_0|^{\mathfrak{m}(\cdot)}$ in $L^{q'(\cdot)}(\Omega)$.
Thus $|u_k|^{\mathfrak{m}(\cdot)}\to|u_0|^{\mathfrak{m}(\cdot)}$ and  from  \eqref{a15},  we conclude that
\begin{equation*}
\begin{split}
|\mathcal{I}_{\mathfrak{a}}(u_k)-\mathcal{I}_{\mathfrak{a}}(u_0)| \leqslant \frac{2c_5}{\underline{\mathfrak{m}}_{1}^{-}}\|\mathfrak{a}\|_{L^{q(\cdot)}(\Omega)}\||u_k|^{\mathfrak{m}_{1}(x)}- |u_0|^{\mathfrak{m}_{1}(x)}\|_{L^{q'(\cdot)}(\Omega)} \to 0 \mbox{ as } k\to +\infty.
\end{split}  
\end{equation*}
 Therefore, 
\begin{equation}\label{a38}
\lim_{k\to +\infty} \mathcal{I}_{\mathfrak{a}}(u_k) = \mathcal{I}_{\mathfrak{a}}(u_0).
\end{equation}
In view of \eqref{a34}  and \eqref{a38}, we get
\begin{equation*}
\frac{\Phi(u_0)}{\mathcal{I}_{\mathfrak{a}}(u_0)}= \lambda^{\star \star} \mbox{ if } u_{0}\neq 0.
\end{equation*}
It remains to be shown that $u_0$ is nontrivial. We suppose by contradiction  that $u_0=0$. Then $u_k \rightharpoonup 0 $ in $\mathscr{W}$ and $u_k \to 0$ in $L^{\mathfrak{m}_{1}(\cdot)q'(\cdot)}(\Omega)$. Therefore,
\begin{equation}\label{a40}
\lim_{k\to+\infty} \mathcal{I}_{\mathfrak{a}}(u_k)=0.
\end{equation}
Now using \eqref{a33d}, give  $\varepsilon \in (0,  \lambda^{\star \star})$  fixed  there exists $k_0$ such that
\begin{equation*}
\vert \mathcal{J}_{\lambda,0}(u_k)-\lambda^{\star \star}\mathcal{I}_{\mathfrak{a}}(u_k)\vert< \varepsilon \mathcal{I}_{\mathfrak{a}}(u_k) \mbox{ for all }  k \geqslant k_0, 
\end{equation*}
that is,
\begin{equation*}
(\lambda^{\star \star} - \varepsilon)\mathcal{I}_{\mathfrak{a}}(u_k)< \mathcal{J}_{\lambda, 0}(u_k)<( \lambda^{\star \star}+\varepsilon)\mathcal{I}_{\mathfrak{a}}(u_k)\mbox{ for all }  k \geqslant k_0 .
\end{equation*}
 Passing to the limit in the above inequalities as $k \to +\infty$ and using  \eqref{a40}, we get
\begin{equation}\label{l1241}
\lim_{k\rightarrow+\infty} \mathcal{J}_{\lambda, 0}(u_k)=0.
\end{equation} 
Consequently  by \eqref{energia}, \eqref{a40}, and \eqref{l1241},  we obtain
\begin{equation}\label{zero}
\lim_{k\rightarrow+\infty}\Phi(u_k)=0.
\end{equation}
In contrast,  by $(a_{2})$, $(a_{3})$, and $(\mathcal{K})$, it follows that
\begin{equation}\label{a41}
\begin{split}
\Phi(u_k) \geqslant & \frac{c_{\mathcal{A}}b_0}{p^{+}} \int_{\mathbb{R}^{N}\times \mathbb{R}^{N}}\frac{|u_{k}(x)-u_{k}(y)|^{p(x,y)}}{|x-y|^{N+sp(x,y)}}\,dx\,dy \geqslant 0.
\end{split}
\end{equation}
Thus, by \eqref{zero} and \eqref{a41} we deduce 
$$\int_{\mathbb{R}^{N}\times \mathbb{R}^{N}}\frac{|u_{k}(x)-u_{k}(y)|^{p(x,y)}}{|x-y|^{N+sp(x,y)}}\,dx\,dy \to 0 \mbox{ as } k \to +\infty \mbox{ for all } k \in \mathbb{N}.$$
Then by Proposition \ref{lw0}, we conclude
\begin{equation}\label{conver}
u_k \to 0   \mbox{ in } \mathscr{W} \mbox{ as } k \to +\infty.
\end{equation}
Hence by \eqref{conver}  and  Claim $(a)$, we have 
$$ \lim _{\|u_k\|_{\mathscr{W}}\to 0, u \in \mathscr{W}^{+}} \frac{\Phi(u_k)}{\mathcal{I}_{\mathfrak{a}}(u_k)}= +\infty$$
which is a contradiction with \eqref{a33d}. Thus $u_0\neq0$. Therefore, $u_0$ is an eigenfunction and Step 2 is proved.\\
\textbf{\textit{Step 3.}} Given any $\lambda \in (\lambda^{\star \star}, + \infty)$, $\lambda$ is an eigenvalue of problem \eqref{p1}.

\noindent Let $\lambda \in  (\lambda^{\star \star}, +\infty)$  fixed. Then $\lambda$ is an eigenvalue of problem \eqref{p1} if and only if there exists  $u_{\lambda} \in  \mathscr{W}^{+}\setminus \{0\}$ a critical point of $ \mathcal{J}_{\lambda, 0}$.
Note that  $ \mathcal{J}_{\lambda, 0}$ is coercive. Indeed, using  \eqref{a21} and \eqref{a22}, we infer that $p^{-}>1> 2(\underline{\mathfrak{m}}_{1}^{+}-\eta) \geqslant 2 (\underline{\mathfrak{m}}_{1}^{-}-\eta) \geqslant 2\eta\frac{\underline{r}^{+}}{\underline{r}^{-}}> 2\eta$, thus using \eqref{a16} and \eqref{a31}, for all $u \in \mathscr{W}^{+}$ with $\|u\|_{\mathscr{W}}\geqslant 1$, we have
\begin{equation*}\label{a48}
 \begin{split}
 \mathcal{J}_{\lambda, 0}(u) \geqslant & \frac{c_{\mathcal{A}}b_0}{ p^{+}} \|u\|_{\mathscr{W}}^{p^{-}} - \lambda\Bigg[ M_1 + M_2\Bigg( \|u\|^{2(\underline{\mathfrak{m}}_{1}^{+}-\eta) }_{\mathscr{W}} +\|u\|^{2\eta \frac{\underline{r}^{+}}{\underline{r}^{-}}}_{ \mathscr{W}    }\Bigg)\Bigg]  \to  + \infty  \mbox{ as } \|u\|_{\mathscr{W}}\to + \infty.
 \end{split}
 \end{equation*}
 In addition, by  proof of  Step 2 enable us to affirm that $ \mathcal{I}_{\mathfrak{a}}$ is weakly-strongly continuous,  namely $u_k \rightharpoonup u$ implies $\mathcal{I}_{\mathfrak{a}}(u_k) \to \mathcal{I}_{\mathfrak{a}}(u)$. Thus by Lemma \ref{ll1}, the functional $\Phi$ is weakly lower semicontinuous, then  $ \mathcal{J}_{ \lambda,0}$ is weakly lower semicontinuous. Therefore, by   \cite[Theorem 1.2]{costa} there is  a global minimum point $u_{\lambda} \in  \mathscr{W}^{+}$ of $ \mathcal{J}_{\lambda, 0}$, hence $u_{\lambda}$ is a critical point of $ \mathcal{J}_{\lambda, 0}$.

 \noindent To complete Lemma proof we will show that $u_\lambda$ is nontrivial. Indeed, since 
$$\lambda^{\star \star}=\inf_{u \in \mathscr{W}^{+}}\frac{\Phi(u)}{\mathcal{I}_{\mathfrak{a}}(u)}$$
 and $\lambda >\lambda^{\star \star}$, we deduce that there is $v_{\lambda}$ in  $\mathscr{W}^{+}$ such that 
$ \displaystyle{\frac{\Phi(v_{\lambda})}{\mathcal{I}_{\mathfrak{a}}(v_{\lambda})}< \lambda}$, this is, $ \mathcal{J}_{\lambda, 0}(v_{\lambda})< 0.$ 
Thus $\displaystyle{\inf_{v \in \mathscr{W}^{+}} \mathcal{J}_{\lambda, 0}(v) < 0.}$
Consequently, we conclude that  $u_\lambda$ is a nontrivial critical point of $ \mathcal{J}_{\lambda, 0}$. Therefore,  $\lambda$  is an eigenvalue of problem \eqref{p1}.\\
\textbf{\textit{Step 4.}} Given any  $\lambda  \in(0, \lambda_{\star})$, $\lambda$ is not an  eigenvalue of problem \eqref{p1}.\\
\noindent  Indeed, suppose by contradiction  that there exists an eigenvalue $  \lambda \in(0, \lambda_{\star})$ of problem \eqref{p1}. Then  there exists $u_{\lambda} \in \mathscr{W}^{+}$, such that
\begin{equation*}
\langle \Phi'(u_{\lambda}), v \rangle = \lambda\langle \mathcal{I}_{\mathfrak{a}}'(u_{\lambda}), v \rangle \mbox{ for all } \hspace{0.1cm} v \in \mathscr{W}^{+}.
\end{equation*}
Then taking  $v=u_{\lambda}$, since $  \lambda \in (0, \lambda_{\star})$  and by definition of $ \lambda_{\star} $, we have that
\begin{equation*}
\begin{split}
 \langle \Phi'(u_{\lambda}), u_\lambda\rangle = & \lambda\langle \mathcal{I}_{\mathfrak{a}}'(u_{\lambda}), u_{\lambda}\rangle <  \lambda_{\star}\langle \mathcal{I}_{\mathfrak{a}}'(u_{\lambda}), u_{\lambda}\rangle  \leqslant  \langle \Phi'(u_{\lambda}), u_{\lambda}\rangle
 \end{split}
\end{equation*}
which is a  contradiction. Hence, there does not exist $\lambda \in (0 ,  \lambda_{\star})$ eigenvalue of  problem \eqref{p1}. Thus the claim is verified.

 \noindent  Therefore, the proof of Theorem \ref{peso2} is  complete. 
 \end{proof}

\section{Appendix}\label{apendice}
\hfill \break
\subsection{Embedding results } 
\hfill \break
 In this part we present the key embedding results that were used throughout the research.

\noindent \textit{Proof of Lemma \ref{lw1}} 
We consider the set $ \mathbb{M}=\{ u \in \mathscr{W}; \| u \|_{L^{p(\cdot)}(\Omega)}=1\}$. Then to prove this Lemma  it suffices to prove that
\begin{equation*}
\inf_{u\in \mathbb{M}}[u]^{s, p(\cdot, \cdot)}_{\mathbb{R}^{N}}= \zeta_1>0.
\end{equation*}
Initially, we observe that $\zeta_1 \geqslant 0$ and we prove $\zeta_1$ is attained in $\mathbb{M}$. Let $(u_{k})_{k\in \mathbb{N}} \subset \mathbb{M}$ be a minimizing sequence, that is, $[u_{k}]^{s, p(\cdot, \cdot)}_{\mathbb{R}^{N}}\rightarrow \zeta_1$ as $k \to +\infty$. This implies that $(u_{k})_{k\in \mathbb{N}}$ is bounded in  $ \mathscr{W} $ and $L^{p(\cdot)}(\Omega)$, therefore in $W^{s,p(\cdot, \cdot)}(\Omega)$. Consequently  up to a
subsequence $u_k \rightharpoonup u_0$ in $W^{s,p(\cdot, \cdot)}(\Omega)$ as $k\to +\infty$. Thus, from Corollary \ref{3.4a}, it follows that $u_k \to u_0$ in $L^{p(\cdot)}(\Omega)$ as $k\to +\infty$. We extend $u_0$ to $\mathbb{R}^{N}$ by setting $u_0(x) = 0$ on $x \in \mathbb{R}^{N}\setminus \Omega$. This implies $u_k(x) \to u_0(x)$ a.e. $x \in \mathbb{R}^{N}$ as $k \to +\infty$. Hence by using Fatou’s Lemma, we have
\begin{equation*}
\int_{\mathbb{R}^{N}\times \mathbb{R}^{N}} \frac{|u_0(x)-u_0(y)|^{p(x,y)}}{|x-y|^{N+sp(x,y)}}\,dx\,dy \leqslant \liminf_{k\to +\infty}\int_{\mathbb{R}^{N}\times \mathbb{R}^{N}} \frac{|u_k(x)-u_k(y)|^{p(x,y)}}{|x-y|^{N+sp(x,y)}}\,dx\,dy
\end{equation*}
which implies that
$$[u_{0}]^{s, p(\cdot, \cdot)}_{\mathbb{R}^{N}} \leqslant \liminf_{k\to +\infty}[u_{k}]^{s, p(\cdot, \cdot)}_{\mathbb{R}^{N}}=\zeta_1,$$
and thus $u_0 \in \mathscr{W}$. Moreover,  $\| u_0 \|_{L^{p(\cdot)}(\Omega) }=1$ and then $u_0 \in \mathbb{M} $. In particular, $u_{0}\neq 0$ and $[u_{0}]^{s, p(\cdot, \cdot)}_{\mathbb{R}^{N}}= \zeta_1>0$.

\hfill $\Box$

\noindent\textit{Proof of Lemma \ref{2.11}}
First we observe that by Lemma \ref{lw1},  for all $ u \in \mathscr{W}$, we get
\begin{equation}\label{liim}
\begin{split}
\|u\|_{W^{s,p(\cdot,\cdot)}(\Omega)}  \leqslant \|u\|_{L^{p(\cdot)}(\Omega)} + \|u\|_{\mathscr{W}}  \leqslant \Bigg(  \frac{1}{\zeta_1} +1 \Bigg) \|u\|_{\mathscr{W}},
\end{split}
\end{equation}
that is,  $ \mathscr{W}$ is continuously embedded in $ W^{s,p(\cdot,\cdot)}(\Omega)$, and by Corollary \ref{3.4a}   we conclude that  $ \mathscr{W} $ is continuously embedded in $ L^{r(\cdot)}(\Omega)$. To prove that the embedding above  is compact  we consider  $(u_{k})_{k\in \mathbb{N}}$ a bounded sequence in $ \mathscr{W}$. Using \eqref{liim}, follows that $(u_{k})_{k\in \mathbb{N}}$ is bounded in $  W^{s,p(\cdot, \cdot)}(\Omega) $. Hence by Corollary \ref{3.4a}, we infer that there  exists $u_0 \in {L^{r(\cdot)}(\Omega)} $  such that up to a subsequence $u_k\to u_0$  in ${L^{r(\cdot)}(\Omega)} $ as $k \to +\infty$.   Since that $u_k = 0$ a.e. in $\mathbb{R}^{N}\setminus\Omega$ for all $k \in \mathbb{N}$,  so we define $u_0 = 0$ a.e. in  $\mathbb{R}^{N}\setminus\Omega$ and obtain that the convergence occurs in $L^{r(\cdot)}(\Omega)$. This completes the  proof this Lemma.

\hfill $\Box$

 \subsection{Functional properties operator $ \mathscr{L}_{\mathcal{A}K}$}
 \hfill \break
\noindent\textit{Proof of Lemma \ref{ll1}}  
 $(i)$  Using standard arguments proof this item.

\noindent $(ii)$  From $(i)$ the functional $\Phi$ is of   class  $C^{1}(\mathscr{W}, \mathbb{R})$, and by  hypothesis $(a_{1})$,  the functional $\Phi'$ is   monotone. Thus, by  \cite[Lemma 15.4]{kavian}  we conclude that
$\langle \Phi'(u), u_k-u\rangle +\Phi (u) \leqslant \Phi (u_k) $ for all $k \in \mathbb{N}$.

\noindent Thus, since  $u_k \rightharpoonup u$ in $\mathscr{W}$, as $k\to +\infty$ we obtain $\displaystyle{   \Phi (u) \leqslant \liminf_{k\to +\infty}\Phi (u_k)}$. That is, the functional $\Phi$ is weakly lower semicontinuous.
\\
\\
\noindent $(iii)$ Let $(u_{k})_{k\in \mathbb{N}}$ be a sequence in  $\mathscr{W}$ as in the statement. We have that by $(i)$, $\Phi'$ is a continuous  functional. Therefore,
 \begin{equation}\label{ss0}
\displaystyle{\lim_{k\to +\infty}\langle \Phi'(u), u_k-u \rangle = 0}.
\end{equation}
Now, we observe that
\begin{equation}\label{ss1}
\langle \Phi'(u_k)- \Phi'(u), u_k-u \rangle=  \langle \Phi'(u_k), u_k-u \rangle -\langle  \Phi'(u), u_k-u \rangle \mbox{ for all } k \in \mathbb{R}.
\end{equation}
Thus by \eqref{inffo}, \eqref{ss0}, and \eqref{ss1}, we infer
\begin{equation}\label{sa}
\limsup_{k\to +\infty}\langle \Phi'(u_k)- \Phi'(u), u_k-u \rangle \leqslant 0.
\end{equation}
Furthermore, since $\Phi$ is  strictly convex by hypothesis $(a_1)$,   $\Phi'$ is monotone (see \cite[Lemma 15.4]{kavian}), we have
\begin{equation}\label{ss2}
\langle \Phi'(u_k)- \Phi'(u), u_k-u \rangle \geqslant 0 \hspace{0.2cm }\mbox{ for all}  \hspace{0.1cm }k \in \mathbb{N}.
\end{equation}
Therefore, by \eqref{sa} and \eqref{ss2}, we infer that
\begin{equation}\label{sb}
\lim_{k\to +\infty}\langle \Phi'(u_k)- \Phi'(u), u_k-u \rangle = 0.
\end{equation}
Consequently, by  \eqref{ss0}, \eqref{ss1}, and \eqref{sb}, we conclude
\begin{equation}\label{sc}
\lim_{k\to +\infty}\langle \Phi'(u_k), u_k-u \rangle = 0.
\end{equation}
Since $\Phi$ is   strictly convex,  we get
 \begin{equation}\label{sd}
  \Phi(u)+ \langle \Phi'(u_k), u_k-u \rangle \geqslant  \Phi(u_k)\hspace{0.2cm }\mbox{for all} \hspace{0.1cm } k \in \mathbb{N}.
\end{equation}
Thus, by \eqref{sc} and  \eqref{sd}, we have
\begin{equation}\label{se}
   \Phi(u) \geqslant\lim_{k \to +\infty}\Phi(u_k). 
\end{equation}
Since  $\Phi$ is  weakly lower semicontinuous (see $(ii)$) and by \eqref{se}, we conclude that
\begin{equation}\label{ss4}
  \Phi(u) =\lim_{k \to +\infty}\Phi(u_k).
\end{equation}
On the other hand, by \eqref{sb} the sequence $( \mathcal{U}_{k}(x,y) )_{k\in \mathbb{N}}$  converge to $0$   in $L^{1}(\mathbb{R}^{N}\times\mathbb{R}^{N})$ as $k \to +\infty$, where
\begin{equation*}
\mathcal{U}_{k}(x,y):= [\mathcal{A}(u_k(x)-u_k(y))- \mathcal{A}(u(x)-u(y))[(u_k(x)-u_k(y))-(u(x)-u(y)]K(x,y)\geqslant 0.
\end{equation*} 
Hence there exists  a subsequence $(u_{k_{j}})_{j\in \mathbb{N}}$ of  $(u_{k})_{k\in \mathbb{N}}$ such that     
\begin{equation}\label{sg}
\mathcal{U}_{k_j}(x,y)\to 0 \mbox{ a.e.  } (x,y)\in \mathbb{R}^{N}\times \mathbb{R}^{N} \mbox{ as } j \to +\infty. 
\end{equation}
We denoted $\mu_{j}(x,y)=u_{k_j}(x)-u_{k_j}(y)$ and $\mu(x,y)=u(x)-u(y)$  for all $(x,y)\in \mathbb{R}^{N}\times \mathbb{R}^{N}$.
\\
\textbf{Claim a.} If $\mathcal{U}_{k_j}(x,y)\to 0 \mbox{ a.e.  } (x,y)\in \mathbb{R}^{N}\times \mathbb{R}^{N}$, 
then  $\mu_{j}(x,y)\to\mu(x,y)$  as $j\to +\infty$ for almost all $(x,y)\in \mathbb{R}^{N}\times \mathbb{R}^{N}$. \\ Indeed,  fixed $ (x,y)\in \mathbb{R}^{N}\times\mathbb{R}^{N}$ with $x\neq y$   we suppose  by contradiction  that the  sequence $(\mu_{j}(x,y))_{j\in \mathbb{N}}$ is unbounded for $(x,y)\in \mathbb{R}^{N}\times \mathbb{R}^{N}$ fixed. Using \eqref{sg} we get $\mathcal{U}_{k_j}(x,y) \to 0$ as $j\to +\infty$ in $\mathbb{R}$, consequently there is  $M>0$ such that for all $ j \in \mathbb{N}$
\begin{equation}\label{sh}
\Big|[\mathcal{A}(\mu_{j}(x,y))- \mathcal{A}(\mu(x,y))](\mu_{j}(x,y)-\mu(x,y))K(x,y)\Big|\leqslant M. 
\end{equation}
Thus, denoting 
$$V_{\mathcal{A}}:= [\mathcal{A}(\mu_{j}(x,y))\mu_{j}(x,y)+ \mathcal{A}(\mu(x,y))\mu(x,y)]K(x,y) \mbox{for all } j \in \mathbb{R},$$
 we get from \eqref{sh} that
\begin{equation*}
\begin{split}
V_{\mathcal{A}} \leqslant & M +  \mathcal{A}(\mu(x,y))\mu_{j}(x,y)K(x,y) +\mathcal{A}(\mu_{j}(x,y))\mu(x,y)K(x,y) \mbox{for all } j \in \mathbb{R}.
\end{split}
\end{equation*}
So using  $({a}_{2})$ and  $(\mathcal{K})$ in  inequality above, we have  $\mbox{for all } j \in \mathbb{R}$ that,
\begin{equation}\label{si}
\begin{split}
c_{\mathcal{A}}b_0\frac{|\mu_{j}(x,y)|^{p(x,y)}}{|x-y|^{N+sp(x,y)}}+c_{\mathcal{A}}b_0\frac{|\mu(x,y)|^{p(x,y)}}{|x-y|^{N+sp(x,y)}}\leqslant & M + C_{\mathcal{A}}b_1\frac{|\mu(x,y)|^{p(x,y)-1}|\mu_{j}(x,y)|}{|x-y|^{N+sp(x,y)}} \\ &+ C_{\mathcal{A}}b_1\frac{|\mu_{j}(x,y)|^{p(x,y)-1}|\mu(x,y)|}{|x-y|^{N+sp(x,y)}}.
\end{split}
\end{equation}
\noindent  Dividing \eqref{si} by $|\mu_{j}(x,y)|^{p(x,y)}$, we achieve
 \begin{equation}\label{estre}
 \begin{split}
\frac{c_{\mathcal{A}}b_0}{|x-y|^{N+sp(x,y)}} +\frac{c_{\mathcal{A}}b_0|\mu(x,y)|^{p(x,y)}}{|\mu_{j}(x,y)|^{p(x,y)}|x-y|^{N+sp(x,y)}}\leqslant & \frac{M}{|\mu_{j}(x,y)|^{p(x,y)}} \\& + \frac{C_{\mathcal{A}}b_1|\mu(x,y)|^{p(x,y)-1}}{|\mu_{j}(x,y)|^{p(x,y)-1}|x-y|^{N+sp(x,y)}} \\ &+ \frac{C_{\mathcal{A}}b_1|\mu(x,y)|}{|\mu_{j}(x,y)||x-y|^{N+sp(x,y)}}
\end{split}
\end{equation}
$\mbox{for all } j \in \mathbb{R}.$
Since  we are supposing that the sequence $(\mu_{j}(x,y))_{j \in \mathbb{N}}$  is unbounded, we can assume that $|\mu_{j}(x,y)| \to + \infty$ as $j \to + \infty$, then by  \eqref{estre} we obtain $c_{\mathcal{A}}b_0 \leqslant 0$ which is an contradiction.

\noindent Therefore,  the sequence $(\mu_{j}(x,y))_{j \in\mathbb{N}}$   is bounded in $\mathbb{R}$ and up to a subsequence $(\mu_{j}(x,y))_{j \in\mathbb{N}}$     converges to some $\nu \in \mathbb{R}$. Thus we obtain $\mu_{j}(x,y)\to \nu$ $\mbox{as } j \to + \infty$. Thence denoting $$\mathcal{U}(x,y):= [\mathcal{A}(\nu)- \mathcal{A}(\mu(x,y))](\nu -(\mu(x,y))K(x,y)$$ and using  $(a_{1})$  we conclude that
\begin{equation}\label{zeroa}
\mathcal{U}_{k_j}(x,y)\to \mathcal{U}(x,y) \mbox{ as } j \to + \infty.
\end{equation}
Consequently,  by \eqref{sg} and \eqref{zeroa}, we get
\begin{equation}\label{mono}
\mathcal{U}(x,y)= [\mathcal{A}(\nu)- \mathcal{A}(\mu(x,y))](\nu -(\mu(x,y))K(x,y)= 0.
\end{equation}
In this way, by strictly convexity  of $\mathscr{A}$, \eqref{mono} this occurs only if  $\nu =\mu(x,y)= u(x)-u(y)$. Therefore, by uniqueness of limit 
\begin{equation}\label{converi}
 u_{k_{j}}(x)-u_{k_{j}}(y)=\mu_{j}(x,y) \to\mu(x,y)= u(x)-u(y) \mbox{  in } \mathbb{R} \mbox{ as } j\to +\infty
\end{equation}
for almost all $ (x, y) \in \mathbb{R}^{N}\times \mathbb{R}^{N}$.

\noindent Now we consider   the sequence $(g_{k_j})_{j \in\mathbb{N}}$ in $L^{1}(\mathbb{R}^{N}\times \mathbb{R}^{N})$ defined pointwise for all $j \in \mathbb{N}$ by
\begin{equation*}
g_{k_j}(x,y):= \bigg[\frac{1}{2}\bigg(\mathscr{A}(\mu_{j}(x,y))+ \mathscr{A}(\mu(x,y))\bigg)- \mathscr{A}\bigg(\frac{ \mu_{j}(x,y)-\mu(x,y)}{2}\bigg)\bigg]K(x,y).
\end{equation*}
By convexity the map $ \mathscr{A}$ (see $(a_{1})$),  $g_{k_j}(x,y)\geqslant 0$ for almost all $(x,y)\in \mathbb{R}^{N}\times\mathbb{R}^{N}$. Furthermore, by continuity of map  $ \mathscr{A} $ (see $(a_{1})$) and \eqref{converi}, we have $$ g_{k_j}(x,y)\to \mathscr{A} (\mu(x,y))K(x,y) \mbox{ as } j \to +\infty \hspace{0.2cm}\mbox{ for all } \hspace{0.1cm}(x,y)\in \mathbb{R}^{N}\times\mathbb{R}^{N}.$$ 
 Therefore, using this above  information,  \eqref{ss4}, and Fatou's Lemma,  we get
\begin{equation*}
\begin{split}
\Phi(u)\leqslant \liminf_{j \to +\infty} g_{k_j}(x,y)= \Phi(u)- \limsup_{j \to +\infty}\int_{\mathbb{R}^{N}\times \mathbb{R}^{N}}\mathscr{A}\bigg(\frac{ \mu_{j}(x,y)-\mu(x,y)}{2}\bigg)K(x,y)\,dx\,dy.
\end{split}
\end{equation*}
Then 
\begin{equation}\label{lm1}
\limsup_{j \to +\infty}\int_{\mathbb{R}^{N}\times \mathbb{R}^{N}}\mathscr{A}\bigg(\frac{ \mu_{j}(x,y)-\mu(x,y)}{2}\bigg)K(x,y)\,dx\,dy\leqslant 0. 
\end{equation}
On the other hand, by $(a_{2})$, $(a_{3})$,  $(\mathcal{K})$, and Proposition \ref{lw0}, we infer that
\begin{equation}\label{sn}
\begin{split}
\int_{\mathbb{R}^{N}\times \mathbb{R}^{N}}\mathscr{A}\bigg(\frac{ \mu_{j}(x,y)-\mu(x,y)}{2}\bigg)K(x,y)\,dx\,dy \geqslant &  \frac{c_{\mathcal{A}}b_0}{2^{p^{-}}p^{+}} \inf \big\{ \| u_{k_j} - u \|_{\mathscr{W}}^{p^{-}}, \| u_{k_j} - u \|_{\mathscr{W}}^{p^{+}} \big\}\\\geqslant & 0  \mbox{  for all } j \in \mathbb{N}.
\end{split}
\end{equation}
Consequently, by \eqref{lm1} and \eqref{sn}, we achieve
\begin{equation*} 
\lim_{j\to+\infty}\| u_{k_j} - u \|_{\mathscr{W}} = 0. 
\end{equation*}
Therefore,   we can conclude that $u_{k_j} \to u$ in $\mathscr{W}$ as $j \to +\infty$. Since $(u_{k_j})_{j\in \mathbb{N}}$ is an arbitrary subsequence of $(u_k)_{k \in \mathbb{N}}$, this shows that  $u_k \to u$  as $k \to +\infty$ in  $\mathscr{W}$, as required.
\hfill $\Box$

\end{document}